\documentclass[11pt,reqno,tbtags]{amsart}
\allowdisplaybreaks

\usepackage[foot]{amsaddr}
\usepackage[final]{showkeys} 

\usepackage{amsthm,amssymb,amsfonts,amsmath}
\usepackage{mathtools}
\usepackage{amscd} 
\usepackage{mathrsfs}
\usepackage[numbers, sort&compress]{natbib}
\usepackage{subfigure,graphicx}
\usepackage{vmargin, enumerate}
\usepackage{setspace}
\usepackage{paralist}
\usepackage[pagebackref=true]{hyperref}
\usepackage{color}
\usepackage[normalem]{ulem}
\usepackage{stmaryrd} 
\usepackage[refpage,intoc]{nomencl} 
\usepackage{tikz}
\usetikzlibrary{arrows}
\usetikzlibrary{decorations.markings}
\usepackage{wrapfig}
\usepackage{mathabx}
\usepackage{etoolbox}
\usepackage{caption}
\usepackage{tcolorbox}
\tcbuselibrary{theorems}
\usepackage{fancybox}
\usepackage{bbm}
\usepackage{mathrsfs}
\usepackage{textcase}
\usepackage{float}
\usepackage{type1cm}

\usepackage[T1]{fontenc}
\usepackage[utf8]{inputenc}
\usepackage[english]{babel}

\numberwithin{equation}{section}

\newcommand{\R}{\mathbb{R}}
\newcommand{\Z}{\mathbb{Z}}
\newcommand{\N}{{\mathbb N}}

\newcommand{\Q}{\mathbb{Q}}

\newcommand{\E}[1]{{\mathbb E}\left[#1\right]}

\newcommand{\p}[1]{{\mathbb P}\left(#1\right)}

\renewcommand{\P}{\mathbb{P}}

\newtcbtheorem[number within=section]{thm}{Theorem}%
{colback=green!5,colframe=green!35!black,fonttitle=\bfseries}{th}
\newtcbtheorem[number within=section]{theo}{Theorem}%
{colback=yellow!5,colframe=yellow!50!black,fonttitle=\bfseries}{theo}
\newtheorem{theorem}{Theorem}[section]

\newtheorem{lemma}[theorem]{Lemma}
\newtheorem{proposition}[theorem]{Proposition}
\newtheorem{corollary}[theorem]{Corollary}

\newtheorem{claim}[theorem]{Claim}
\newtheorem{remark}[theorem]{Remark}

\newtheorem{definition}[theorem]{Definition}

\newtheorem*{rep@theorem}{\rep@title}
\newcommand{\newreptheorem}[2]{%
  \newenvironment{rep#1}[1]{%
    \def\rep@title{#2 \ref{##1}}%
    \begin{rep@theorem}}%
    {\end{rep@theorem}}}
\makeatother

\newreptheorem{theorem}{Theorem}
\newcommand\cA{\mathcal A}
\newcommand\cB{\mathcal B}
\newcommand\cC{\mathcal C}
\newcommand\cD{\mathcal D}

\newcommand\cF{\mathcal F}

\newcommand\cM{\mathcal M}

\newcommand\cS{{\mathcal S}}
\newcommand\cT{{\mathcal T}}

\newcommand{\rF}{\mathrm{F}}

\newcommand{\bT}{\mathbf{T}}

\providecommand{\eps}{}
\renewcommand{\eps}{\varepsilon}

\providecommand{\ora}[1]{}
\renewcommand{\ora}[1]{\overrightarrow{#1}}
\newcommand{\diam}{\mathrm{diam}}

\newcommand{\eqdist}{\ensuremath{\stackrel{\mathrm{d}}{=}}}

\hypersetup{
	bookmarks=true,         
	unicode=false,          
	pdftoolbar=true,        
	pdfmenubar=true,        
	pdffitwindow=true,      
	pdftitle={},    
	pdfauthor={},     
	pdfsubject={Mathematics},   
	pdfnewwindow=true,      
	pdfkeywords={keywords}, 
	colorlinks=true,       
	linkcolor=blue,          
	citecolor=blue,        
	filecolor=blue,      
	urlcolor=blue           
}

\begingroup
\divide\count255 by 60
\multiply\count255 by -60
\advance\count255 by \time
\ifnum \count255 < 10 \xdef\oclock{\the\count1:0\the\count255}
\else\xdef\oclock{\the\count1:\the\count255}\fi
\endgroup

\DeclareRobustCommand{\SkipTocEntry}[5]{}

\newcommand{\op}{\operatorname}
\newcommand{\frk}{\mathfrak}
\newcommand{\eqd}{\overset{d}{=}}
\newcommand{\wt}{\widetilde}
\newcommand{\wh}{\widehat} 
\newcommand{\bdy}{\partial} 
\renewcommand{\P}{\mathbb{P}}

\newcommand{\notion}[1]{{\emph{#1}}}
\newcommand{\ol}[1]{{\overline{#1}}}
\newcommand{\Map}{{\mathsf{M}}}
\newcommand{\simple}{\op{S}^{\bullet}}
\newcommand{\Bol}{\op{Bol}}

\newcommand{\defeq}{:=}
\newcommand{\IV}{\mathring{V}} 
\newcommand{\IF}{\mathring{F}}
\newcommand{\IE}{\mathring{E}}
\newcommand{\IC}{\mathring{C}}
\newcommand{\BD}{\op{BD}}
\newcommand{\old}[1]{{}}
\newcommand{\Var}{\op{Var}}
\newcommand{\LP}{\mathbf{\Lambda}}
\newcommand{\LPO}{\mathbf{\Lambda}^{\mathbf{0}}}
\newcommand{\CT}{\mathbf{C}}
\newcommand{\BR}{\mathbf{b}} 
\newcommand{\BM}{\mathbf{M}} 
\newcommand{\1}{\mathbf{1}} 
\newcommand{\BB}{\mathbbm}
\def\BMS{\BB M}
\newcommand{\GHPU}{{\mathrm{GHPU}}}
\newcommand{\GHP}{{\mathrm{GHP}}}
\newcommand{\ep}{\epsilon}
\newcommand{\rta}{\rightarrow}

\newcommand{\mcon}{{\mathbbm c}}
\newcommand{\Bolm}{\Bol^{\Delta}_{\op{III}}}

\providecommand{\v}{}
\renewcommand{\v}{\mathrm{v}}

\newcommand{\pl}{\ensuremath{\preceq_{\mathrm{lex}}}}
\newcommand{\pc}{\ensuremath{\preceq_{\mathrm{ctr}}}}

\newcommand{\pcy}{\ensuremath{\preceq_{\mathrm{cyc}}}}

\newcommand{\bbrcy}[1]{\ensuremath{[#1]_{\mathrm{cyc}}}}

\newcommand{\out}{\ensuremath{\mathrm{out}}}

\newcommand{\mtrig}[1][p]{\ensuremath{\Delta^\vartriangle_{#1}}}
\newcommand{\mtrign}[2][p]{\ensuremath{\Delta^\vartriangle_{#1,#2}}}

\newcommand{\inn}{\ensuremath{\mathrm{inn}}}

\begin{document}

\title[Scaling limit of triangulations of polygons]{Scaling limit of triangulations of polygons}
\author{Marie Albenque, Nina Holden and Xin Sun}
\date{\today{}}

\begin{abstract}
We prove that random triangulations of types I, II, and III with a simple boundary under the critical Boltzmann weight 
converge in the scaling limit to the  Brownian disk. The proof uses a bijection due to Poulalhon and Schaeffer between type III triangulations of the $p$-gon and so-called blossoming forests. A variant of this bijection was also used by Addario-Berry and the first author to prove convergence  of type III triangulations to the Brownian map, but new ideas are needed to handle the simple boundary. Our result is an ingredient in the program of the second and third authors on the convergence of uniform triangulations under the Cardy embedding.
\end{abstract}
\maketitle

\section{Introduction}

The metric geometry of random planar maps is  a central topic in the study of random planar geometry.    Le Gall \cite{LeGall-uniqueness} and Miermont \cite{Miermont-BM}
independently proved that  uniformly sampled  quadrangulations of large size converge in the scaling limit as metric measure  spaces to 
a random metric measure space known as \emph{the Brownian map}.
Le Gall \cite{LeGall-uniqueness} also proved the same result  for uniform $p$-angulations with $p=3$ or $p$ even.
It was proved earlier that the Brownian map has the topology of the sphere \cite{LeGall-Sphere} and Addario-Berry and the first author extended this result to every dodd value of $p\geq 5$ \cite{AddarioAlbenqueOdd}.  The same scaling limit result was also proved for uniform random planar maps with a fixed number of edges by Bettinelli, Jacob, and  Miermont~\cite{bjm-general}. Therefore the Brownian map is a universal random surface with sphere topology.
Following Le Gall and Miermont's breakthrough, there have been extensive studies on  the universality of the Brownian map. One extension of \cite{LeGall-uniqueness,Miermont-BM} is the case when
the planar maps of interest satisfy certain connectivity properties, for example that self-loops and multiple edges are not allowed.
In this direction, it was shown in \cite{ABA-Simple,ABW-Core} that simple or 2-connected triangulations/quadrangulations converge to the Brownian map.
The convergence result was also shown for quadrangulations with no pendant vertices \cite{bl-pendant}.

Scaling limit results for random planar maps have also been extended to other topologies. Bettinelli and Miermont \cite{BeMi15} proved that uniform quadrangulations with boundary under proper rescaling converge to a random metric measure space called the Brownian disk. It is shown by Bettinelli \cite{Bet} that the Brownian disk has the  disk topology. Both these papers focus on the case where the boundary of the planar map is allowed to be non-simple. On the other hand, random planar maps with simple boundary arise naturally  in the so-called peeling process. In \cite{gwynne-miller-disk}, Gwynne and Miller adapted the result in \cite{BeMi15} to the simple boundary case by a perturbation argument.

The purpose of this paper is to further demonstrate the universality of the Brownian disk. In particular, we prove that random triangulations with a simple boundary, with or without self loops or multiple edges, converge to the Brownian disk.
Beside its independent interest, our result supplies a missing ingredient in the study of the scaling limit of percolation on uniform triangulations (see Section~\ref{subsec:cardy}).

\subsection{Main results}
A \notion{planar map} is an embedding of a finite connected (multi-)graph into the 2-dimensional sphere, considered up to orientation-preserving homeomorphisms. The \notion{faces} of the map are the connected components of the complement of edges. For any planar map $m$, we denote by $V(m)$, $E(m)$, and $F(m)$ the set of vertices, edges, and faces, respectively, of $m$.  
For any $v\in V(m)$, there is a unique cyclic (clockwise) ordering of the edges around $v$. A corner of $m$ is an ordered pair $\xi=(e,e')$, where $e$ and $e'$ are incident to a common vertex $v$ and $e'$ immediately follows $e$ in the clockwise order around $v$. We write $\v(\xi)=v$ and say that $\xi$ is incident to $v$ (and also to $e$ and $e'$). We denote by $\cC(m)$ the set of corners of $m$.
If $e=\{u,v\}$, $e'=\{v,w\}$, and $f$ is the face on the left when going along $e$ and $e'$ starting at $u$ and ending at $w$, we say that $f$ is incident to $\xi=(e,e')$, and vice versa. The degree of $f$ is the number of corners incident to it. A \notion{triangulation} is a map in which all faces have degree 3. 

For technical reasons, the maps we consider are always rooted, meaning that one of the corners is distinguished and called the \notion{root} or \notion{root corner}. The face and the vertex incident to the root corner are called the \notion{root face} and the \notion{root vertex}, respectively. The  edge incident to the root corner which is clockwise before the root corner is called the \notion{root edge}. For $M$ a rooted planar map, its root corner is denoted by $\xi(M)$ and its root vertex by $\rho(M)$.
A \notion{triangulation with boundary} of length $p$ is a map such that all its non-root faces have degree 3 and its root face has degree $p$.  A \notion{triangulation of the $p$-gon} is a triangulation with boundary of length $p$ such that the boundary of the root face is a simple curve.

When self-loops and multiple edges are allowed, the triangulations are  called  
triangulations of \notion{type I} or \notion{general triangulations}.  
Triangulations of \notion{type II} or 
\notion{loopless triangulations} 
are triangulations in which self-loops are not allowed but multiple edges are. 
Finally, triangulations of \notion{type III} or \notion{simple triangulations} are triangulations with neither self-loops nor multiple edges.\footnote{Recall that a graph is said to be $k$-connected, if at least $k$ vertices have to be removed to disconnect it. Then, triangulations of types I, II, and III are 1-connected, 2-connected, and 3-connected triangulations, respectively.} We define similarly triangulations of the $p$-gon of types I, II, and III. 

The \notion{size} of a map is the number of vertices. 
For integers $n\ge p\ge 3$ and $i\in\{ \op{I},\op{II},\op{III} \}$,
let $\Delta_i(p,n) $ be the set of type  $i$ triangulations of the $p$-gon of size $n$. Define $\Delta_i(p)=\bigcup_{n\ge p}\Delta_i(p,n)$.  
Set 
\begin{align}
  \rho_{\op{I}} = (12\sqrt 3)^{-1},\qquad \rho_{\op{II}} = 2/27, \qquad \textrm{and}\qquad \rho_{\op{III} }= 27/256. 
  \label{eq:parameter}
\end{align}
Given  $i\in\{ \op{I},\op{II},\op{III} \}$, we may assign each $m\in \Delta_i(p)$  weight  $\rho_i^n$, where $n$ is the size of $m$.
It is well known that this defines a finite measure on $\Delta_i(p)$ for all $p$. 
Let $\Bol_i(p)$ be the probability measure obtained by the normalization of this measure. We call a sample drawn from $\Bol_i(p)$ a \notion{Boltzmann triangulation} of the $p$-gon of type $i$.

We may view a triangulation of a polygon as a compact metric measure space decorated with a (boundary) curve.  A natural topology called the \notion{Gromov-Hausdorff-Prokhorov-uniform} (GHPU) topology was introduced in \cite{gwynne-miller-disk} for such objects.
Moreover, in the case of quadrangulations instead of triangulations, it was proved in \cite{gwynne-miller-disk} that  the GHPU scaling limit is the so-called \notion{Brownian disk} as defined in \cite{BeMi15}. 
The main result of this paper is the triangulation version of this theorem, for all of the three types.
\begin{theorem}\label{thm:rigor}
Let $\mcon_{\op{I}}=3/4$,  $\mcon_{\op {II}}=3/2$, and $\mcon_{\op {III}}=3$.
Fix $i\in\{ \op{I},\op{II},\op{III} \}$.  
For $p\ge 3$, let $\Map_p$ be sampled from $ \Bol_i(p)$.   
Assign length $\sqrt{3/2} p^{-1/2}$ to each edge, mass $\mcon_i p^{-2}$ to each inner vertex,  and length $p^{-1}$ to each boundary edge.
Then under this  rescaling,  $\Map_p$ converges to the free Brownian disk with perimeter 1 (see Section~\ref{subsec:BD} for definition) in the GHPU topology.
\end{theorem}

We will restate Theorem~\ref{thm:rigor} precisely at the end of Section~\ref{sec:pre}, before which we will properly introduce the space of curve-decorated compact metric measure spaces and endow it with the GHPU topology. We will also give the precise interpretation of $\Map_p$ under the scaling in Theorem~\ref{thm:rigor} as a random variable in such a space.

  We will first prove the type III case of Theorem~\ref{thm:rigor}.
  Our proof uses a bijection due to Poulalhon and Schaeffer~\cite{PoSc} between type III triangulations of the $p$-gon and ``blossoming forests'', which are forests where vertices carry some decorations. We will review this bijection in Section~\ref{sec:bijection}. 
  Having this bijection, our overall strategy is similar to the approach in \cite[Section~8]{BeMi15}. 
  Namely, we  establish the convergence of the contour and label processes in Section~\ref{sec:process}
  and conclude   the GHPU convergence from there, via the classical re-rooting argument  introduced in~\cite[Section~9]{LeGall-uniqueness}.
  However, contrary to Schaeffer's classical bijection~\cite{Sch98}, the Poulalhon-Schaeffer bijection does not encode directly metric properties of the triangulation in the blossoming forest. Borrowing some techniques developed by Addario-Berry and the first author~\cite{ABA-Simple}, we explain in Sections~\ref{sec:upper} and~\ref{sec:lower} how to define a labeling on the forest which gives some approximation of the distances. 
  
  There is also a new  difficulty in relating distances in the maps with labels in the trees coming from the presence of the macroscopic boundary. 
  We overcome this new issue by considering  a type III Boltzmann triangulation of a certain \emph{random} perimeter. 
  The type I and II cases can be reduced to the type III case by the canonical coupling of 
  $\Bol_{\op{I}} (p)$, $\Bol_{\op{II}} (p)$, and $\Bol_{\op{III}} (p)$ through the so-called \emph{core} construction. 
  This reduction was carried out in \cite{ABW-Core} in the case of quadrangulations without boundary and we adapt their technique to our setting.
  Some technical argument in the type III convergence  relies on the coupling between $\Bol_{\op{II}} (p)$ and $\Bol_{\op{III}} (p)$. Therefore we present the canonical coupling in Section \ref{sec:kernel} and prove the type III convergence in Section~\ref{sec:BD}.

\subsection{Universality}
Using our method and the quadrangulation variant of the Poulalhon-Schaeffer bijection (see \cite[Chapter 3]{FusyThesis}), it is possible to prove that  the Boltzmann quadrangulation with a simple boundary, its  2-connected core, and its simple core  jointly converge  to the same Brownian disk. This in particular  would give another proof of the main result of~\cite{gwynne-miller-disk}.
We remark that~\cite{ABA-Simple} and~\cite{ABW-Core} treated both the quadrangulation and the triangulation;
for the triangulation case of \cite{ABW-Core}, see \cite[Remark (2) after Proposition 1.5]{ABW-Core}. 

\subsection{Application to the scaling limit of percolations on random triangulations}\label{subsec:cardy}
Beyond its own interest, Theorem~\ref{thm:rigor} has important consequences for the scaling limit of percolation-decorated triangulations.  According to \cite[Theorem~8.3]{GM-Percolation}, given Theorem~\ref{thm:rigor},  the interface of the Bernoulli-$1/2$ site percolation on  $\Map_{p}$ converges to chordal SLE$_6$ on an independent free Brownian disk. Our paper makes this theorem unconditional. Based on this convergence result and \cite{bhs-site-perc},  Gwynne and the second and third authors of this paper
establish the annealed scaling limit of the full collection of interfaces for percolation-decorated Boltzmann type II triangulations \cite{ghs-metric-peano}.
The second and third authors of this paper then upgrade this to a quenched result, based on which they show that uniform triangulations converge to a so-called $\sqrt{8/3}$-Liouville quantum gravity surface under a discrete conformal embedding called the Cardy embedding \cite{hs-cardy}. 

\subsection*{Acknowledgements}
We are grateful for enlightening discussions with Louigi Addario-Berry and Yuting Wen. The research of M.A.\ is supported by the ANR under the agreement ANR-14-CE25-0014 (ANR GRAAL) and ANR-16-CE40-0009-01 (ANR GATO). 
The research of N.H.\ is supported by the Norwegian Research Council, Dr.\ Max R\"ossler, the Walter Haefner Foundation, and the ETH Z\"urich Foundation.
The research of X.S.\ is supported by the Simons Foundation as a Junior Fellow at the Simons Society of Fellows, by NSF grant DMS-1811092, and by Minerva fund at the Department of Mathematics at Columbia University.

\section{Preliminaries}\label{sec:pre}

\subsection{Basic notations}\label{subsec:notation}
We introduce some further notations which will be used frequently.

Let $\N=\{1,2,\dots \}$ be the set of positive integers. For $n\in \N$, let $[n]=\{1,\cdots,n\}$. 
Given  $a<b$, let $[a,b]_\Z=[a,b]\cap\Z$.  
Given a finite set $A$, let $|A|$ be the cardinality of  $A$. 

If $\{a_n\}_{n\in\N}$ is a sequence  such that $\lim_{n\to\infty}|a_n|=0$, we write $a_n=o_n(1)$.  For two non-negative sequences $\{a_n\}_{n\in \N}$ and $\{b_n\}_{n\in \N}$, 
if there exists a constant $C$ not depending on $n$ such that $a_n\le Cb_n$ for each $n\in\N$, then we write $a_n\lesssim b_n$.  
If $a_n/ b_n$ converges to a positive constant as $n\rta\infty$ then we write $a_n\sim b_n$. 

Given any two random variables $X$ and $Y$, we write $X\eqd Y$ if they have the same law.

Let $M$ be a triangulation of the $p$-gon. We call a vertex (resp.\ edge and corner) of $M$ a  boundary vertex (resp.\ boundary edge and boundary corner)
if it is incident to the root face, and inner vertex (resp.\ inner edge and inner corner) otherwise. We also call a non-root face an inner face.  
We let $\IV(\Map)$, $\IF(\Map)$, $\IE(\Map)$, and $\IC(\Map)$ denote the set of inner vertices, faces, edges, and corners, respectively.

If $M$ is a simple map and $\{u,v\}\in E(M)$, let $\kappa^{\op \ell}(uv)$ and $\kappa^{\op r}(uv)$ denote the corners incident to $u$ situated just before $uv$ and just after $uv$, respectively, in clockwise direction around $u$. Note that $\kappa^{\op \ell}(uv)$ is equal to $\kappa^{\op r}(uv)$ if $u$ is of degree one and that $\kappa^{\op \ell}(uv)\neq \kappa^{\op \ell}(vu)$ and $\kappa^{\op r}(uv)\neq \kappa^{\op r}(vu)$. If there is an ambiguity, we may write $\kappa_M^{\op \ell}$ and $\kappa_M^{\op r}$ to emphasize the dependence on $M$.

\subsection{Trees}\label{sub:Trees}
A \notion{plane tree} is a rooted planar map with only one face. Equivalently, it is a rooted planar map whose underlying graph has no cycle. For the remaining subsection, let $T$ be a plane tree rooted at a corner $\xi$. 

For each $w\in V(T)$, we call the graph distance from $w$ to $\v(\xi)$ the \notion{generation} of $w$ and denote it by $|w|$.
For $w\in V(T)\backslash \{\v(\xi)\}$, we define its \notion{parent} $p(w)$ to be its only neighbor such that its generation is equal to $|w|-1$, and the children of $w$ are its neighbors with generation equal to $|w|+1$ (if any). We write $k(w)$ for the number of children of $w$. For $i \in [k(w)]$, the $i$-th child of $w$ is the vertex incident to the $i$-th edge around $w$ in the clockwise direction and starting from the edge $\{w,p(w)\}$ or from $\xi$ if $w=\v(\xi)$. 
Finally, the ancestors of $w$ are the vertices lying on the unique path between $\v(\xi)$ and $w$. 

The {\em Ulam--Harris encoding} is the injective function $U=U_{T}:V(T) \to \bigcup_{i \ge 0} \mathbb{N}^i$ defined as follows, where $\N^0=\{\emptyset\}$ by convention. First, set $U(\v(\xi)) = \emptyset$. For every other vertex $w\in V(T)$, consider the unique path $\v(\xi)=v_0,v_1, \ldots,v_k=w$ from $\v(\xi)$ to $w$.
For $1 \le j \le k$ let $i_j$ be such that $v_j$ is the $i_j$-th child of $v_{j-1}$.
Then set $U(w)= i_1i_2\ldots i_k \in \N^{k}$.
In other words, the root receives label $\emptyset$ and for each $i \ge 1$ the label of any $i$-th child is obtained recursively by concatenating the integer $i$ to the label of its parent. In particular, the length of the label of a vertex is equal to its generation. The set of Ulam-Harris labels $\{U(w): w \in V(T)\}$ determines clearly (the isomorphism class of) $T$. Hence, when there is no ambiguity, we identify vertices with their Ulam-Harris label.\footnote{The Ulam--Harris labeling is only introduced to define the lexicographic order on vertices below and will not be used in subsequent sections.}

\subsection{Forests}\label{sub:Forests}
For $p$ in $\N$, a \notion{cyclic forest} (or forest) of perimeter $p$ is a planar map with 2 faces, such that the root face has a simple boundary. Equivalently, it is a sequence of trees grafted on the vertices of a simple cycle of $p$ edges with a marked corner, such that all the trees lie in the face without the marked corner, see Figure~\ref{fig:forest}$(a)$. For reasons that will appear later, we adopt in this article the unusual convention that a forest is embedded in the plane with its inner face as its unbounded face. In the following, if we define a sequence of trees $(T_1,T_2,\ldots,T_p)$ as a forest $F$, then we always consider it as a cyclic forest, with the convention that $\rho(T_1)$ is incident to $\xi(F)$ and that $T_{j+1}$ is ordered after $T_j$ in clockwise order for $j=1,\dots,p-1$. In particular, $V(F)=\bigcup_{i=1}^p V(T_i)$ and $E(F)=\Big(\bigcup_{i=1}^p E(T_i)\Big)\cup\Big(\bigcup_{i=1}^p \{\rho(T_i),\rho(T_{i+1})\}\Big)$, with $T_{p+1}:=T_{1}$. For $F$ a cyclic forest, we denote by $\rho_1,\ldots,\rho_p$ its vertices incident to the root face starting from $\rho(F)$ and in clockwise order around the root face. We also often adopt the convention that $\rho_{p+1}=\rho_1$.
\begin{figure}
  \centering
  \includegraphics[scale=0.8,page=1]{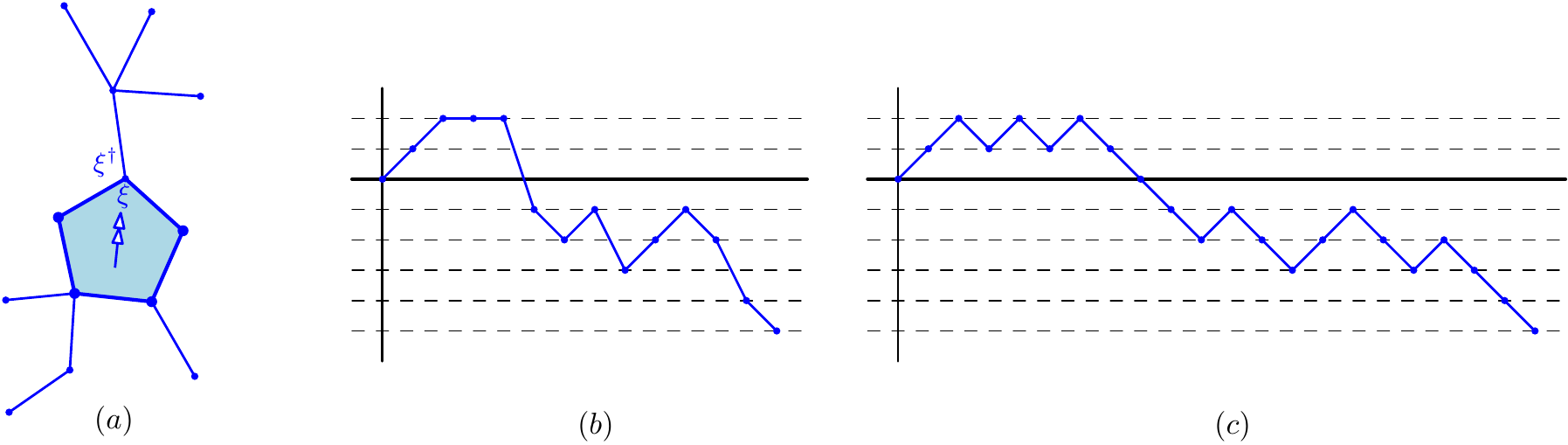}
  \caption{\label{fig:forest}$(a)$ A cyclic forest of perimeter 5, $(b)$ its height function, $(c)$ its contour function.}
\end{figure}
For $v\in V(F)$, we denote by $\tau(v)\in[p]$ the index of the tree $v$  belongs to. 
The Ulam-Harris encoding $U:V(F)\to \bigcup_{i \ge 0} \mathbb{N}^i$ of a forest $F=(T_1,T_2,\ldots,T_p)$ is defined as follows: for $v \in V(F)$, $U(v)=\tau(v) U_{T_{\tau(v)}}(v)$. 

For the remaining section, let $p\in \mathbb{N}$ and $F=(T_1,T_2,\ldots,T_p)$ be a cyclic forest with $p$ trees rooted at a corner $\xi$. Denote by $\xi^\dagger$ the next corner after $\xi$ in clockwise order around $\v(\xi)$.   

The {\em lexicographic order} $\pl$ of $V(F)$ is the total order on $V(F)$ induced by the lexicographic order on $\{U(v): v \in V(F)\}$. Similarly, the lexicographic order on $E(F)$ 
(also denoted $\pl$ by a slight abuse of notation) is defined such that $e\pl\{\rho_p,\rho_1\}$ for any $e\in E(F)$, and for any other edges $\{u,v\},\{u',v'\}$ different from $\{\rho_p,\rho_1\}$, we have $\{u,v\}\pl \{u',v'\}$ if and only if $u,v\pl u'$ or $u,v\pl v'$. The lexicographic order corresponds to the order in which a clockwise contour exploration of $F$ started at $\xi^\dagger$ encounters respectively the vertices and edges of $F$.
Let $(v_{0},v_1,\ldots,v_{|V(F)|-1)})$ be the list of vertices of $F$ sorted in the lexicographic order. We encode the forest $F$ by its \notion{height function} $H_F$, which is the $\cC([0,|V(F)|],\R)$-function defined as follows:
\[
\forall i\in[0,|V(F)|-1]_{\mathbb{Z}},\ H_{F}(i) = |v_{i}|-\tau(v_i)+1\quad \text{ and }\quad H_{F}(|V(F)|)=-p.
\]
and by linear interpolation otherwise, see Figure~\ref{fig:forest}$(b)$.

We now define the \notion{contour exploration} $\beta : [0,2|V(F)|-p]_{\Z} \to V(F)$. 
We set $\beta(0)=\v(\xi)$. 
Then for $1\leq i \leq 2|V(F)|-p$, let $\beta(i)$ be the smallest (for the lexicographic order) child of $\beta(i-1)$ that has not yet been explored, if such a vertex exists. Otherwise, if $\beta(i-1)\notin \{\rho(T_1),\ldots,\rho(T_p)\}$, let $\beta(i)$ be the parent of $\beta(i-1)$, while if $\beta(i-1)=\rho(T_k)$ for $1\le k\le p$, let $\beta(i)=\rho(T_{k+1})$ with the convention $T_{p+1}=T_1$.  A cyclic forest with $|V(F)|$ vertices has $|V(F)|$ edges. Since exactly $p$ edges are incident to both the root face and the unique inner face of $F$, the contour process explores $2|V(F)|-p$ ``sides of edges'', hence $\beta(2|V(F)|-p)=\v(\xi)$.

Recall that $\cC(F)$ is the set of corners of $F$ and that $\mathring{\cC}(F)$ denotes the set of corners of $F$ not incident to its root face. The contour exploration induces an order on $\mathring{\cC}(F)$ as follows. 
Let $\xi(0)=\xi^\dagger$,  
and for $1 \le i < 2|V(F)|-p$ let $\xi(i)=(\{\beta(i-1),\beta(i)\},\{\beta(i),\beta(i+1)\})$.  The \notion{contour order}, denoted by $\pc$, is the total order of $\mathring{\cC}(F)$ induced by $(\xi(i),0 \le i < 2|V(F)|-p)$.  
For convenience, also let $\xi(2|V(F)|-p)=\xi^\dagger$. 
Finally, write $\pcy$ for the cyclic order on $\mathring{\cC}(F)$ induced by $\pc$. It can be verified that $\pcy$ does not depend on the choice of root corner $\xi$. We define cyclic intervals accordingly: for $c,c' \in \cC(F)$, let
\begin{equation}
\bbrcy{c,c'} =  \begin{cases}
\{c'': c \pc c'' \pc c'\}& \mbox{ if } c \pc c'\, , \\
\{c'': c'' \pc c \mbox{ or } c'\pc  c''\} & \mbox{ if } c' \pc c\, .
\end{cases}
\label{eq:cyclic-interval}
\end{equation}

We also encode the forest $F$ by its \notion{contour function} $C_F$, which is the $\cC([0,2|V(F)-p|],\R)$-function defined as follows:
\[
\forall i\in[0,2|V(F)|-p-1]_{\mathbb{Z}},\ C_{F}(i) = |\beta(i)|-\tau(\beta(i))+1\quad \text{ and }\quad C_{F}(2|V(F)|-p)=-p.
\]
and by linear interpolation otherwise, see Figure~\ref{fig:forest}$(c)$.

\subsection{GHPU topology}\label{subsec:GHPU}
Given a metric space $(X,d)$, 
for two closed sets $E_1,E_2 \subset X$, their \notion{Hausdorff distance}   is given by 
\[
\BB d_d^{\op{H}}(E_1,E_2):=\max\{ \sup_{x\in E_1}\inf_{y\in E_2} d(x,y), \sup_{y\in E_2}\inf_{x\in E_1} d(x,y) \}.
\]
For two finite Borel measures $\mu_1,\mu_2$ on $X$, their \notion{Prokhorov distance} is given by   
\begin{equation*}
\begin{split}
\BB d^{\op{P}}_d (\mu_1,\mu_2) =\inf \{ \ep>0\,:\, & \mu_1(A)\le \mu_2(A^\ep)+\ep \; \\
& \textrm{and}\; \mu_2(A)\leq\mu_1(A^\ep)+\ep \; \textrm{for all closed sets }A\subset X \},
\end{split}
\end{equation*}
where $A^\ep$ is the set of elements of  $X$ at distance less than $\ep$ from $A$, i.e. $A^\ep=\{x\in X\text{ such that }\exists a\in A,\,d(a,x)<\ep\}$.

A finite (continuous) curve with length $T\in [0,\infty)$ on $X$ is a continuous map $\eta:[0,T]\to X$. Given a finite curve $\eta$ of length $T$, let $\bar\eta$ be the extension of $\eta$ to $\R$ such that $\bar \eta(t)=\eta(0)$ for $t\le 0$ and $\bar \eta(t)=\eta(T)$ for $t\ge T$.
For two finite  curves $\eta_1, \eta_2$ on $X$,  their \notion{uniform distance}  is given by   
$$\BB d^{\op U}_d(\eta_1,\eta_2)\defeq\sup_{t\in \R} d(\bar \eta_1(t) ,\bar\eta_2(t)).$$

Let $\BMS^\GHPU$ be the set of quadruples $\frk X  = (X , d , \mu , \eta)$ where $(X,d)$ is a compact metric space, 
$\mu$ is a finite Borel measure on $X$, and $\eta$ is a finite curve on $X$. 
If we are given elements $\frk X^1 = (X^1 , d^1, \mu^1 , \eta^1) $ and $\frk X^2 =  (X^2, d^2,\mu^2,\eta^2) $ of $ \BMS^\GHPU$ and isometric embeddings 
$\iota^1 : X^1 \rta W$ and $\iota^2 : X^2 \rta W$
 for some metric space $(W,D)$, we define the \notion{GHPU distortion} of $(\iota^1,\iota^2)$ by
\begin{equation}
\label{eqn-ghpu-var}
\begin{split}
\op{Dis}_{\frk X^1,\frk X^2}^\GHPU\left(W,D , \iota^1, \iota^2 \right)   
:=&  \BB d^{\op{H}}_D \left(\iota^1(X^1) , \iota^2(X^2) \right) +   
\BB d^{\op{P}}_D \left((\iota^1)_*\mu^1 ,(\iota^2)_*\mu^2) \right) \\
&+ 
\BB d_D^{\op{U}}\left( \iota^1 \circ \eta^1 , \iota^2 \circ\eta^2 \right).
\end{split}
\end{equation}
The \notion{Gromov-Hausdorff-Prokhorov-Uniform distance} between $\frk X^1$ and $\frk X^2$ is given by
\begin{align} \label{eqn-ghpu-def}
  \BB d^\GHPU\left( \frk X^1 , \frk X^2 \right) 
  = \inf_{(W, D) , \iota^1,\iota^2}  \op{Dis}_{\frk X^1,\frk X^2}^\GHPU\left(W,D , \iota^1, \iota^2 \right)      ,
\end{align}
where the infimum is over all compact metric spaces $(W,D)$ and isometric embeddings $\iota^1 : X^1 \rta W$ and $\iota^2 : X^2\rta W$.
By~\cite{gwynne-miller-uihpq}, $\BB d^\GHPU$ is a complete separable metric on~$\BMS^\GHPU$ provided we identify any two elements of~$\BMS^\GHPU$ which differ by a measure- and curve-preserving isometry.

Given two compact metric measure spaces $\frk X^1=(X^1,d^1,\mu^1)$ and $\frk X^2=(X^2,d^2,\mu^2)$,  their \notion{Gromov-Hausdorff-Prokhorov} (GHP) distance  $\BB d^\GHP\left( \frk X^1 , \frk X^2 \right) $ is defined in the same manner as in \eqref{eqn-ghpu-def} with the  curve component removed. 
We denote the space of compact metric measure spaces under the GHP metric by $\BMS^\GHP$.

Given a finite graph $G$,  identify each edge of $G$ with a copy of the unit interval $[0,1]$. We define the metric $d_G$   by requiring that  this identification is an isometric  embedding  from $[0,1]$ to $(G,d_G)$. Let  
$\mu_G$ be the counting measure on the vertex set of $G$.  A \notion{path} on  $G$ of length $m\in \Z_{>0}$ is a sequence of vertices $v_0,\ldots, v_m$ such that $v_i$ and $v_{i+1}$ are adjacent for all $0\le i<m$. A  path of length $m$ can be identified with a continuous curve of length $m$ on the metric space $(G, d_G)$, which is not necessarily unique if the map allows multiple edges between two vertices. Given a planar map $m$, the edges on the root face form a path $\bdy_m$ on $m$ which starts and ends at its root vertex, and has orientation consistent with the root edge. 
For $i\in \{\op{I},\op{II},\op{III} \}$, recall the constant $\mcon_i$ and the map $\Map_p\in \Delta_i(p)$ in Theorem~\ref{thm:rigor}.
Set
\begin{equation}\label{eq:scale}
d_p\defeq   \sqrt{3/2} p^{-1/2}d_{\Map_p}, \quad\quad   \mu_p\defeq\mcon_i p^{-2}\mu_{\Map_p},\quad\textrm{ and }\quad \bdy_p(t)\defeq\bdy_{\Map_p}(t p)\; \textrm{for }t\in[0,1].
\end{equation}
Then $\cM_p\defeq\left( \Map_p , d_p , \mu_p , \bdy_p \right)$ is a random variable in $\BMS^\GHPU$.

\subsection{Brownian disk}\label{subsec:BD}
In this subsection we review the definition of the Brownian disk following \cite{BeMi15}. See \cite{LeGall-disk} for further details.

Let $\CT$ be a standard Brownian motion and let $\bT_y\defeq  \inf\{t\ge 0: \CT_t\le -y \}$ for each $y\ge 0$. 
Define $\cA\defeq \bT_1$.  We now introduce a random process $\LPO$ coupled with $\CT$ and defined on $[0,\cA]$. 
The conditional law of $\LPO$ given $\CT$ is a centered Gaussian process (more precisely, the continuous modification of a centered Gaussian process) with covariance given by: 
\[
\text{Cov}(\LPO_s,\LPO_{s'})=\inf_{u\in[s\wedge s',s\vee s']}(\CT_u-\underline{\CT}_u) \text{ for }s,s'\in[0,\cA],
\]
where $\underline{\CT}_u\defeq\inf_{0\leq v\leq u}\CT_v$ is the past infimum of $\CT$. 

Next, let $\BR$ be a standard Brownian bridge of duration $1$ independent of $(\CT,\LPO)_{[0,\cA]}$ and  with covariance given by
\[\text{Cov}(\BR_y,\BR_{y'})=y(1-y') \text{ for } 0\le y \le y' \le 1.\]
We define the \notion{labeling process} $\LP$ by: 
\begin{equation}\label{eq:defZ}
\LP_s \defeq \LP^{\mathbf{0}}_s+\sqrt{3}\BR_{\bT^{-1}(s)} \text{ for }0\le s\le \cA, 
\end{equation}
where $\bT^{-1}(s)\defeq\sup\{y\ge 0\, :\, \bT_{y}\le s\}$, which almost surely equals  $-\inf\{ \CT_t\,:\,t\in[0,s] \} $.

For $0\le s\le s'\le \cA$, let $\underline\CT_{s, s'}=\inf\{\CT_u: u\in [s,s']\}$ and 
\begin{equation}\label{eq:metric1}
d_{\CT} (s,s')= \CT_s+\CT_{s'} -2\underline{\CT}_{s,s'}.
\end{equation} 
The function $d_{\CT}$ defines  a pseudo-metric on $[0,\cA]$, which we still denote by $d_{\CT}$.
Let $\underline\LP_{s, s'}=\inf\{\LP_u: u\in [s,s']\}$ for $0\le s\le s'\le \cA$ and  $\underline\LP_{s, s'}=\underline\LP_{s, \cA}\wedge \underline \LP_{0,s}$ for $0\le s'<s\le \cA$. Let 
\begin{equation}\label{eq:metric2}
d_{\LP} (s,s')= \LP_s+\LP_{s'} -2\max\{\underline{\LP}_{s,s'}, \underline{\LP}_{s',s} \} \qquad \textrm{for} \qquad s,s'\in [0,\cA].
\end{equation}

Let $\cD$ be  the set  of all pseudo-metrics $d$ on $[0,\cA]$ satisfying the following two properties: $\{d_{\CT}=0\}\subset\{d=0\}$ and $d\le d_{\LP}$. The set $\cD$ is nonempty (it contains the zero pseudo-metric) and contains a maximal
element $D^*$.  Let  $\BM$ be the quotient space $[0,\cA]/\{D^*=0 \}$ and let $D$ be the metric on $\BM$ induced by $D^*$. 
Let $\mu$ be the pushforward of the Lebesgue measure on $[0,\cA]$ by the quotient map  $\pi:[0,\cA]\to \BM$.   
This defines a metric measure space $\BD_1^*\defeq (\BM, D,\mu)$, which is called the \notion{free pointed Brownian disk} 
(with perimeter 1). The term ``pointed'' comes from the fact that $\BD_1^*$ has a natural distinguished point which corresponds to the image by $\pi$ of the (a.s) unique point in $[0,\cA]$ at which $\LP$ reaches its minimum, see~\cite[Lemma~11]{Bet}. We can view $\BD_1^*$ as a random variable on $\BMS^{\GHP}$.

To make $(\BM, d,\mu)$ an element of $\BMS^{\GHPU}$, let $\beta(s)\defeq \pi\circ \bT_{s}$.
Then $\beta$ can be viewed as a continuous closed simple curve on $\BM$, which is called the \notion{boundary curve} of $\BM$.
We abuse notation and let $\BD_1^*\defeq (\BM,D,\mu,\beta)$ so that $\BD_1^*$ is a random variable in $\BMS^{\GHPU}$.

If we reweight the law of $\BD_1^*$ by $\cA^{-1}=\mu(\BM)^{-1}$, then under the new measure the quadruple  $\BD_1\defeq (\BM,D,\mu,\beta)$ is  a random variable in $\BMS^{\GHPU}$ which is called a \notion{free Brownian disk} (with perimeter 1).

Now we are ready to give a precise statement of Theorem~\ref{thm:rigor}.
\begin{reptheorem}{thm:rigor}$(\mathrm{Precise}$  $\mathrm{version})$
  Fix $i\in\{ \op{I},\op{II},\op{III} \}$. For $p\ge 3$, let $\Map_p$ be sampled from $ \Bol_i(p)$ and let $\cM_p$ be defined as in Section~\ref{subsec:GHPU}.  Then  $\cM_p$ converges in law to $\BD_{1}$ in the GHPU topology.
\end{reptheorem}

\section{Bijections for simple triangulations of polygons}\label{sec:bijection}
For $n\ge p\ge 3$, let $\mtrign{n}$ be the set of type III triangulations of the $p$-gon of size $n$ with a marked triangular face (different from the root face) and let $\mtrig=\bigcup_{n\ge p}\mtrign{n}$. 
By Euler's formula, elements of $\mtrign{n}$ have $3n+p-3$ edges and $2n+p-1$ faces. The pushforward of the uniform distribution on $\mtrign{n}$ obtained by forgetting the marked face is just the uniform distribution on $\Delta_{\op{III}} (p,n)$. 
In this section we review and reformulate an encoding of the elements of $\mtrign{n}$ in terms of certain decorated forests due to Poulalhon and Schaeffer \cite{PoSc}.

\subsection{Orientations on simple triangulations}\label{sub:orientations}

Let $M$ be a rooted planar map. An \emph{orientation} of $M$ is the choice of an orientation for each of its edges. We identify an orientation $\overrightarrow{O}$ with the set of the corresponding oriented edges, i.e. for every $\{u,v\}\in E(M)$, either $\overrightarrow{uv}\in \overrightarrow{O}$ or $\overrightarrow{vu}\in \overrightarrow{O}$. For $v$ in $V(M)$, the outdegree of $v$, denoted by $\out(v)$, is the number of edges oriented away from 
$v$. 

For a triangulation $M$ of the $p$-gon a \emph{p-gonal 3-orientation} (or a \emph{3-orientation} when there is no ambiguity), is an orientation of $M$ such that for all inner vertices $v\in V(M)$, $\out(v)=3$. One can check that if this condition is satisfied then the sum of the outdegree of the vertices incident to the root face is equal to $2p-3$. For $M\in \mtrig$ endowed with an orientation, a directed cycle is called \emph{clockwise} (respectively, \emph{counterclockwise}) if the triangular marked face lies on its left (respectively, on its right).\footnote{The definition of clockwise and counterclockwise cycles coincides with the standard definition when considering the natural embedding of an element of $\mtrig$ in the plane with its marked triangular face as the unbounded face}

A combination of results from~\cite{Ossona,Felsner04,PoSc} yields the following characterization of simple triangulations of the $p$-gon. See Figure~\ref{fig:BijectionPS}. A similar characterization of simple triangulations (rather than triangulations of the $p$-gon) in terms of orientations was given previously by Schnyder~\cite{Schnyder}.
\begin{lemma}\label{lem:3orientation}
  A triangulation of the $p$-gon is simple if and only if it admits a $p$-gonal 3-orientation. 
  
  Moreover, for any $M\in \mtrig$, there exists a unique $p$-gonal 3-orientation -- called the \emph{minimal orientation} -- without counterclockwise cycles and such that the boundary 
  of the root face is a clockwise cycle, see Figure~\ref{fig:BijectionPS}(c).
\end{lemma}

\subsection{The Poulalhon-Schaeffer bijection for simple triangulations of a polygon}\label{sub:PS}
For an integer $p\geq 3$, a \emph{blossoming forest} (see Figure~\ref{fig:BijectionPS}) of perimeter $p$ is a plane forest $F$ of perimeter $p$ such that: 
\begin{itemize}
  \item Any inner vertex of degree greater than one, is incident to exactly two vertices of degree 1.
  \item There are exactly $p-3$ vertices of degree one whose unique neighbor is incident to the root face. 
\end{itemize}
We write $\cB(F)$ for the set of degree-one vertices of $F$ and call them the \emph{blossoms} of $F$. A vertex which is not a blossom is called a \emph{proper vertex}.
The \emph{size} of a blossoming forest is the number of its proper vertices. 
We denote by $\cF^*_{p,n}$ the set of rooted blossoming forests of size $n$ and perimeter $p$.  
When there is no ambiguity, we identify blossoms with their incident corners. A corner not incident to a blossom is called a \emph{proper corner}. An edge incident to a blossom is called a \emph{stem} and the other edges are called \emph{proper edges}.
\smallskip

\begin{figure}
  \begin{center}
    \includegraphics[width=0.9\linewidth,page=2]{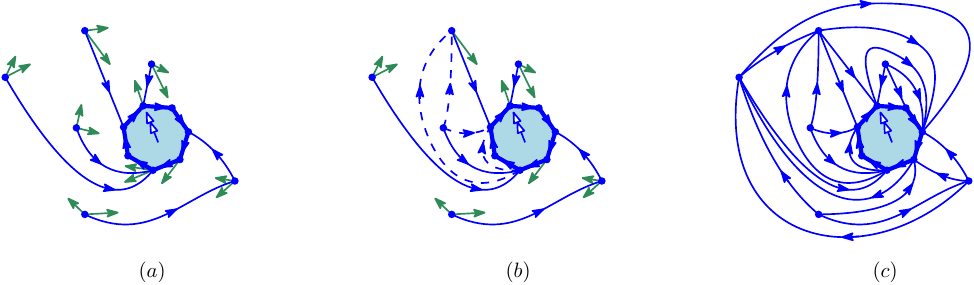}
  \end{center}
  \caption{\label{fig:BijectionPS}Illustration of the Poulalhon-Schaeffer bijection. (a) A blossoming forest of perimeter 7, where the stems are represented by green arrows, $(b)$ after some local closures, and $(c)$ the simple triangulation of the 7-gon with a marked face (endowed with its canonical minimal 3-orientation).}
\end{figure}

Let $F$ be a blossoming forest. Given a stem $\{b,u\}$ with $b\in \cB(F)$, if $\{b,u\}$ is followed by two proper edges in a clockwise contour exploration of $F$ -- $\{u,v\}$ and $\{v,w\}$, say -- then the local closure of $\{b,u\}$ consists in replacing the blossom $b$ and its stem by a new edge $\{u,w\}$ in such a way that the unbounded face lies to the left of $\{u,w\}$ when it is oriented from $u$ to $w$. The edge $\{u,w\}$ is considered to be a proper edge in subsequent local closures. The \emph{closure} of $F$, which we denote by $\chi(F)$, is defined to be the map obtained after performing all possible local closures; see Figure~\ref{fig:BijectionPS}. It is proved in~\cite{PoSc} that $\chi(F)$ does not depend on the order in which the local closures are performed. The closure is rooted in the same corner as $F$ and has a natural marked face which corresponds to the unbounded face in the planar embedding. 

The closure is naturally endowed with an orientation. We first specify an orientation of $F$:
First, orient the edges incident to the root face such that its contour forms a clockwise cycle and orient each other proper edge from a vertex towards its parent. Then, orient all the stems towards their blossom. When performing a local closure, orient the newly created edge in the clockwise direction. Observe that, the outdegree of each vertex is preserved in the closure operation. 

The following result is due to Poulalhon and Schaeffer~\cite{PoSc} (see also \cite{AP} for a proof that does not rely on an independent enumeration result).

\begin{theorem}[Poulalhon-Schaeffer \cite{PoSc}]\label{th:PS}
  For $p\geq 3$ and $n\geq p$, the closure operation $\chi$ is a bijection between $\cF^*_{p,n}$ and $\mtrign{n}$. Moreover, the orientation of the closure as described above is the minimal orientation. 
\end{theorem} 

Let $F$ be a $p$-gonal blossoming forest and let $M$ be its closure. We identify the vertices (respectively, edges) of $M$ with the proper vertices (respectively, edges) of $F$. Edges of $M$ which correspond to proper edges of $F$ are called \emph{tree-edges}, while the other (corresponding to the stems of $F$) are called \emph{closure-edges}. 

Let $\kappa \in \cC(F)$ and let $e$ be the unique directed edge of $F$ such that $\kappa^{\op \ell}(e)=\kappa$. 
We define the image of $\kappa$ in $\cC(M)$ (still denoted by $\kappa$ by a slight abuse of notation) to be $\kappa^{\op \ell}(e)$, where $e$ denotes the image of $e$ in $M$ as described in the preceding paragraph. This is clearly a bijection between $\mathring{\cC}(F)$ and $\mathring{\cC}(M)$.

\subsection{Reformulation of the Poulalhon-Schaeffer bijection with labels}\label{sub:reformulationPS}
We now present a reformulation of the Poulalhon-Schaeffer bijection, based on a labeling of the corners of the blossoming forest. This reformulation is an extension of the presentation given in \cite[Section~5.2]{ABA-Simple}.

Given a rooted blossoming forest $\rF=(F,\xi)$,
define $\lambda=\lambda_{\rF}:\mathring{\cC}(\rF) \to \Z$ as follows. 
Recall the definition of the contour order $(\xi(i),0 \le i \le 2|V(F)|-p)$ from Section~\ref{sub:Forests}, and in particular that $\xi(0)=\xi^\dagger$. 
Let $\lambda(\xi(0))=0$ and, for $0 \le i < 2|V(F)|-p$, set 
\[
\lambda(\kappa(i+1)) = 
\begin{cases}
\lambda(\kappa(i)) -1  & \text{if } \kappa(i)\not\in \cB(F), \kappa(i+1) \not\in \cB(F) , \\
\lambda(\kappa(i))  & \text{if } \kappa(i) \not\in \cB(F), \kappa(i+1)\in \cB(F) , \\
\lambda(\kappa(i)) +1  & \text{if } \kappa(i)\in \cB(F), \kappa(i+1) \not\in \cB(F) , 
\end{cases}
\]
This labeling is depicted in Figure~\ref{fig:BijectionLabel}.
\begin{figure}
  \begin{center}
    \includegraphics[width=0.9\linewidth,page=2]{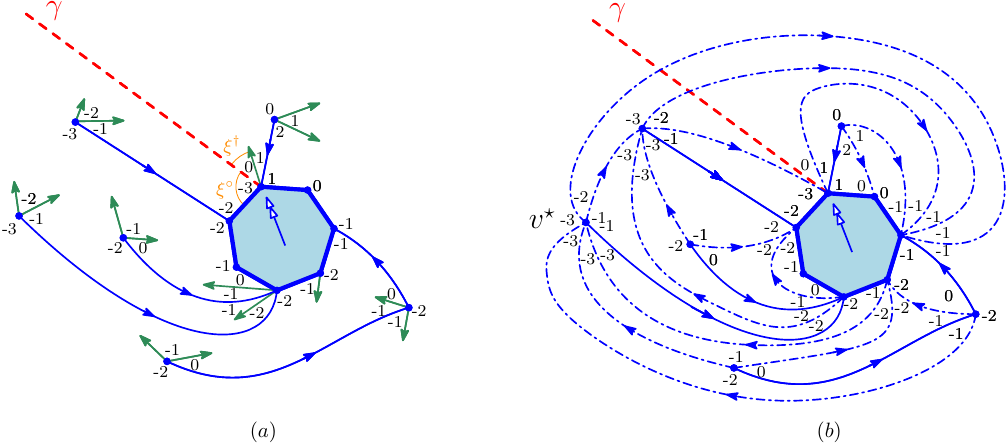}
  \end{center}
  \caption{\label{fig:BijectionLabel}Illustration of the reformulation of the Poulalhon-Schaeffer bijection with labels. For the sake of clarity, we only label inner corners of the blossoming forest.}
\end{figure} 
Informally, corners are labeled iteratively in the following manner. We walk around the forest in clockwise direction starting from the corner $\xi^\dagger$, which we label 0. Then, 
\begin{itemize}
  \item when we reach a blossom, the label remains the same;
  \item  when we leave a blossom, the label increases by one; 
  \item when we follow a proper edge, the label decreases by one. 
\end{itemize}
A simple counting argument shows that $\lambda(\kappa(2|V(F)|-p))=-3\neq\lambda(\xi^\dagger)$. To take care of this slight shift, we artificially duplicate $\xi^\dagger$ into $\xi^\circ$ and $\xi^\dagger$, we set $\lambda(\xi^\circ)=-3$, and we say that $\xi^\circ$ is the last corner for $\pc$, see~Figure~\ref{fig:BijectionLabel}(a).

For $\kappa\in \mathring{\cC}(\rF)$, the \emph{successor} $s(\kappa)\in \mathring{\cC}(\rF)$ of $\kappa$ is defined as follows:
\begin{itemize}
  \item Let $s(\kappa)$ be the first (for $\pc$) corner $\kappa'$ after $\kappa$ such that $\lambda(\kappa')<\lambda (\kappa)$, if such a corner exists (and in fact $\lambda(\kappa')=\lambda (\kappa)-1$ in this case).
  \item Otherwise, let $s(\kappa)$ be the first (for $\pc$) corner $\kappa'$ after $\xi^{\dagger}$ such that $\lambda(\kappa')<\lambda (\kappa)+3$ (and in fact $\lambda(\kappa')=\lambda(\kappa)+2$ in this case.)
\end{itemize}
The fact that the successor for any $\kappa\in \mathring{\cC}(\rF)$ is well defined is clear from the definition of $\lambda$. Note that the second case can only happen if $\kappa \in \cB(\rF)$ or if $\kappa$ precedes a blossom for $\pc$. Depending on whether the first or the second case of the definition applies, a successor is respectively called successor of the first type or of the second type.  Furthermore, the term $+3$ in the second part of the definition takes into account the shift of $-3$ in the definition of the labeling. 

For a stem $\{u,b\}$, with $b\in \cB(F)$, the \emph{local labeled closure} associated with $b$ consists in identifying $b$ with $\v(s(b))$ (recall that we identify blossoms with their unique incident corner), in such a way that the new edge splits $s(b)$ into two corners that both inherit the label of $s(b)$. The \emph{labeled closure} is the map obtained after performing all possible labeled closures. We make the following observation.

\begin{claim}
  The closure and the labeled closure obtained from a blossoming forest coincide (after removing the labels of the labeled closure).
\end{claim} 

\begin{remark}\label{rmk:embedding}
  For $F$ a blossoming forest, consider an embedding of $F$ in the plane such that its unbounded face is the non-root face. Let $\gamma$ be a semi-infinite ray emanating from $\v(\xi)$ which splits $\xi^\dagger$ (to create $\xi^\circ$) and does not intersect any other vertex or edge. 
  
  Consider an embedding of $\chi(F)$ built on this embedding of $F$ (by that we mean that the embedding of the proper edges of $F$ coincide with their image by $\chi$), such that each closure-edge intersects $\gamma$ at most once. Then, closure-edges that intersect $\gamma$ exactly correspond  to blossoms with successors of the second type. Moreover, if $uv$ is a closure-edge crossing $\gamma$ oriented from $u$ to $v$ then the unbounded part of $\gamma$ lies on the left of $uv$.
\end{remark}

\subsection{Properties of labels in blossoming forests and triangulations of a polygon}\label{sub:labelsBlossoming}
For the next  two subsections, let $F$ be a fixed blossoming forest and let $M=\chi(F)$ be its closure endowed with its minimal 3-orientation $\overrightarrow{O}$. 

Each corner of $M$ corresponds either to a proper corner of $F$ or was obtained by splitting a proper corner of $F$ during a closure. The label of a corner in $M$ is defined as the label of the corresponding corner in the blossoming forest, see Figure~\ref{fig:BijectionLabel}. Recall the bijection between $\mathring{\cC}(F)$ and $\mathring{\cC}(M)$ given at the end of Section~\ref{sub:PS}. Then, for any $\kappa \in \mathring{\cC}(F)$ not incident to a blossom, $\lambda_F(\kappa)=\lambda_M(\kappa)$, 
where we identify $\kappa$ with the associated corner in $\mathring{\cC}(M)$.

Thanks to the label on the corners of $F$, we can determine (without performing the closure) which corners belong to the marked triangular face of $\chi(F)$.

\begin{claim}
  Let $\ell_{\min}$ be the minimal label in $F$ and let $\kappa_{\min}$ be the first corner (for $\preceq_{\mathrm{ctr}}$) in $F$ such that $\lambda(\kappa_{\min})=\ell_{\min}$. Similarly, let $\kappa_{\min+1}$ and $\kappa_{\min+2}$ be the first corners of respective label $\ell_{\min}+1$ and $\ell_{\min}+2$. Then the images of $\kappa_{\min}$, $\kappa_{\min+1}$, and $\kappa_{\min+2}$ in $M$ are the three corners of its  marked triangular face. 
  
  The vertex incident to $\kappa_{\min}$ is denoted by $v^\star$. 
  \label{claim:vstar}
\end{claim}
\begin{proof}
  First observe that $\kappa_{\min}$ cannot be incident to a blossom. So, either there exists a closure-edge $\{v^\star,u\}$, with $v^\star u \in \overrightarrow{O}$ and $\kappa_{\min} = \kappa_M^{\op\ell}(v^\star u)$, or there exists a tree-edge $\{v^\star,u\}$, such that $\kappa_{\min} = \kappa_M^{\op\ell}(v^\star u)$ (the latter case can only happen if $\kappa_{\min} = \xi^\circ$).
  
  Now, let $\kappa \in \mathcal{C}(F)$. By definition of the successor and thanks to the fact that the increments of labels in $F$ belong to $\{-1,0,1\}$, then $s(\kappa)\in \bbrcy{\kappa,\kappa_{\min}}$. Hence the closure of a stem cannot separate $\kappa_{\min}$ from the unbounded face, which proves the claim for $\kappa_{\min}$. A similar reasoning applies for $\kappa_{\min+1}$ and $\kappa_{\min+2}.$
\end{proof}
\smallskip

It will be useful in the following to also give a label to vertices and edges of $F$ and $M$. For $v$ a proper vertex of $F$ (or equivalently a vertex of $M$), the label of $v$ is the minimum of the label of corners incident to $v$, where, in the special case $v=\rho_1$, we do not consider the corner $\xi^\circ$ with label $-3$ so that $X(\rho_1)=0$. We denote the label by $X_F(v)$, $X_M(v)$, or even $X(v)$ if there is no ambiguity. An easy consequence of the labeling function is the following.
\begin{claim}\label{claim:labelVertex}
  Let $v$ in $\mathring{V}(F)\backslash \mathcal{B}(F)$ be a proper inner vertex of $F$, let $u$, $b_1$ and $b_2$ be the parent of $v$ and its two blossoms, respectively, such that $u$, $b_1$ and $b_2$ appear in this order when turning clockwise around $v$. Then the corners of $F$ incident to $v$ are labeled:
  \begin{itemize}
    \item $X(v)$ if they lie between $u$ and $b_1$,
    \item $X(v)+1$ if they lie between $b_1$ and $b_2$,
    \item $X(v)+2$ if they lie between $b_2$ and $u$.
  \end{itemize}
\end{claim}

For $\{u,v\}\in E(M)$, recall that $\kappa^{\op \ell}(uv)$ and $\kappa^{\op r}(uv)$ are the corners incident to $u$ situated respectively just before $uv$ and just after $uv$ in clockwise direction around $u$.
Set $\lambda^{\op \ell}(uv)\defeq \lambda(\kappa^{\op \ell}(uv))$ and $\lambda^{\op r}(uv)\defeq\lambda(\kappa^{\op r}(uv))$.
The evolution of labels along an edge satisfies the following.
\begin{claim}\label{claim:labelEdges}
  Let $\{u,v\}$ be an edge of $M$ such that $uv\in \overrightarrow{O}$. Then: 
  \begin{enumerate}[(i)]
    \item if $\{u,v\}$ is a tree-edge, there exists $i\in \mathbb{Z}$ such that: \[(\lambda^{\op \ell
    }(uv),\lambda^{\op r}(vu),\lambda^{\op \ell
  }(vu),\lambda^{\op r}(uv))=(i+1,i,i,i-1),\]
  \item if $\{u,v\}$ is a closure-edge not crossing $\gamma$, there exists $i\in \mathbb{Z}$ such that: \[(\lambda^{\op \ell
  }(uv),\lambda^{\op r}(vu),\lambda^{\op \ell
}(vu),\lambda^{\op r}(uv))=(i,i-1,i-1,i+1),\]
\item if $\{u,v\}$ is a closure-edge crossing $\gamma$, there exists $i\in \mathbb{Z}$ such that: \[(\lambda^{\op \ell
}(uv),\lambda^{\op r}(vu),\lambda^{\op \ell
}(vu),\lambda^{\op r}(uv))=(i,i+2,i+2,i+1).\]
\end{enumerate}
\end{claim}
\begin{proof}
  In $F$, two corners on either side of an incoming edge have the same label by Claim \ref{claim:labelVertex}. The same property holds for $M$, since when we do a closure operation a corner is split into two corners, and these two corners get the same label. This justifies why the second and third coordinates in the quadruples are the same. The remaining elements of the tuple are immediate from the definition of $\lambda$.
\end{proof}

As a direct consequence of Claims~\ref{claim:labelVertex} and~\ref{claim:labelEdges}, we obtain the following.
\begin{claim}\label{claim:labelsNeighbours}
  Let $u$ and $v$ be two inner vertices of $M$ such that $\{u,v\}\in E(M)$. Then: 
  \begin{itemize}
    \item $|X(u)-X(v)|\leq 3$ if $\{u,v\}$ is a closure-edge, and
    \item $|X(u)-X(v)|\leq 1$ otherwise.
  \end{itemize}
\end{claim}

\subsection{Validly labeled forests}
For $p\geq 3$ fixed, a labeled forest $(F,X)$ of perimeter $p$ is a rooted cyclic forest $F$ together with a labeling $X$ of its vertices, where $X:V(F)\rightarrow \mathbb{Z}$. For $e=\{u,p(u)\}\in E(F)$, the \emph{displacement} at $e$, denoted by $D_e$, is equal to $X(u)-X(p(u))$. Next, for $u$ a vertex of $F$, recall that $k(u)$ denotes the number of children of $u$ and that $(u(1),u(2),\ldots,u(k(u)))$ is the list of its children in lexicographic order. Then the \emph{displacement vector} of $F$ at $u$, denoted by $D_F(u)$ (or $D(u)$ if there is no ambiguity), is the vector $\big(D_{\{u,u(i)\} },1\leq i \leq k(u)\big)$.

Let $(F,X)$ be a labeled forest, and let $\rho_1,\rho_2,\ldots,\rho_p$ denote its vertices incident to the root face starting from the root vertex and in clockwise order. Then $(F,X)$ is \emph{validly labeled} if: 
\begin{itemize}
  \item $X(\rho_i)\geq X(\rho_{i-1})-1$, for $1\leq i \leq p$,
  \item $X(\rho_1)\geq X(\rho_p)+2$,
  \item the sequence $\big(D_{\{u,u(i)\}}, 1\leq i \leq k(u)\big)$ is non-decreasing for any $u \in V(F)$,
  \item $D_{e}\in \{-1,0,1\}$, for $e\in E(F)$ such that both extremities of $e$ are inner vertices, 
  \item $D_e \in \{-1,0,1,2,\ldots\}$ for $e\in E(F)$ with at least one extremity incident to the root face.
\end{itemize}
The set of validly labeled forests of perimeter $p$ is denoted by $\cF^{\op{vl}}_p$, and the set of validly labeled forests of perimeter $p$ and with $n$ vertices is denoted by
$\cF^{\op{vl}}_{p,n}$.
\smallskip

Let $F$ be a blossoming forest endowed with the labeling $X$ of its vertices as defined in Section~\ref{sub:labelsBlossoming}. The following claims are clear from the definition of the labeling. 
\begin{claim}\label{claim:stems}
  For any $e=\{v,p(v)\}\in E(F)$, 
  \[D_e+1=|\{ e'\,:\, e'\preceq_{\op{lex}}e,\,e'\text{ is a stem incident to }p(v)\}|.\] 
  For any edge $\rho_i\rho_{i+1}$, 
  \[X(\rho_{i+1})-X(\rho_i)+1+2\delta_{i=p}=|\{\text{stems incident to }\rho_i\}|.\]
\end{claim}
This claim directly implies that erasing the stems of a blossoming forest gives a validly labeled forest. In fact, this operation is a bijection between blossoming forests and validly labeled forests, and its inverse is described as follows, see also Figure~\ref{fig:LabeledForest}.
\begin{figure}
  \centering
  \includegraphics[scale=0.8]{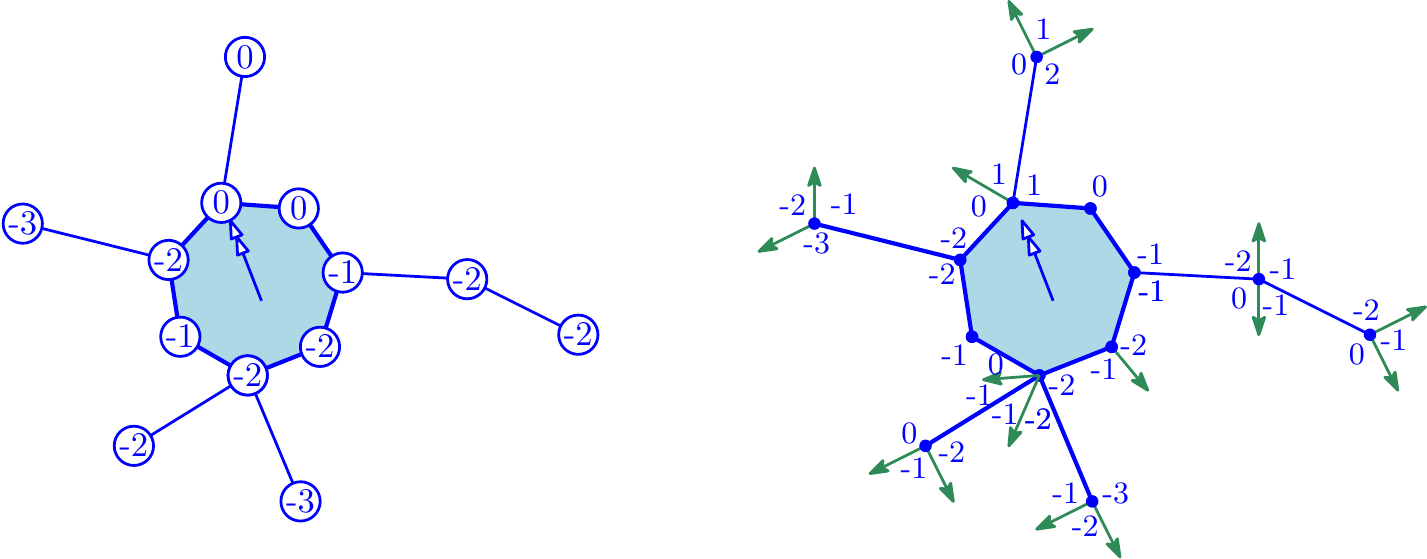}
  \caption{\label{fig:LabeledForest}Transformation of a validly labeled forest $(F,X)$ into the corresponding blossoming forest $\Phi(F,X)$.}
\end{figure}

Let $(F,X)\in \cF^{\op{vl}}_{p}$ be a validly labeled forest and let $\Phi(F,X)$ be the blossoming forest obtained in the following way. For every inner vertex $u$, we add two stems $\{u,b_1\}$ and $\{u,b_2\}$, such that $\{u,b_1\}\preceq_{\op{lex}}\{u,b_2\}$ and if $e$ is an edge between $u$ and one of its children
\begin{itemize}
  \item $e\preceq_{\op{lex}}\{u,b_1\}$, if $D_e = -1$, 
  \item $\{u,b_1\}\preceq_{\op{lex}}e\preceq_{\op{lex}}\{u,b_2\}$, if $D_e = 0$, 
  \item $\{u,b_2\}\preceq_{\op{lex}} e$, if $D_e=1$.
\end{itemize}
For $1\leq i \leq p$, we define $s_i:=X(\rho_{i+1})-X(\rho_i)+1+2\delta_{i=p}$. Next, we add $s_i$ stems $e^i_1,\ldots,e^i_{s_i}$ at $\rho_i$, such that $e^i_1\preceq_{\op{lex}}\ldots\preceq_{\op{lex}} e^i_{s_i}$ and if $e$ is an edge between $\rho_i$ and one of its children: 
\begin{itemize}
  \item $e\preceq_{\op{lex}}e^i_{1}$, if $D_e = -1$, 
  \item $e^i_{j}\preceq_{\op{lex}}e\preceq_{\op{lex}}e^i_{j+1}$, if $D_e = j-1$, for $j\in\{1,\ldots,s_{i}-1\}$,
  \item $e^i_{s_i}\preceq_{\op{lex}} e$, if $D_e=s_{i}-1$.
\end{itemize} 
Thanks to Fact~\ref{claim:stems}, the following is immediate.
\begin{claim}\label{claim:label}
  The map $\Phi$ is a bijection between $\cF^{\op{vl}}_{p,n}$
  and $\cF^{*}_{p,n}$.
\end{claim}

\subsection{Sample the free marked triangulation}\label{subsec:sample}
Fix $p\ge 3$ an integer. Let $\Map_p $ be sampled from $\Bol_{\op{III}}(p)$, \emph{reweighted} by the total number of inner faces, i.e., the measure from which $\Map_p $ is sampled has Radon-Nikodym derivative relative to $\Bol_{\op{III}}(p)$ which is proportional to the number of inner faces. 
Conditionally on $\Map_p$, uniformly sample an inner face $f^*$ and let $\Map^*_p\defeq(\Map_p,f^*)$. 
Let $\Bolm(p)$ be the law of $\Map^*_p$, which is a probability measure on $\mtrig$.
In this subsection, we describe a way of sampling a blossoming forest $F^*_p$ such that $\chi(F^*_p)$ has the law of $\Bolm(p)$, where $\chi$ is the closure operation defined in Theorem~\ref{th:PS}. The construction is illustrated on Figure~\ref{fig:sampling}.
\begin{figure}
  \centering
  \includegraphics[scale=0.9,page=1]{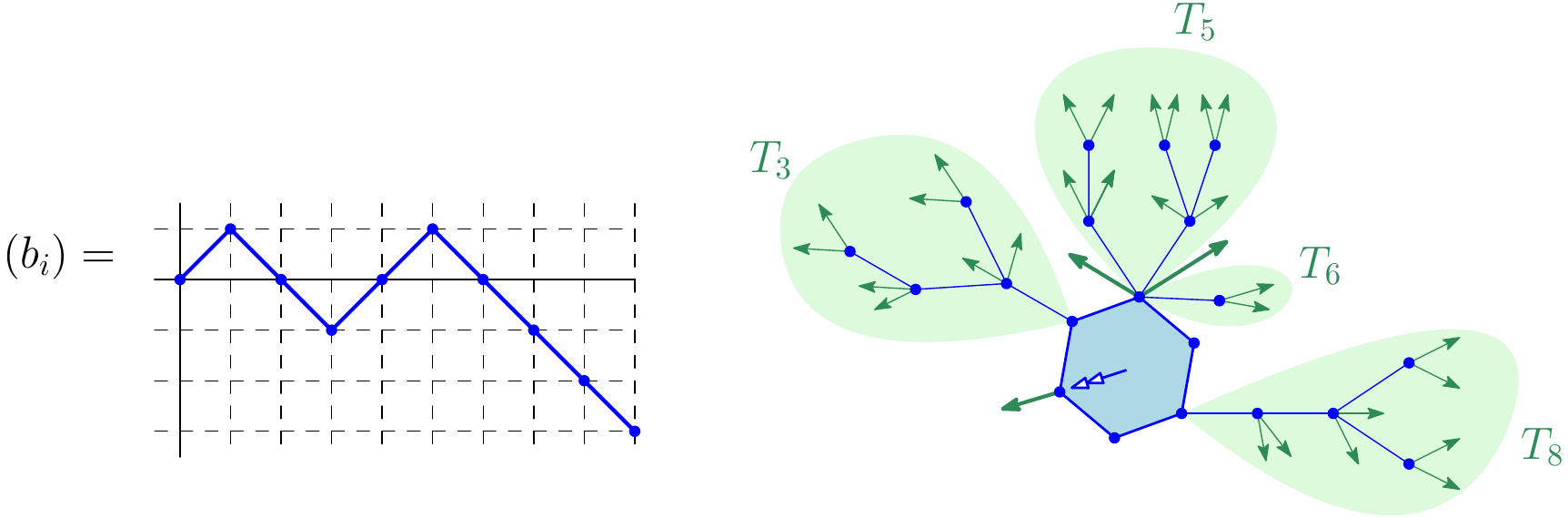}
  \caption{\label{fig:sampling}Illustration of the sampling procedure in the case $p=6$. The trees $T_1,T_2,T_4,T_7$, and $T_9$ are reduced to single vertices. On this example $(Z_j)=(1,0,2,0,0,0)$.}
\end{figure}

The first step is to sample a sequence of random blossoming trees. 
Let $G$ be a geometric random variable on $\{0,1,\cdots \}$ with parameter $3/4$. Namely,  $\P[G=k]= 3/4^{k+1}$ for $k=0,1,2\cdots$. 
Let $B$ be a random variable with probability distribution given by 
\begin{equation}\label{eq:B}
\p{B=k}=\frac{\binom{k+2}{2}\p{G=k}}{\BB E\big[\binom{G+2}{2}\big]}\qquad \textrm{for all }k=0,1,2,\cdots.
\end{equation}
Now sample a Galton-Watson tree $T^{\circ}$
such that its offspring distribution is given by $G$ for  the root vertex, and $B$ for all other vertices. Conditioning on $T^\circ$, 
for each non-root vertex $v\in V(T^\circ)$, add two stems incident to $v$, uniformly at random among the $\binom{k(v)+2}{2}$ possibilities. 
This gives rise to a random blossoming tree which we denote by $T$. Let $\{T_i\}_{i\in\N}$ be a sequence of independent copies of $T$.

Consider $(X_1,\ldots,X_{2p-3})$ to be uniformly distributed on the set of $(2p-3)$-tuples $(x_1,\ldots,x_{2p-3})\in \mathbb{Z}^{2p-3}$ such that:
\[
\begin{cases}
x_i\in\{-1,1\}\text{ for all }1\leq i \leq 2p-3,\\
\sum_{i=1}^{2p-3}x_i=-3.
\end{cases}
\]
Let $b_0=0$ and $b_k=\sum_{1}^{k} X_i$ for $1\le k\le 2p-3$. Then $(b_k)_{0\le k\le 2p-3}$ has the same distribution as the steps of a random walk with step distribution uniform on $\{+1,-1\}$,  starting from 0 and conditioned to visit $-3$ at time $2p-3$. 
For such a tuple $(X_1,\ldots,X_{2p-3})$, define $(t_j)_{0\leq j\leq p}$ inductively as follows: $t_0=0$ and
\[
t_{j}=\inf\{t>t_{j-1}\,:\,X_{t_j}=-1\}, \quad \text{for }j>0.
\]
Next, set $Z_j:= t_j-t_{j-1} -1$. 
Informally, $Z_j$ is the number of up-steps between the $(j-1)$-th and the $j$-th down-steps.

Finally, take a rooted $p$-gon and let $\rho_1,\ldots,\rho_p$ denote its boundary vertices, ordered as above starting from the root vertex. 
For $i=1,\ldots,p$, add $Z_i$ stems at $\rho_i$ in the non-root face. Since $\sum_{i=0}^{
  2p-3 } X_i=-3$, we have added $p-3$ stems in total, so that there are $2p-3$ (proper) corners incident to the non-root face. 
We call these corners \emph{boundary corners. }
Now  for $1\leq i \leq 2p-3$, graft the blossoming tree $T_i$ from the first step  in the $i$-th corner in clockwise order around the root face and starting from $\xi^{\dagger}$. 
The resulting blossoming forest is a random variable  in  $\cF_{p}^*$ which we denote by $F^*_p$.

\begin{remark}\label{rmk:bridge}
  For $1\le i \le 2p-3$, let $k_i$ be the  number of children of the root vertex of $T_i$. 
  Then attaching the tree $T_i$ to the $i$-th boundary corner divides this corner into $k_i+1$ corners of $F^*_p$.
  Recall the corner labeling rule for $F^*_p$. Since two adjacent corners are bordered by an incoming arrow,  
  these $k_i+1$ corners share the same label, hence the $i$-th boundary corner  has a well-defined label. 
  It is easy to check that this label is exactly equal to $b_i$.
\end{remark}

\begin{proposition}\label{prop:Bol}
  Suppose $F^*_p$ is sampled as described above.
  Let $\Map^*_p=(\Map_p,f)  =\chi(F^*_p)$ where $\chi$ is the closure operation in Theorem~\ref{th:PS}. 
  Then the law of  $\Map^*_p$ is  $\Bolm(p)$.
\end{proposition}
\begin{proof}
  Let $\alpha\defeq3/4$, $\beta\defeq1-\alpha=1/4$, and $\gamma\defeq\BB E\big[\binom{G+2}{2}\big]= 16/9$.  
  For each blossoming tree $t$ such that $\P[T=t]>0$, let $|t|$ be the number of proper edges. Then 
  \[
  \P[T=t]=\alpha \cdot  \left(\alpha/\gamma\right)^{|t|} \cdot  \beta^{|t|}=\alpha (\alpha\beta/\gamma)^{|t|}.
  \]
  Here the factor $\binom{k+2}{2}$ for each non-root vertex cancels with the $\binom{k+2}{2}$ possibilities of attaching the two stems. 
  
  Let $c_p$ be the probability that a random walk started from 0  hits $-3$ at time $2p-3$, where the random walk has step distribution uniform on $\{+1,-1\}$. 
  
  For  $M^*_0=(M_0,f_0^*)\in \mtrig$ with $\chi^{-1}(M_0^*)=(t_i)_{1\le i\le 2p-3}$, our sampling procedure gives 
  \begin{align*}
    \P[\Map^*_p=M^*_0] 
    = &c_p^{-1}  (1/2)^{2p-3}  \prod_{i=1}^{2p-3}\P[T_i=t_i]\\ 
    =&  c_p^{-1}  (1/2)^{2p-3} \alpha^{2p-3} (\alpha\beta/\gamma)^{\sum_{i=1}^{2p-3} |t_i| }\\
    =& ( 27 c_p/512)^{-1} (9/64)^{p} \cdot (27/256)^{|\IV(M_0)|},
  \end{align*}
  where we use $\sum_{i=1}^{2p-3} |t_i|=|\IV(M_0)|$ in the last equality.
  Recalling $\rho_{\op{III}}=27/256$ in \eqref{eq:parameter}, we conclude the proof by the definition of $\Bolm(p)$.
\end{proof}

\section{Convergence of the height  and label processes}\label{sec:process}
Let $F^*_p$  be as in Proposition~\ref{prop:Bol}.
Recall that  $|F^*_p|$ is the number of proper (i.e., 
non-blossom) vertices of $F^*_p$. 
In other words, if $F_p$ is the forest obtained by removing all the blossoms from $F_p^*$, then $|F^*_p|$ is the number of vertices of $F_p$.
Let $H_p$ be the height function of $F_p$ as defined in Section~\ref{sub:Forests}. 
Recall from Claim~\ref{claim:label} that  $F^*_p$  can be viewed as a validly labeled forest with labels on the vertices of $F_p$.
For $i\in[0,|F^*_p|]_\Z$, let $\lambda_p(i)$ be  the label of the $i$-th vertex of $F_p$ in the lexicographic order,
where we identify the $|F^*_p|$-th and the $0$-th vertex.
Then $(H_p,\lambda_p)$ is a function from $[0,|F^*_p|]_\Z$ to $\Z^2$. 
We extend it to be a continuous function on $[0,|F^*_p|]$ via linear interpolation.
For $s>0$ such that $0\le p^2s/3 \le |F^*_p|$, define
\[
H_{(p)} (s)= p^{-1} H_p(p^2s/3)\quad \textrm{and}\quad \lambda_{(p)} (s)=\Big(\frac{3}{2p}\Big)^{1/2} \lambda_p(p^2s/3).
\]
Recall the notations in Section~\ref{subsec:BD}. 
The main result of this section is the following proposition.
\begin{proposition}\label{prop:contour-label}
  $(H_{(p)},\lambda_{(p)})$ converges in law to  $(\CT(t),  \LP(t))_{t\in [0,\cA]}$ in the uniform topology.
\end{proposition} 
Throughout this section, whenever we say that a real-valued process converges we mean in the uniform topology.

\subsection{Convergence of a modified height process}\label{subsec:height}
Recall the notations in Section~\ref{subsec:sample}.
For $1\le i\le 2p-3$, let  $\rho_i$ be the root of $T_i$.
For $1\le i\le 2p-3$,  let $G_i$ be the number of non-leaf children of $\rho_i$. Let $Y_0=0$ and  $Y_j=\sum_{i=1}^{j} G_i$ for $1\le j\le 2p-3$.
The next lemma says that each ${Y_j}$ is very close to $j/3$.
\begin{lemma}\label{lem:LLN}
  For each $\eps,\delta\in(0,1)$ there exist positive constants $C$ and $c$ such that 
  \[
  \P[\max \{j^{-1}|Y_j-j/3|: \delta p\le j\le 2p-3  \} \ge \eps]\le Ce^{-cp} \qquad \textrm{for all }p\ge 3.
  \]
\end{lemma}
\begin{proof}
  By the sampling rule in Section~\ref{subsec:sample}, $(G_i)_{1\le i\le 2p-3}$ is a sequence of independent copies of $G$ described in Section~\ref{subsec:sample}.
  Since $\E{G}=1/3$, Lemma~\ref{lem:LLN} follows from  concentration for the sum of independent geometric random variables (see e.g.\ \cite{exponential}).
\end{proof}

Consider the sequence $(\ol v_j)_{1\le j\le Y_{2p-3}}$, 
where $(\ol v_j)_{Y_{i-1}< j\le Y_{i}}$ are the non-leaf children of  $\rho_i$ in clockwise order for $1\le i\le2p-3$. 
Let $T'_j$ be the subtree rooted at $\ol v_j$ with all stems removed. 
Conditioning on $(G_i)_{1\le i\le 2p-3}$, 
the law of $(T'_j)_{1\le j\le Y_{2p-3}}$ is a sequence of independent samples of the Galton-Watson tree with offspring distribution $B$ as in Section~\ref{subsec:sample}.
By an elementary calculation  (see \cite[Appendix]{ABA-Simple}), 
\begin{equation}\label{eq:EB}
\E{B}=1\qquad \textrm{and}\qquad \Var[B]=4/3.
\end{equation}
Moreover, the distribution of $B$ has an exponential tail.

Let $F'_p$ be the forest $(T'_j)_{1\le j\le Y_{2p-3}}$ and let $H'_p$ be the height function of  $F'_p$. Let $|F'_p|$ be the number of vertices in $F'_p$. Then $|F'_p|= |F^*_p|-p$ because the $p$ vertices on the boundary of the root face are not part of $F'_p$. 
For $1\le i\le |F'_p|$, let $v_i'$ be the $i$-th vertex on $F'_p$ in the lexicographical order and let $\xi_i$ be the number of children of $v'_i$. 
To obtain the scaling limit of the height function, we consider the so-called \emph{depth first queue process} $S'_p$ of $F'_p$. 
Let $S'_p(0)=0$ and
\[
S'_p(j)= \sum_{i=1}^{j}(\xi_i-1) \qquad \textrm{for all }1\le j\le |F'_p|.
\]
Conditioning on $(G_i)_{1\le i\le 2p-3}$,   
$S'_p$ has the law of the trajectory of a simple random walk  with step distribution given by  the law of $B-1$, starting at 0 and terminating at the first time hitting $-Y_{2p-3}$. 
It is clear that  $F'_p$ is determined by $S'_p$. Given a plane tree, we may define its associated depth first queue process in the same way.
Let $\tau_0=0$. For $1\le j\le Y_{2p-3}$, let $\tau_j=\inf\{i:  S'_p(i)=-j \}$. 
Then $\{S'_p(i+\tau_{j-1})-S(\tau_{j-1})\}_{0\le i\le \tau_j-\tau_{j-1}}$ is  the depth first queue process  of the tree $T'_j$.

\begin{lemma}\label{lem:CT}
  Let $S'_{(p)} (s)=\frac32 p^{-1}S'_p(p^2s/3)$ and  $H'_{(p)} (s)=p^{-1}H'_p(p^2s/3)$. 
  Let $\CT$, $\cA$, and $\bT^{-1}$ be as in Section~\ref{subsec:BD}.
  Then $S'_{(p)}$ converges in law to  $(\CT(t))_{0\le t\le \cA}$ as $p\to\infty$. 
  If we are under a coupling of $\{ S'_{(p)} \}_{p\ge 3}$ where this convergence holds in probability, then
  $H'_{(p)}$  converges  to  $(\CT(t)-\frac13 \bT^{-1}(t))_{0\le t\le \cA}$, and  $3|F^*_p|/p^2$ converges to $\cA$, in probability.
\end{lemma}
\begin{proof}
  By Lemma~\ref{lem:LLN} and the invariance principle for the Brownian motion, $S'_{(p)}$ converge in law to a linear Brownian motion $\CT$ with variance $\frac94\times \Var[B]\times \frac13=1$,
  until the first time $\cA$ when it hits $-\frac32\times \frac23=-1$. 
  By the now classical relation between height functions and depth first queue processes
  (see e.g.\ \cite[Theorem~3]{Height-process}),
  \begin{equation}\label{eq:concentration}
  \max_{1\le j\le Y_{2p-3}}\max_{\tau_{j-1}\le i\le\tau_j} \left\{p^{-1}\left|\left(S'_p(i)-S'_p(\tau_{j-1})\right)  -2^{-1} \Var[B] \left(   H'_p(i) -  H'_p(\tau_{j-1})  \right) \right| \right\}
  \end{equation}
  converges to 0 in probability.  Therefore, as $p\to\infty$, $S'_{(p)}-H'_{(p)}$ is determined by the variation when jumping from one tree to the next. 
  By Lemma~\ref{lem:LLN},  under a coupling where $S'_{(p)}$  converge to  $(\CT(t))_{0\le t\le \cA}$  in probability,
  $S'_{(p)}-H'_{(p)}$ converge to $\frac13 \bT^{-1}$. Under the same coupling, it is clear that  $\lim_{p\to\infty}3|F'_p|/p^2=\cA$ in probability, which yields the same convergence for $3|F^*_p|/p^2$.
\end{proof}

\subsection{Proof of Proposition~\ref{prop:contour-label}}\label{subsec:label}
We start the proof with two easy results. Let $(Z_i)_{1\leq i \leq p}$ denote the sequence which gives the number of stems incident to each boundary vertex in the construction of $F^*_p$ given in Section~\ref{subsec:sample}. It follows directly from the definition that $(Z_i)$ has the distribution of a sequence of $p$ independent geometric random variable on $\{0,1,\ldots\}$ with parameter $1/2$, conditioned to have sum equal to $p-3$. Therefore, for all $\epsilon >0$ there exists $K>0$ such that
\begin{equation}\label{eq:maxZ}
\p{\max_{1\leq i \leq p}Z_i\geq K\ln p}\leq \varepsilon \quad \text{for any }p\geq 3.
\end{equation}

Next, let $b_p$ denote the bridge $(b_k)_{0\leq k\leq 2p-3}$ associated with $F^*_p$ as introduced in Section \ref{subsec:sample}.
\begin{lemma}\label{lem:BR}
  For $0\le s\le 1$, let $b_{(p)} (s) = (\frac{3}{2p})^{1/2} b_p((2p-3)s)$.  Then $b_{(p)}$ converges in law to $\sqrt{3}\BR$ as defined in Section~\ref{subsec:BD}.
\end{lemma}
\begin{proof}
  This is the classical scaling limit result for the convergence of simple  random walk conditioning on the endpoint. 
  The scaling constant $(\frac32)^{1/2}$ ensures that the limiting Brownian bridge has the right variance.
\end{proof}

\begin{lemma}\label{lem:coupling}
  There is a coupling of $(F^*_p)_{p\ge 3}$ and $(\CT,\BR,\LP)$ (defined in Section~\ref{subsec:BD}) such that the convergences in Lemmas~\ref{lem:CT} and~\ref{lem:BR} hold almost surely, and moreover,  for each fixed $s>0$  $\lim_{p\to\infty}\lambda_{(p)}(s) \1_{\cA>s}=\LP(s)\1_{\cA>s}$ in probability. 
\end{lemma}
\begin{proof}
  Since $b_p$ and $\{T_i\}_{i\in \N}$ are independent,  it is immediate from the Skorokhod embedding theorem that there is a coupling of $(F^*_p)_{p\ge 3}$ and $(\CT,\BR,\LP)$ such that the convergence results of Lemmas~\ref{lem:CT} and~\ref{lem:BR} are almost sure. It remains to upgrade the coupling so that the last assertion of the lemma about pointwise convergence of $\lambda_{(p)}$. 
  Let $\sigma(s)=\sup\{t\le s: \CT(t) \leq-\bT^{-1}(s)\}$ and $\tau(s)\defeq \inf\{ t\ge s:  \CT(t)<-\bT^{-1}(s) \}$. Then $\sigma(s)<s<\tau(s)$ almost surely.  
  Let \(j'_p(s)=-\min\{S'_p(k):k\in[0,sp^2/3]\}+1\). Let
  \begin{align*}
  \begin{split}
    \sigma_p(s)&\defeq \inf\left\{j\le sp^2/3: S'_p(j) =-j'_p(s)+1  \right\},\\
    \tau_p(s)&\defeq \inf\left\{j\ge sp^2/3: S'_p(j)= -j'_p(s) \right\}.
  \end{split}
  \end{align*}
  On the event $\cA>s$ define $\sigma_{(p)} (s) \defeq 3\sigma_p(s)/p^2$ and $\tau_{(p)} (s) \defeq 3\tau_p(s)/p^2$. 
  By the convergence of $S'_{(p)}$ to $\CT$,
  we have $\lim_{p\to\infty}(\sigma_{(p)}(s),\tau_{(p)}(s))=(\sigma(s),\tau(s))$ and $\lim_{p\to\infty} \frac32p^{-1} j'_p(s)=\bT^{-1}(s)$ 
  in probability.

  For  $0\le i\le \tau_p(s)-\sigma_p(s)$, let $S^s_p(i)\defeq S'_p(i+\sigma_p(s))-S'_p(\sigma_p(s) )$ and $H^s_p(i)\defeq H'_p(i+\sigma_p(s))-H'_p(\sigma_p(s) )$.
  Then $ (H^s_p(i))_{0\le i\le \tau_p(s)-\sigma_p(s)} $ is the height function of the tree $T'_{j'_p(s)}$ in $F'_p$. 
  Recall the definition of boundary corners in the construction of $F^*_p$ in Section~\ref{subsec:sample}.
  Let $j_p(s)$ be the index of the boundary corner  at which we attach the tree of $F_p$ that contains  $T'_{j'_p}(s)$ as a subtree.
  By Lemma~\ref{lem:LLN}, $\lim_{p\to\infty} (2p)^{-1} j_p(s)=\bT^{-1}(s)$ in probability. Recall Remark~\ref{rmk:bridge}. By Claims \ref{claim:labelVertex} and \ref{claim:labelEdges}, the label of the root vertex of $T'_{j'_p(s)}$ is equal to $b_p(j_p(s))-1$. 
  Therefore  Lemma~\ref{lem:BR} gives that the label of the root vertex of $T'_{j'_p(s)}$ 
  rescaled by $(\frac{3}{2p})^{1/2}$ converge to $\sqrt{3}\BR_{\bT^{-1}(s)}$ in probability.
  
  For  $0\le i\le \tau_p(s)-\sigma_p(s)$, let $\lambda^s_p(i)$ be the label of the $i$-th vertex on $T'_{j'_p(s)}$  in  the lexicographic order. 
  Let  $H^s_{(p)} (t)\defeq p^{-1} H^s_p(tp^2/3)$ and $\lambda^s_{(p)} (t)\defeq (\frac{3}{2p})^{1/2}\lambda^s_p(tp^2/3)$.
  Then by \cite{ABA-Simple},  $(H^s_{(p)}, \lambda^s_{(p)})$  jointly converge in law to $(\CT(\cdot-\sigma(s)),\LPO(\cdot-\sigma(s))_{0\le t\le \tau(s)-\sigma(s)}$.
  By our sampling procedure in Section~\ref{subsec:sample}, we may recouple so that the convergence holds in probability. 
  Combined with the previous paragraph, we see that $\lim_{p\to\infty}\lambda_{(p)}(s)=\LP(s)$ in probability.
  
  For a fixed $s'\neq s$, if $s'\in (\sigma(s),\tau(s))$, then we still have $\lim_{p\to\infty}\lambda_{(p)}(s')=\LP(s')$  in probability.   
  Otherwise, $[\sigma(s'),\tau(s')]\cap [\sigma(s),\tau(s)]=\emptyset$. 
  We can repeat the argument for $[\sigma(s'),\tau(s')]$ to re-couple such that $\lim_{p\to\infty}\lambda_{(p)}(s')=\LP(s')$ in probability. 
  Since $\CT-\underline\CT$  has countably many excursions on $[0,\cA]$, we conclude the proof.
\end{proof} 
\begin{proof}[Proof of Proposition~\ref{prop:contour-label}]
  Suppose we are under a coupling satisfying the condition in Lemma~\ref{lem:coupling}.
  Comparing $H_{(p)}$ and $H'_{(p)}$, it is straightforward from the concentration of i.i.d.\ geometric variables that 
  $H_{(p)}$ converges in probability to $(\CT(t)+c\bT^{-1}(t))_{0\le t\le \cA}$ for some constant $c$. 
  By matching the final condition $\CT(\cA)+c\bT^{-1}(\cA)=-p^{-1}\times p =-1$, we must have $c=0$.
  It remains  to show that in our coupling $\lim_{p\to\infty}\lambda_{(p)}=\LP$ in probability. By Lemma~\ref{lem:coupling},
  it suffices to show that  $\{\lambda_{(p)} \}_{p\ge 3}$ is  tight in the uniform topology. 
  
  Let $\omega_p(\delta)\defeq p^{-1/2}
  \sup\{|\lambda_p (j)- \lambda_p(i) | : j-i\le \delta p^2,    1\le i\le j\le  |F^*_p|\}$. 
  To show the tightness of $\{\lambda_{(p)} \}_{p\ge 3}$ we only need to  show that 
  \begin{equation}\label{eq:tight}
  \textrm{$\lim_{\delta\to 0}\limsup_{p\to \infty} \omega_p(\delta)=0$ in probability.}
  \end{equation}
  
  To prove \eqref{eq:tight}, we will prove a tightness result for a related process. Recall the definition of the contour exploration $\beta$ defined in Section~\ref{sub:Forests}. Then, for $1\leq i \leq 2|F^*_p|-p$, let $\lambda_p^{\!\mathrm{ctr}}(i)$ be the label of the vertex $\beta(i)$. In other words, the process $\lambda_p^{\!\mathrm{ctr}}$ gives the label of vertices of $F_p$ encountered in a contour exploration of $F_p$. We also set $\omega^{\mathrm{ctr}}_p(\delta)\defeq p^{-1/2}
  \sup\{|\lambda^{\!\mathrm{ctr}}_p (j)- \lambda^{\!\mathrm{ctr}}_p(i) | : j-i\le \delta p^2,    1\le i\le j\le  2|F^*_p|-p\}$. Standard results -- see for instance \cite[Proposition~13]{AddarioAlbenqueOdd}) --  ensure that \eqref{eq:tight} holds if and only if the similar result holds for $\omega^{\mathrm{ctr}}_p(\delta)$. We hence now focus on proving that
  \begin{equation}\label{eq:tightctr}
  \textrm{$\lim_{\delta\to 0}\limsup_{p\to \infty} \omega^{\mathrm{ctr}}_p(\delta)=0$ in probability.}
  \end{equation}
The idea for proving~\eqref{eq:tightctr} is to embed the forest $F^*_p$ into a tree $\wh T$ such that both $\wh T$  and the discrepancy between $F^*_p$ and $\wh T$ are easy to control. See Figure~\ref{fig:process} for an illustration.
  
  Given $p\ge 3$, let $\{S_i\}_{i\ge 0}$ be a simple random walk with step distribution given by the law of $B-1$.
  Moreover, suppose $Y_{2p-3}$ is independent of $S$. We will work on the overwhelmingly high probability event that $Y_{2p-3}>0$ to avoid some trivial degenerations in the discussion below.
  For each $i\ge 0$, let $\underline S(i)=\inf\{S(j): j\le i\}$.
  Let $\sigma=\inf\{ i\ge 0: S(i)-\underline S(i) \ge Y_{2p-3} -1\}$.
  Let $\tau=\inf\{i>\sigma : S(i)-S(\sigma) =-Y_{2p-3}\}$.
  Then $\sigma$ and $\tau$ are stopping times for $S$. (We hide the dependence of $\sigma,\tau$ on $p$ for simplicity.)
  By the strong Markov property of $S$, we see that $\{S(i+\sigma)-S(\sigma) \}_{0\le i\le \tau-\sigma}$ has the same distribution as $S_p'$.
  Hereafter we assume that $S'_p$ equals $\{S(i+\sigma)-S(\sigma)\}_{0\le i\le \tau-\sigma}$.
  
  Let $\wh \sigma=\sup\{i\le \sigma: S(i)=\underline S(i)\}$ and $\wh \tau=\inf\{i>\wh \sigma: S(i)=S(\wh \sigma)-1 \}$. Then $\wh \sigma\le \sigma\le \tau = \wh \tau$. Moreover,  $\{S(i+\wh\sigma) -S(\wh\sigma) \}_{0\le i\le \wh\tau-\wh\sigma}$ can be viewed as the depth first queue process  of a tree, which we denote by $\wh T$. 
  Since $S_p'$ equals $\{S(i+\sigma)-S(\sigma)\}_{0\le i\le \tau-\sigma}$, 
  we see that $\{T'_i\}_{1\le i\le Y_{2p-3}}$ are subtrees of $\wh T$.  
  Now we  attach two stems to each vertex of $\wh T$ in a uniform way to obtain a blossoming tree $\wh T^*$.
  For each $1\le i\le 2p-3$, since $T_i'$ is a subtree of $\wh T$,
  we maybe assume that $\wh T^*$ and $F^*_p$ are coupled such that 
  for a non-root vertex $v$ of $T_i'$, the way to attach in $\hat T^*$ the two stems around the corners of $v$ is the same as in the blossoming forest $F^*_p$, see Figure~\ref{fig:process}.  
  \begin{figure}
    \includegraphics[scale=0.8]{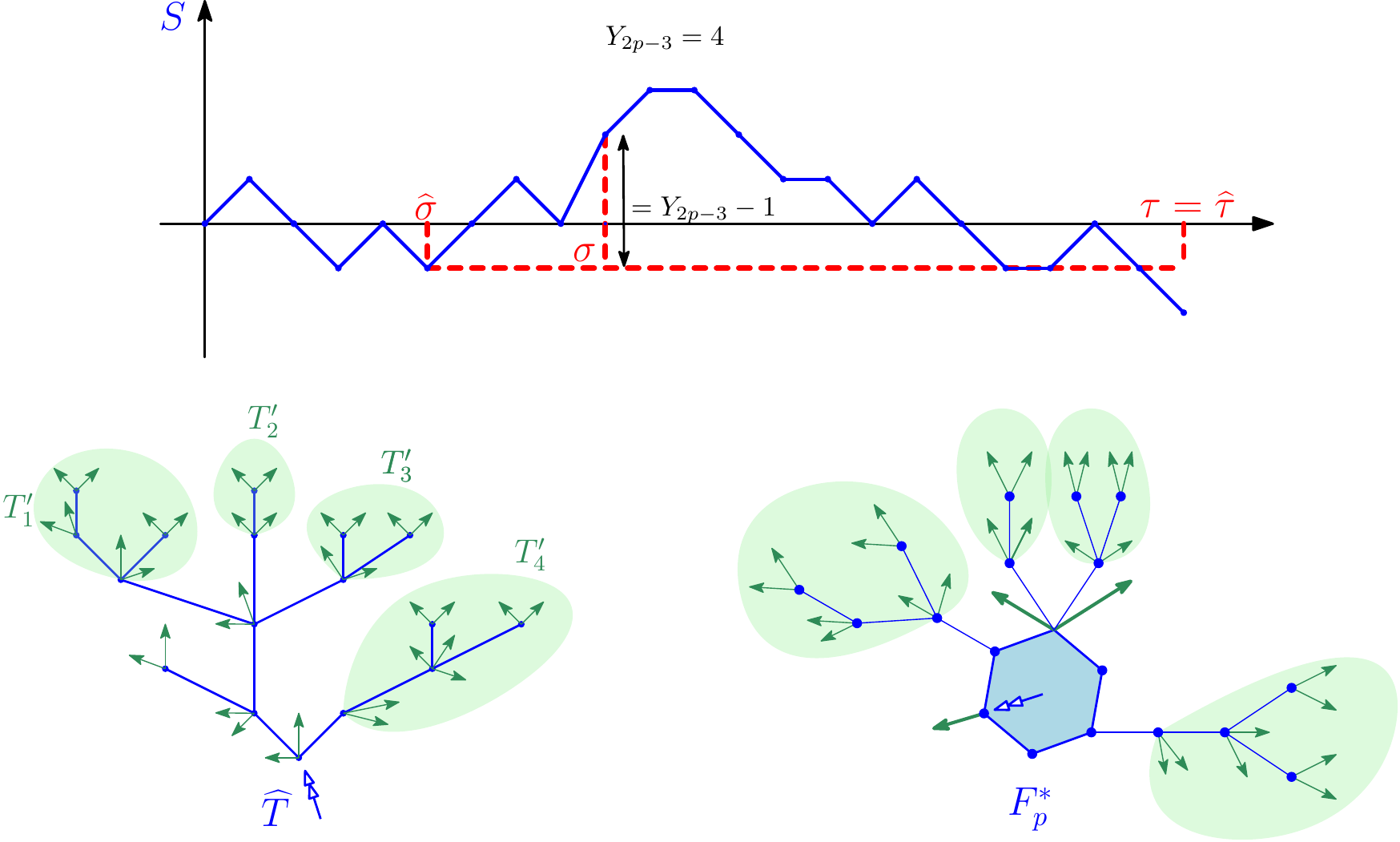}
    \caption{\label{fig:process}An example of the construction of $\wh T$ from $S$ and of the coupling between $\wh T$ and $F_p^*$.}
  \end{figure}
  Let $\wh \lambda^{\!\mathrm{ctr}}_{p}$ be the label process of $\wh T^*$, where the vertices are explored in the order of a contour exploration. 
  Let $\wh\lambda^{\!\mathrm{ctr}}_{(p)}$ be obtained from rescaling $\wh\lambda_p$ as in the definition of $\lambda^{\!\mathrm{ctr}}_{(p)}$.
  Then by \cite[Proposition~6.1]{ABA-Simple}, $\wh \lambda^{\!\mathrm{ctr}}_{(p)}$ converges in law in the uniform topology. 
  Given $\delta>0$, let $\wh\omega^{\mathrm{ctr}}_p(\delta)\defeq p^{-1/2}\sup\{|\wh \lambda_p (j)-\wh \lambda_p(i) | : j-i\le \delta p^2,  1\le i\le j \le \wh\tau-\wh\sigma \}$. Then $\lim_{\delta\to 0}\limsup_{p\to \infty} \wh\omega^{\mathrm{ctr}}_p(\delta)=0$ in probability for each $\delta>0$.
  
  We are now ready to prove \eqref{eq:tight}.   Recall that $|F^*_p|$ is the number of vertices of $F_p$.
  For  $1\le i \le  2|F^*_p|-p$,  define $i_1$ and $i_2$ as follows: 
  \begin{align}
    i_1 &= \max\{j\leq i, \text{ such that }\beta(i)\text{ is a boundary vertex}\}, \\
    i_2 &= \min\{j\geq i, \text{ such that }\beta(j)\text{ is a boundary vertex}\}.
  \end{align}
  Let $1\le i\le j\le  2|F^*_p|-p$ and assume that $j-i\le \delta p^2$. If $i\le i_2\le j_1\le j$,
  then by the triangle inequality and the coupling between $F^*_p$ and $\wh T^*$ above, we have
  \begin{equation}\label{eq:boundlambda}
  \begin{split}
  |\lambda^{\!\mathrm{ctr}}_p (j)- \lambda^{\!\mathrm{ctr}}_p(i) |
  &\le |\lambda^{\!\mathrm{ctr}}_p(j)-\lambda^{\!\mathrm{ctr}}_p(j_1)|+ |\lambda^{\!\mathrm{ctr}}_p (j_1)- \lambda^{\!\mathrm{ctr}}_p(i_2) |
  +|\lambda^{\!\mathrm{ctr}}_p(i_2)-\lambda^{\!\mathrm{ctr}}_p(i)|\\
  &\le 2p^{1/2}\wh \omega^{\mathrm{ctr}}_p(\delta) +2\Big(\max_{1\leq i \leq p}Z_i+1\Big)+ |\lambda^{\!\mathrm{ctr}}_p (j_1)- \lambda^{\!\mathrm{ctr}}_p(i_2) |,
  \end{split}
  \end{equation}
  where we recall that $Z_i$ denotes the number of stems incident to $\rho_i$ in $F^*_p$.  
  If $i\le i_2\le j_1\le j$ does not hold, then  $\beta(i)$ and $\beta(j)$ are in the same tree, hence  $ |\lambda^{\!\mathrm{ctr}}_p (j)- \lambda^{\!\mathrm{ctr}}_p(i) |\le p^{1/2}\wh \omega^{\mathrm{ctr}}_p(\delta)$.
  
  Let $\omega'_p(\delta)\defeq p^{-1/2}
  \sup\{|\lambda_p (j_1)- \lambda_p(i_2) | : j-i\le \delta p^2,    1\le i\le j\le  2|F^*_p|-p\}$. 
  In view of \eqref{eq:maxZ}, to conclude the proof of \eqref{eq:tight} it suffices to show that  
  \begin{equation}\label{eq:tight2}
  \textrm{$\lim_{\delta\to 0}\limsup_{p\to \infty} \omega'_p(\delta)=0$ in probability for each $\delta>0$.}
  \end{equation} 
  Recall the definition of the contour function defined in Section~\ref{sub:Forests}. Denote $C_p$ the contour function of $F_p$ and define $\mathfrak {a}_p(\delta)= \sup\{|C_p(j)- C_p(i) | : j-i\le \delta p^2,    1\le i\le j\le  2|F^*_p|-p\}$.  Again standard results (see \cite[Section~1.6]{LeGallSurvey}) allow to transfer the convergence of $H_{(p)}$ to a convergence result for the scaled contour function, from which we get that  $\lim_{\delta\to 0}\limsup_{p\to \infty}p^{-1}\mathfrak{a}_p(\delta)=0$. 
  
  Since $|C_p(j_2)- C_p(i_1)|$ measures the distance along the boundary between $\beta(j_2)$ and $\beta(i_1)$, \eqref{eq:tight2} follows from Remark~\ref{rmk:bridge} and Lemma~\ref{lem:BR}.
\end{proof}

\subsection{Uniform integrability of the total mass}\label{subsec:UI}
For $p\ge 3$, let $\mathbb E^*_p$ and $\P^*_p$ be the probability and expectation, respectively, corresponding to $\Bolm(p)$, as in Proposition~\ref{prop:contour-label}. 
Let $\mathbb E_p,\P_p$ the ones corresponding to $\Bol_{\op{III}}(p)$. 
Let $\mathbb E, \P$ be the ones corresponding to  $(\CT(t),  \LP(t)_{t\in [0,\cA]})$.
As a byproduct of our proof of Proposition~\ref{prop:contour-label}, we  obtain the following lemma, which will be useful when comparing 
$\P^*_p$ and $\P_p$.
\begin{lemma}\label{lem:RN}
  The $\P^*_p$-law of the Radon-Nikodym derivative 
  between $\P_p$ and $\P^*_p$ is uniformly integrable.  The same statement holds with $\P^*_p$ and $\P_p$ swapped.
  Moreover,
  \begin{equation}\label{eq:limexp}
  \lim_{p\to\infty}3p^{-2}\mathbb E_p[|V(\Map_p)|]=\E{\cA^{-1}}^{-1}.
  \end{equation}
\end{lemma}
\begin{proof}
  Recall $|F'_p|=|F^*_p|-p$ defined in Section~\ref{subsec:height}.
  Since the $\P^*_p$-law of $|F'_p|$ is the hitting time of the simple random walk $S'_p$, as a simple random walk exercise we have
  \(\lim_{p\to\infty}\mathbb E^*_p[(3p^{-2}|F'_p|)^{-1}]=\E{\cA^{-1}}.\) 
  Therefore
  \begin{equation}\label{eq:RN0}
  \lim_{p\to\infty}\mathbb E^*_p[(3p^{-2}|V(\Map_p)|)^{-1}]=\E{\cA^{-1}}.
  \end{equation}
  Let $N_p$ be the number of inner faces of $\Map_p$ so that $N^{-1}_p/\mathbb E^*_p[N^{-1}_p]$ is the Radon-Nikodym derivative of $\P_p$ with respect to $\P^*_p$. 
  By Euler's formula, we obtain from \eqref{eq:RN0}  that
  \begin{equation}\label{eq:RN1}
  \lim_{p\to\infty}\mathbb E^*_p[(1.5p^{-2}N_p)^{-1}]=\E{\cA^{-1}}.
  \end{equation}
  Moreover, we  obtain from the last statement of Lemma~\ref{lem:CT} that 
  the $\P^*_p$-law of $(1.5p^{-2}N_p)^{-1}$ weakly converges to the $\P$-law of $\cA^{-1}$.
  Therefore the $\P^*_p$-law of $N^{-1}_p \mathbb E^*_p[N^{-1}_p]^{-1}$  weakly converges to the $\P$-law of $\cA^{-1}\E{\cA^{-1}}^{-1}$. 
  We conclude that the $\P^*_p$-law of $N^{-1}_p \mathbb E^*_p[N^{-1}_p]^{-1}$  is uniformly integrable since this family of random variables is positive, converge in law, and has a converging expectation.
  
  Since~\eqref{eq:RN1} is equivalent to
  \[
  \lim_{p\to\infty}\mathbb E_p[(1.5p^{-2} N_p)^{-1}\cdot N_p\mathbb E_p [N_p]^{-1}]=\E{\cA^{-1}},
  \]
  we have $\lim_{p\to\infty}\mathbb E_p[1.5p^{-2} N_p]=\E{\cA^{-1}}^{-1}$.
  Using Euler's formula again we get~\eqref{eq:limexp}.
  
  It remains to show that the $\P_p$-law of $N_p\mathbb E_p[N_p]^{-1}$ is uniformly integrable.  
  Since the $\P^*_p$-law of $N_p\mathbb E_p[N_p]^{-1}$ weakly converges to $\cA\E{\cA^{-1}}$ and 
  the $\P^*_p$-law of $N^{-1}_p \mathbb E^*_p[N^{-1}_p]^{-1}$ is uniformly integrable, the $\P_p$-law of $N_p\mathbb E_p[N_p]^{-1}$ weakly converges to
  the $\P'$-law of $\cA\E{\cA^{-1}}$, where $d\P'=\cA^{-1}\E{\cA^{-1}}^{-1} d\P$.  
  Since $\cA\E{\cA^{-1}}$  has expectation 1 under $\P'$, the $\P_p$-law of  $N_p\mathbb E_p[N_p]^{-1}$ is uniformly integrable. 
\end{proof}

\section{Leftmost paths: definition, properties, and link with distances} \label{sec:upper}
This section is an adaptation of the work of Addario-Berry and the first author~\cite{ABA-Simple} to \emph{deterministically} relate distances and labels in $M$. 
Throughout this section, $F$ is a $p$-gonal blossoming forest endowed with its orientation and its labeling. We denote by $M \in \mtrig$ its closure, which is also endowed with its minimal 3-orientation $\overrightarrow{O}$ as defined in Lemma~\ref{lem:3orientation} and with its canonical labeling (see Section~\ref{sub:labelsBlossoming}). 
As usual, we identify the vertices and edges of $M$ with the proper vertices and edges of $F$. 

\subsection{Modified leftmost paths and distance to $\v^\star$}
We introduce in this section a variant of leftmost paths, the so-called \emph{modified leftmost paths}, and present some of their properties. Modified leftmost paths were first defined in~\cite{ABA-Simple}.
\begin{figure}[t]
  \centering
  \includegraphics[page=2,scale=0.8]{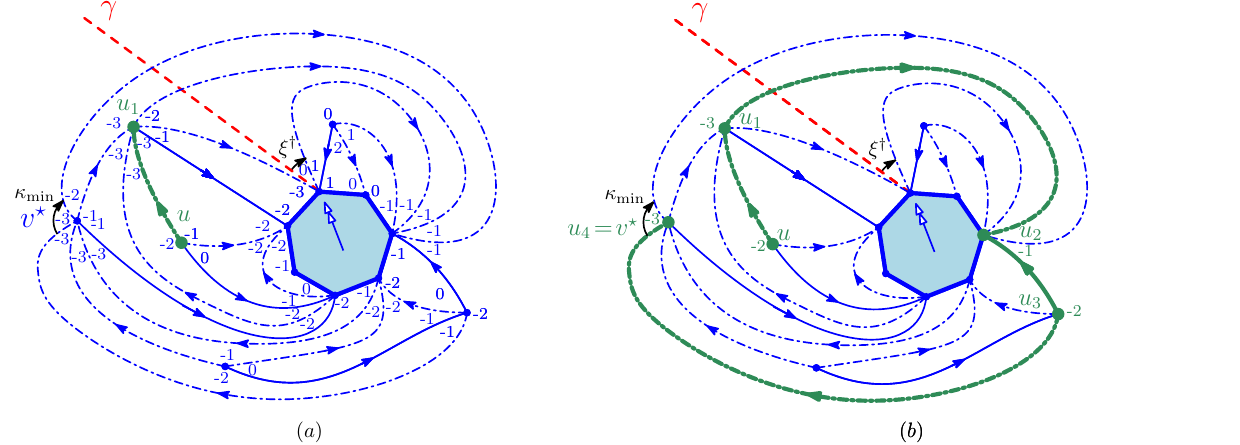}
  \caption{\label{fig:LMPMaps}$(a)$ A triangulation of the $7$-gon endowed with its canonical orientation and with a marked oriented edge $uu_1$. Tree-edges are plain and closure edges are dashed. $(b)$ The modified leftmost path started at $uu_1$ is represented in bold green edges.}
\end{figure}

\begin{proposition}\label{prop:LMPreaches}
  Let $e=v_0v_1$ be an oriented edge in $\overrightarrow{O}$. Consider the sequence of vertices $v_0,v_1,v_2,\ldots$ such that for each $i\in\N$, $v_iv_{i+1}$ is the first edge incident to $v_i$, among edges after $v_{i-1}v_{i}$ in clockwise direction around $v_i$ such that either (see Figure~\ref{fig:LMPMaps}):
  \begin{itemize}
    \item $v_{i}v_{i+1}$ is a tree-edge or 
    \item $v_iv_{i+1}$ is in $\overrightarrow{O}$. 
  \end{itemize}
  Then, there exists $i\in\N$ such that $v_i=v^\star$. Set 
  \(
  \ell = \min\{i\geq 0\,:\, v_i=v^\star\}.
  \)
  Then \[
  \ell = \lambda^{\op \ell}(v_0v_1)-\lambda(\kappa_{\min})+3\delta,\]
  where $\delta = 0$ if $\kappa^\ell(v_0v_1)\in \bbrcy{\xi^\dagger,\kappa_{\min}}$ and $\delta=1$ otherwise. 
\end{proposition}
It is immediate that $v_i$ can be defined for all $i\in\N$ since each vertex is the end-point of at least one tree-edge. The proof of the proposition will be done later in this subsection. 
\begin{definition}
  Let $e$, $(v_i)_{i\geq 0}$, and $\ell$ be as in Proposition~\ref{prop:LMPreaches}. We set \(
  P(e)=v_0,\ldots, v_\ell, 
  \)
  and call $P(e)$ the \emph{modified leftmost path} started from $e$. 
\end{definition}

\begin{remark}
  The modified leftmost paths are a variant of the classical leftmost paths, which are paths $v_0,v_1,\ldots,v_k$ such that $v_0v_{1} \in \overrightarrow{O}$, $v_k$ is the root vertex, and for each $i\in\{1,\ldots,k-1 \}$, $v_iv_{i+1}$ is the first edge incident to $v_i$, among edges after $v_{i-1}v_{i}$ in clockwise direction around $v_i$ that belong to $\overrightarrow{O}$.
  
  Leftmost paths satisfy some nice properties in triangulations of the sphere (e.g., they are self-avoiding). These properties do not hold anymore in triangulations of the $p$-gon, but fortunately, they do hold for modified leftmost paths.
\end{remark}

Recall the correspondence between edges of $M$ and edges of $F$ given at the end of Section~\ref{sub:PS}. For an oriented edge $e$  of $M$, denote by $\pi_M(e)$ its image in $F$. 
Then, we have the following result (see Figure~\ref{fig:LMP}).

\begin{figure}[t]
  \centering
  {\includegraphics[page=1,scale=0.6]{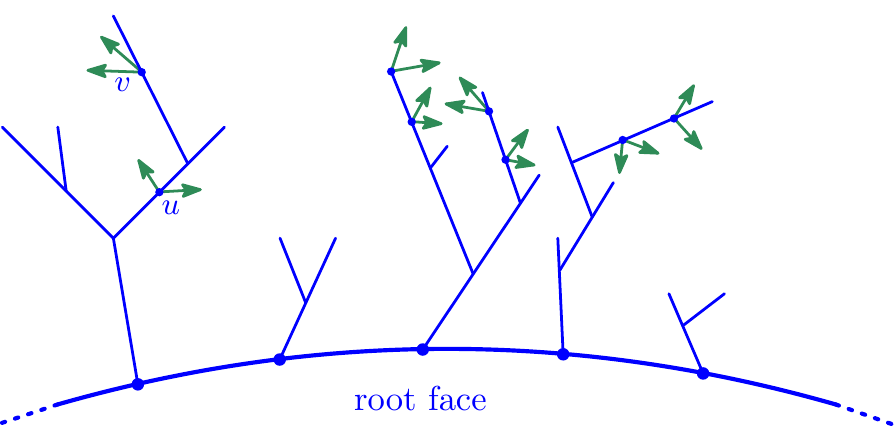}}
  \qquad\qquad
  {\includegraphics[page=2,scale=0.6]{images/LMP}}
  \caption{\label{fig:LMP}The construction of a modified leftmost path seen on the blossoming forest. Note that not all the blossoms are shown in the figure. }
\end{figure}
\begin{lemma}\label{prop:successors}
  Let $P=(v_0,v_1,\ldots)$ be as in Proposition~\ref{prop:LMPreaches}. For any $i\geq 1$, we have:
  \[
  s\Big(\kappa_F^{\op \ell}\big(\pi_M(v_{i-1}v_i)\big)\Big) = \kappa_F^{\op \ell}\big(\pi_M(v_{i}v_{i+1})\big),
  \]
  where $s$ is the successor function defined in Section~\ref{sub:reformulationPS}. 
\end{lemma}
\begin{proof}
  In the proof, with a slight abuse of notations, vertices and tree-edges of $M$ are identified with their images in $F$. 
  
  Assume first that $\{v_{i-1},v_{i}\}$ is a tree-edge. Then, by definition of the labeling and of the successor function, $s\Big(\kappa_F^{\op \ell}(v_{i-1}v_i)\Big)=\kappa^{\op r}(v_iv_{i-1})$. Let $w \in V(F)$ be such that $\kappa^{\op r}(v_{i}v_{i-1})=(\{v_{i-1},v_{i}\},\{v_{i},w\})$. If $w \notin \mathcal{B}(F)$ then 
  $w=v_{i+1}$. If $w \in \mathcal{B}(F)$, then $\v(s(w))=v_{i+1}$. In either case, $\kappa_F^{\op \ell}\big(\pi_M(v_{i}v_{i+1})\big)=\kappa^{\op r}_F(v_{i}v_{i-1})$, and the proposition is proved in this case.
  
  Next assume that $\{v_{i-1},v_{i}\}$ is a closure-edge. Then $s\Big(\kappa_F^{\op \ell}\big(\pi_M(v_{i-1}v_i)\big)\Big)$ is the corner -- say $\kappa$ -- of $F$ which is split during the closure operation of the stem $\pi_M(v_{i-1}v_i)$. Again, by definition of $P$, $\kappa = \kappa_F^{\op \ell}\big(\pi_M(v_iv_{i+1})\big)$, which concludes the proof.
\end{proof}

We can now give the proof of Proposition~\ref{prop:LMPreaches}.
\begin{proof}[Proof of Proposition~\ref{prop:LMPreaches}]
  For $\kappa \in \mathcal{C}(F)\backslash \{\kappa_{\min}\}$ it follows from the definition of the successor function and of $\kappa_{\min}$, that $s(\kappa)\in (\kappa,\kappa_{\min}]_{\mathrm{cyc}}$. By Lemma~\ref{prop:successors}, since the number of corners of $F$ is finite, there exists $i\in\N$ such that 
  $\kappa_F^{\op \ell}\big(\pi_M(v_iv_{i+1})\big))=\kappa_{\min}$. Since by Claim~\ref{claim:vstar}, $v^\star$ is incident to $\kappa_{\min}$, it proves the first assertion. 
  The fact that $s(\kappa)\in (\kappa,\kappa_{\min}]_{\mathrm{cyc}}$ also implies that there exists at most one index $1\leq j_0 < \ell$ such that $\kappa_F^{\op \ell}\big(\pi_M(v_{i}v_{i+1})\big)$ is the successor of the second type of $\kappa_F^{\op \ell}\big(\pi_M(v_{i-1}v_{i})\big)$, and such an index exists if and only if $\kappa^{\op \ell}_F\big(\pi_M(v_{0}v_{1})\big) \in (\kappa_{\min},\xi^\dagger)_{\mathrm{cyc}}$.
  
  Let $\{u,v\} \in E(M)$ be such that either $uv\in \overrightarrow{O}$ or $\{u,v\}$ is a tree-edge. Then $\lambda^{\op \ell}_M(uv)=\lambda_F^{\op\ell}\big(\pi_M(uv)\big)$. By the definition of a modified leftmost path, for any $0\leq i \leq \ell$, either $\{v_i,v_{i+1}\}$ is a tree-edge, or $v_iv_{i+1}\in \overrightarrow{O}$ (or both). Together with Lemma~\ref{prop:successors}, it implies  
  \[
  \lambda_{M}^{\op\ell}(v_{i}v_{i+1})=\lambda_{M}^{\op\ell}(v_{i-1}v_{i}) -1 +3\delta_{i=j_0},
  \]
  which concludes the proof.
\end{proof}

\begin{corollary}\label{cor:labelLMP}
  Let $e=v_0v_1$ be an oriented edge of $\overrightarrow{O}$, write $P(e)=(v_0,v_1,\ldots,v_\ell=v^\star)$. Then for any $i\in\{0,\ldots,\ell-1\}$, $\lambda^{\op\ell}(v_iv_{i+1})=\lambda^{\op\ell}(v_0v_{1})-i + 3\delta$, with $\delta\in\{0,1\}$. 
\end{corollary}
\begin{proof}
  The result follows directly from Lemma~\ref{prop:successors} and the fact (already observed in the proof of Proposition~\ref{prop:LMPreaches}) that, along a modified leftmost path, there is at most one successor of the second type.
\end{proof}

\begin{corollary}\label{cor:distvstar}
  For $u\in V(M)$, 
  \[
  d_M(u,v^{\star})\leq X(u)-\ell_{\min}+3.
  \]
\end{corollary}
\begin{proof}
  For $u\in V(M)$, let $uv\in \overrightarrow{O}$ be such that $\lambda^{\op \ell}(uv)=X(u)$ (the existence of such an edge is guaranteed by the definition of $X(u)$). Then the modified leftmost path started at $uv$ reaches $v^\star$ and its length is bounded above by $X(u)-\ell_{\min}+3$ by Proposition \ref{prop:LMPreaches}. 
\end{proof}

\subsection{The distance between two points}
We saw in the previous section that modified leftmost paths give upper bounds on the distances from any vertex to $v^\star$. In this section, we further build on modified leftmost paths to get an upper bound for the distance between any two fixed points. This is similar to the construction given in \cite[Section~7.2]{ABA-Simple} and we only include it here for completeness.

For edges $uu'$ and $vv'$ we recall \eqref{eq:cyclic-interval} and define 
\[
\check X(uu',vv')=\min\{\lambda(\xi)\,:\, \text{for }\xi\in \bbrcy{\kappa^{\op \ell}(uu'),\kappa^{\op \ell}(vv')}\}.
\]

\begin{proposition}\label{prop:dist2points}
  Let $u,v$ be two vertices of $M$ and let $\overrightarrow{uu'}$ and $\overrightarrow{vv'}$ be two edges of $\overrightarrow{O}$ such that $\lambda^{\op \ell}(uu')=X(u)$ and $\lambda^{\op \ell}(vv')=X(v)$.
  Then
  \[
  d_M(u,v)\leq X(u)+X(v)-2\max\{\check X(uu',vv'),\check X(vv',uu')\}+6.
  \]
  
\end{proposition}
\begin{proof}
  Assume without loss of generality that $\kappa^{\op \ell}(uu')\pc\kappa^{\op \ell}(vv')$ (or, in other words, that $\xi^\dagger \notin\bbrcy{\kappa^{\op \ell}(uu'),\kappa^{\op \ell}(vv')})$.

  Consider the modified leftmost paths $P(uu')=(u_0\!=\!u,u_1\!=\!u',\ldots,u_{\ell_u})$ and $P(vv')=(v_0\!=\!v,v_1\!=\!v',\ldots,v_{\ell_v})$, started from $uu'$ and $vv'$, respectively. For $i\in\{0,\ldots,\ell_u-1\}$ and $j\in\{0,\ldots,\ell_v-1\}$ we set:
  \[\kappa^{(u)}_i=\kappa^{\op \ell}(u_iu_{i+1})\quad\text{and}\quad\kappa^{(v)}_j=\kappa^{\op \ell}(v_jv_{j+1}).\] 
  We first prove that
  \[
  d_M(u,v)\leq X(u)+X(v)-2\check X(uu',vv')+6.
  \]
  If $\check X(uu',vv')=\lambda_{\min}$, then the result is proved by the triangle inequality and by Corollary~\ref{cor:distvstar}. We may then assume that $\check X(uu',vv')>\lambda_{\min}$, which implies in particular that $\ell_u$ and $\ell_v$ are greater than $X(u)-\check X(uu',vv')$ and $X(v)-\check X(uu',vv')$, respectively.

  By Lemma~\ref{prop:successors}, $\lambda(\kappa^{(u)}_{X(u)-\check X(uu',vv')})=\check X(uu',vv')$. Furthermore, by definition of $\check X(uu',vv')$, the first corner (for $\preceq_{\mathrm{ctr}}$) labeled $\check X(uu',vv')-1$ after $\kappa^{(u)}_{X(u)-\check X(uu',vv')}$ belongs to $\bbrcy{\kappa^{\op \ell}(vv'),\kappa^{\op \ell}(uu')}$. Now, recall the definition of $\xi^\circ$ from Section \ref{sub:reformulationPS}, and define
  \[
  \check X(vv') =\min\{\lambda(\xi)\,:\, \text{for }\xi\in \bbrcy{\kappa^{\op \ell}(vv'),\xi^\circ}\}.
  \] 
  We have either $\check X(vv') < \check X (uu',vv')$ or $\check X(vv') \geq \check X (uu',vv')$.
  \begin{itemize}
    \item Case $\check X(vv') < \check X (uu',vv')$: In this case, for $j\in\{0,\ldots,X(v)-\check X(uu',vv')+1\}$, $\kappa^{(v)}_j\in \bbrcy{\kappa^{\op \ell}(vv'),\xi^\circ}$ and $\lambda\big(\kappa^{(v)}_{X(v)-\check X(uu',vv')}\big)=\check X(uu',vv')$. Hence
    \[s\big(\kappa^{(u)}_{X(u)-\check X(uu',vv')}\big)=s\big(\kappa^{(v)}_{X(v)-\check X(uu',vv')}\big),
    \]
    and therefore $u_{X(u)-\check X(vv',uu')+1}=v_{X(v)-\check X(vv',uu')+1}$. 
    \item Case $\check X(vv') \geq \check X (uu',vv')$: In this case, $\kappa^{(v)}_{X(v)-\check X(uu',vv')+1} \in \bbrcy{\xi^\dagger,
      \kappa^{\op \ell}(uu')}$ and $\lambda\big(\kappa^{(v)}_{X(v)-\check X(uu',vv')+2}\big)=\check X(uu',vv')$. Hence
    \[
    s\big(\kappa^{(u)}_{X(u)-\check X(uu',vv')}\big)=s\big(\kappa^{(v)}_{X(v)-\check X(uu',vv')+2}\big),\]
    and therefore $u_{X(u)-\check X(vv',uu')+1}=v_{X(v)-\check X(vv',uu')+3}$.
  \end{itemize}
  In both cases, concatenating the corresponding subpaths of $P(uu')$ and $P(vv')$ yields the desired bound. 
  
  The proof to establish that \(d_M(u,v)\leq X(u)+X(v)-2\check X(vv',uu')+6\) is similar and left to the reader.
\end{proof}

\section{Modified leftmost paths are almost geodesic}\label{sec:lower}
\label{sec:geodesic}
For $M\in\mtrig$ and $u\in \mathring{V}(M)$, we denote by $\wt d_M(u,v^\star)$ the minimum length of a path from $u$ to $v^\star$ which does not touch the root face of $M$ (except at its extremity if $v^\star$ is incident to the root face). In particular, if $\mathring{M}$ denotes the subgraph of $M$ generated by the inner vertices of $M$ and $v^\star\in \mathring{V}(M)$, then $\wt  d_M(u,v^\star)=d_{\mathring{M}}(u,v^\star)$. The main result of this section is the following.   
\begin{proposition}\label{prop:lower}
  Fix $\varepsilon >0$. Then, for $p$ sufficiently large, $\Map_p^*=(\Map_p,f^*)$ sampled from $\Bol_{\op{III}}^*(p)$, and $u$ a uniform vertex in $V(\Map_p^*)$,
  \[
  \p{u\in \mathring{V}(\Map_p^*) \text{\,\,\,and\,\,\,} \wt d_{\Map_p}(u,v^\star)<X(u)-X(v^\star)-\varepsilon p^{1/2}}<\varepsilon.
  \]  
\end{proposition}

\subsection{Comparison of the lengths of paths}\label{sub:compLength}
Throughout this section, we assume that $M$ is a fixed element of $\mtrig$ endowed with its minimal orientation $\overrightarrow{O}$ and such that $v^\star\in \mathring{V}(M)$.\footnote{Note that, by Lemma \ref{lem:nottouching} below, if $M$ is sampled from $\Bolm(p)$, then the condition $v^\star\in \mathring{V}(M)$ is satisfied with probability converging to 1 as $p\rta\infty$.} 
We assume that $e=uv$ is a fixed oriented edge of $\overrightarrow{O}$ such that for any $w\in P(uv)$, we have $w\in \mathring{V}(M)$. We write $P=P(uv)=(u=v_0,v=v_1,\ldots,v_{\ell}=v^\star)$. 
Finally, we assume $Q$ is a self-avoiding path (i.e., $Q$ is a path where no vertex is used more than once) from $u$ to $v^\star$ which starts with the edge $uv$ such that $Q$ only contains vertices in $\mathring{V}(M)$.
\medskip

We decompose $Q$ into edge-disjoint subpaths $R_1,R_2,\ldots, R_t$ verifying the following property: the extremities of $R_i$ belongs to $P$ and either no other vertices of $R_i$ belong to $P$ or $R_i$ is a subpath of $P$. We assume the decomposition is maximal in the sense that $t$ is as small as possible; note that the maximal decomposition is unique.

The subpaths $R_j$ such that only their extremity belong to $P$ are called \emph{excursions of $Q$ away from $P$}. For an excursion $R$ of $Q$, denote by $v_i$ and $v_j$ (with $i<j$) the extremities of $R$. We say that $R$ is \emph{separating} if the cycle $R\cup \{v_i,v_{i+1},\ldots,v_j\}$ separates the root face from infinity. 
Note that even if the orientation of the edges of $P$ does not necessarily agree with the orientation of $\overrightarrow{O}$, we may view $P$ as an oriented path from $v_0$ to $v_{\ell}$. With this in mind, if $R$ is an excursion of $Q$ away from $P$ then exactly one of the following holds  (see Figure~\ref{fig:casesLMP}):
\begin{figure}
  \centering
  \includegraphics[scale=0.7]{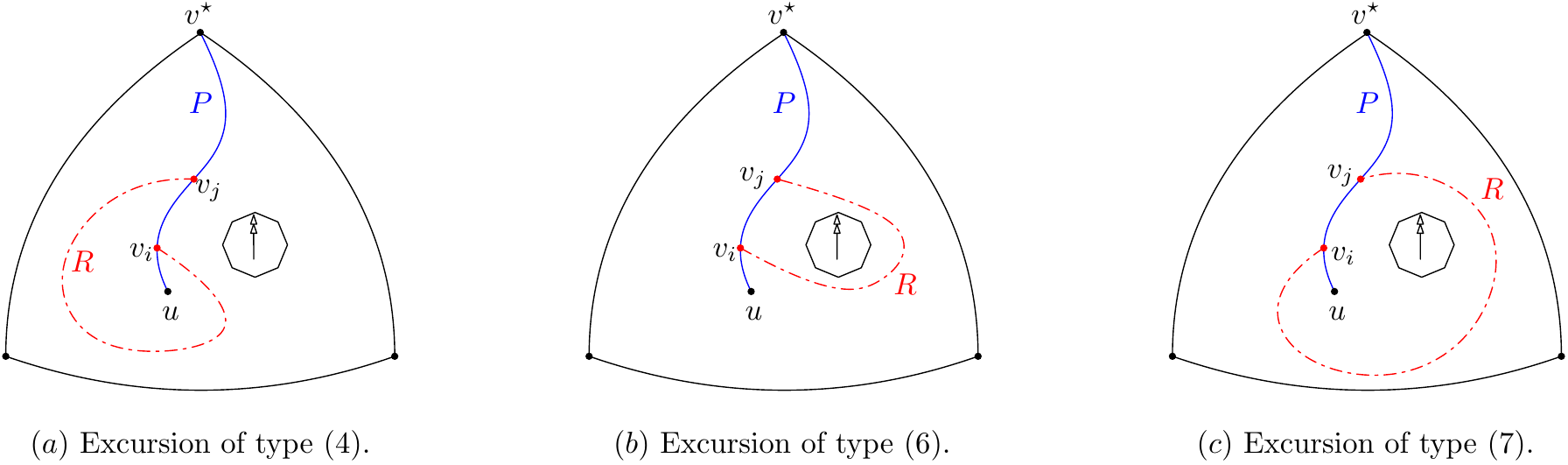}
  \caption{\label{fig:casesLMP}The three types of excursions away from a modified leftmost path that can cause a shortcut, see Lemma~\ref{lem:diffDistances}. The triangular marked face is the unbounded face and the root corner is indicated by a double arrow.}
\end{figure} 

\begin{enumerate}
  \item 
  $R$ leaves $P$ on its left at $v_i$ and returns on its left at $v_j$, and $R$ is not separating. 
  \item $R$ leaves $P$ on its right at $v_i$ and returns on its rig{}ht at $v_j$, and $R$ is not separating. 
  \item $R$ leaves $P$ on its left at $v_i$ and returns on its right at $v_j$, and $R$ is not separating. 
  \item $R$ leaves $P$ on its right at $v_i$ and returns on its left at $v_j$, and $R$ is not separating. 
  \item $R$ leaves $P$ on its left at $v_i$ and returns on its left at $v_j$, and $R$ is separating. 
  \item $R$ leaves $P$ on its right at $v_i$ and returns on its right at $v_j$, and $R$ is separating. 
  \item $R$ leaves $P$ on its left at $v_i$ and returns on its right at $v_j$, and $R$ is separating. 
  \item $R$ leaves $P$ on its right at $v_i$ and returns on its left at $v_j$, and $R$ is separating. 
\end{enumerate}
For $1\leq i \leq 8$, an excursion is said to be of type $i$, depending on which of the above cases applies. The number of excursions of type $i$ of $Q$ is denoted by $n_{i,Q}$ (or $n_i$ if there is no ambiguity).

\begin{lemma}\label{lem:diffDistances}
  Assume that $R$ is an excursion of $Q$ away from $P$ such that $R\cap P=\{v_i,v_j\}$, with $i,j\in\{0,1,\dots,\ell \}$ and $i<j$. Then, depending on which possibility we consider, we have the following bounds for the length $l_R$ of $R$: 
  \begin{align*}
    (1)&\quad l_R\geq j-i & (5)&\quad l_R \geq j-i+6\\
    (2)&\quad l_R\geq j-i & (6)&\quad l_R\geq j-i-6\\
    (3)&\quad l_R\geq j-i+3&(7)&\quad l_R \geq j-i-3\\
    (4)&\quad l_R \geq j-i-3&(8)&\quad l_R \geq j-i+3
  \end{align*}
\end{lemma}
\begin{proof}
  The proof of this lemma relies on Euler's formula, on the property of the orientation we consider, and on the fact that there is no outgoing (for $\overrightarrow{O}$) edge on the left side of $P$. More formally, for $i>1$, there is no $v_i\wh v\in \overrightarrow{O}$ between $v_{i-1}v_i$ and $v_{i}v_{i+1}$ when turning clockwise around $v_i$.
  We write $R=(w_0=v_i,\ldots,w_{l_R}=v_j)$ and denote $P_{i,j}$ the subpath of $P$ between $v_i$ and $v_j$.
  
  \textbf{Cases $(1)$, $(2)$, $(3)$, and $(4)$}. We consider the subtriangulation $\Delta$ that is bounded by $R$ and $P_{i,j}$ and which contains neither the root face nor the outer face. This is a triangulation of a $j-i+l_R$-gon. We denote by $v_{\inn}$ and $e_{\inn}$, respectively, the number of inner vertices and inner edges. Euler's formula directly implies that
  \[
  3v_{\inn}+j-i+l_R=3+e_{\inn}.
  \] 
  Now, for case $(1)$, since $R$ leaves $P$ on its left by definition, the edge $v_{i-1}v_i$ does not belong to $\Delta$ and by definition of $R$, neither does the edge $v_{j}v_{j+1}$. Since a modified leftmost path has no outgoing edges on its left, it implies that $\overrightarrow{w_1w_0} \in \overrightarrow{O}$ and $\overrightarrow{w_{l_R-1}w_{l_R}}\in \overrightarrow{O}$. Further, it gives that the tail of any oriented inner edge of $\Delta$ is either an inner vertex or is equal to $w_{k}$, for $0<k<l_R$. There are exactly $3v_{\inn}$ inner edges of the former kind, since any inner vertex has exactly three outgoing edges. There are at most $3(l_R-1)-l_R$ inner edges of the latter kind, since each vertex $w_{k}$ has exactly three outgoing edges, but exactly $l_R$ of these lie on $R$ and there may be some edges which are not on $\Delta$. Combining these estimates, we get:
  \[
  e_{\inn}\leq 3v_{\inn}+3(l_R-1)-l_R=3v_{\inn}+2l_R-3,
  \] 
  and the result follows. 
  
  For case $(2)$, all the oriented edges whose tail is equal to $v_{k}$, with $i<k<j$, are either inner edges of $\Delta$ or belongs to $P_{i,j}$ and a similar argument shows that: 
  \[
  e_{\inn} \geq 3v_{\inn}+3(j-i-1)-j-i=3v_{\inn}+2(j-i)-3.
  \]
  For case $(3)$,
  all the oriented edges whose tail is equal to $v_{k}$, with $i\leq k<j$, are either inner edges of $\Delta$ or belongs to $P_{i,j}$ (note that indeed $\overrightarrow{w_1w_0}\in \overrightarrow{O}$) and a similar argument shows that: 
  \[
  e_{\inn} \geq 3v_{\inn}+3(j-i)-j-i=3v_{\inn}+2(j-i).
  \]
  For case $(4)$, the tail of any oriented inner edge of $\Delta$ is either an inner vertex or is equal to $w_{k}$, for $0\leq k<l_R$, which gives: 
  \[
  e_{\inn}\leq 3v_{\inn}+3l_R-l_R=3v_{\inn}+2l_R.
  \]
  
  \textbf{Cases $(5)$, $(6)$, $(7)$, and $(8)$}. We consider the subtriangulation $\Delta$ that is bounded by $R$ and $P_{i,j}$ and which contains the root face. This is a triangulation of a cylinder with perimeters $j-i+l_R$ and $p$. We denote by $v_{\inn}$ and $e_{\inn}$ the number of inner vertices and inner edges, respectively. Euler's formula directly implies that
  \begin{equation}\label{eq:EulerCyl}
  3v_{\inn}+j-i+l_R+p=e_{\inn}.
  \end{equation}
  Modulo replacing $3v_{\inn}$ by $3v_{\inn}+p-3$, the bounds we obtain above for cases $(1)$, $(2)$, $(3)$, and $(4)$ give bounds for cases $(5)$, $(6)$, $(7)$, and $(8)$, respectively:
  \begin{align*}
    e_{\inn}&\leq 3v_{\inn}+2l_R+p-6,& \text{for case }(5)&,\\
    e_{\inn}&\geq 3v_{\inn}+2(j-i)+p-6,& \text{for case }(6)&,\\
    e_{\inn}&\geq 3v_{\inn}+2(j-i)+p-3,& \text{for case }(7)&,\\
    e_{\inn}&\leq 3v_{\inn}+2l_R+p-3,& \text{for case }(8)&.
  \end{align*}
  Together with \eqref{eq:EulerCyl}, it gives the desired result.
\end{proof}

\subsection{Modified leftmost paths and shortcuts}\label{sub:shortcuts}
Consider as above the subdivision of $Q$ into maximal subpaths $R_1,\ldots,R_t$. We assume that $R_1,\ldots,R_t$ are ordered in such a way that $Q$ is the concatenation of $R_1,\ldots,R_t$. In particular, $u$ is the first vertex of $R_1$ and $v^\star$ is the last vertex of $R_t$. For $s\in\{1,\dots,t \}$, let $i_s<j_s\in\{0,1,\dots,\ell \}$ be such that $v_{i_s}$ and $v_{j_s}$ are both extremities of $R_s$. 

\begin{lemma}\label{lem:noNested6}
  There exists a shortest path between $u$ and $v^\star$ in $\mathring{M}$ such that all its excursions away from $P$ are of types $(4)$, $(6)$, or $(7)$
  and such that there are no 3 consecutive excursions away from $P$ of type $(6)$. More formally, if there exist $s_1,s_2,s_3\in \{1,\ldots,t\}$ such that $R_{s_i}$ is an excursion of type $(6)$ for $i\in\{1,2,3\}$ and $s_1<s_2<s_3$, then there exists $s\in\{s_1+1,s_1+2,\ldots,s_3-1\}$ such that $R_s$ is of type $(4)$ or $(7)$.\footnote{We could in fact prove a much stronger statement, namely that there exists a shortest path with at most one excursion of type $(6)$. However, since the proof is more involved and requires a lot of case-by-case analysis, we only prove the above lemma, which is sufficient for our purposes.}
\end{lemma} 

\begin{proof}
  Let $Q$ be a shortest path between $u$ and $v^\star$ and consider the decomposition of $Q$ into maximal subpaths as above. Since $Q$ is a shortest path, by Lemma~\ref{lem:diffDistances}, it has no excursion of types $(3)$, $(5)$ and $(8)$. Furthermore, we can replace all excursions of types $(1)$ and $(2)$ by the corresponding subpaths of $P$ without changing its length. This leaves us with excursions of types $(4)$, $(6)$ or $(7)$.
  
  \begin{figure}
    \centering
    \includegraphics[scale=0.8,page=3]{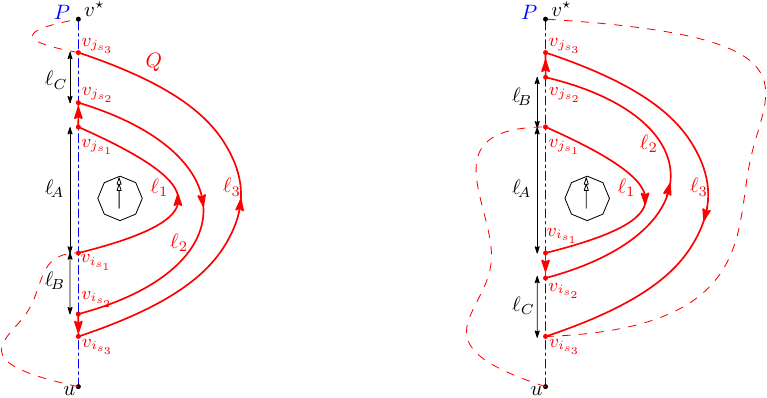}
    \caption{The configurations of paths considered in the proof of Lemma~\ref{lem:noNested6}.\label{fig:noNested}}
  \end{figure}
  
  Assume that $R_{s_1},R_{s_2}$ and $R_{s_3}$ are 3 consecutive excursions of type $(6)$. Since $Q$ is a shortest path, it is self-avoiding. Therefore, $R_{s_1},R_{s_2}$ and $R_{s_3}$ are nested, i.e.,
  \begin{itemize}
    \item $i_{s_3}\leq i_{s_2} \leq i_{s_1}\leq j_{s_1}\leq j_{s_2}\leq j_{s_3}$ or
    \item $i_{s_1}\leq i_{s_2} \leq i_{s_3}\leq j_{s_3}\leq j_{s_2}\leq j_{s_1}$.
  \end{itemize}
  The two cases being fully symmetric, we assume, without loss of generality, that the first one holds. Two further symmetric cases (displayed respectively on the left and on the right of Figure~\ref{fig:noNested}) can then occur, depending on whether $v_{i_{s_1}}$ or $v_{j_{s_{1}}}$ is first reached by $Q$ (when going from $u$ to $v^\star$). 
  
  We start with the case where $v_{i_{s_1}}$ is reached before $v_{j_{s_{1}}}$. Denote by $\ell_1$, $\ell_2$, and $\ell_3$ the length of $R_{s_1}$, $R_{s_2}$, and $R_{s_3}$, respectively, and denote by $\ell_A$, $\ell_B$, and $\ell_C$ the length of $P$ between $v_{i_1}$ and $v_{j_1}$, $v_{i_1}$ and $v_{i_2}$, and $v_{j_2}$ and $v_{j_3}$, respectively.
  Then, by Lemma~\ref{lem:diffDistances}, 
  \[
  \ell_3\geq (\ell_A+\ell_B+\ell_C-6). 
  \]
  If $\ell_A=1$, then we can replace $R_{s_1}$ by the subpath of $P$ between $v_{i_1}$ and $v_{j_1}$ without lengthening $Q$, so we can assume $\ell_A\geq 2$. Similarly, since the subpath of $Q$ between $v_{i_1}$ and $v_{i_2}$ has length at least 2 (because it contains two excursions away from $P$), we can assume that $\ell_B\geq 3$. Hence, the previous bound gives $\ell_3\geq \ell_C-1$. Since $\ell_2\geq 1$ we have $\ell_3+\ell_2\geq \ell_C$, so we can replace the subpath of $Q$ between $v_{j_2}$ and $v_{j_3}$ (consisting of the two excursions of length $\ell_2$ and $\ell_3$ plus the possibly empty subpath of $P$ between $v_{i_2}$ and $v_{i_{3}}$) by the corresponding length $\ell_C$ subpath of $P$  without lengthening it, which concludes the proof in this case.
  
  We now move to the case where $v_{j_{s_1}}$ is reached before $v_{i_{s_{1}}}$. We keep the same definition of $\ell_1,\ell_2,\ell_3$, and $\ell_A$, but denote by $\ell_B$ and $\ell_C$ the length of $P$ between $v_{j_1}$ and $v_{j_2}$, and $v_{i_2}$ and $v_{i_3}$, respectively. Then, the exact same proof as above applies.
\end{proof}

We now make the following assumption:
\begin{center}
  \emph{Until the end of this section, $Q$ is a shortest path 
    in $\mathring{M}$ between $u$ and $v^\star$\\ satisfying the properties of Lemma~\ref{lem:noNested6}.}  
\end{center}

Recall the definition of $n_1,\dots,n_8$ right above the statement of Lemma \ref{lem:diffDistances} and define
\[
\sigma(Q,e)=n_4+n_7.
\]
\begin{claim}\label{claim:shortcuts}
  If $e=uv$, $P(uv)$, and $Q$ are as above, then 
  \[
  |Q| \geq |P(e)|-15\sigma(Q,e)-12.
  \]
\end{claim}
\begin{proof}
  By Lemma~\ref{lem:diffDistances}, we have $|Q| \geq |P(e)|-3(2n_6+n_4+n_7)$. By Lemma~\ref{lem:noNested6}, $n_6\leq 2(n_4+n_7+1)$, which concludes the proof by definition of $\sigma(Q,e)$.
\end{proof}
Now, let $C$ be a cycle in $\mathring{M}$ (where a cycle is a path that starts and ends at the same vertex and uses each vertex a most one, except for its starting point). Let $V^\Delta[C]$ (respectively $V^\inn[C]$) be the subset of vertices of $M$ that lies either on $C$ or on the side of $C$ containing the triangular marked face (respectively, on $C$ or on the side of $C$ which does not contain the marked face), so that  $V^\Delta[C]\cap V^\inn[C] = C$.  

\begin{proposition}\label{prop:wrappings}
  If $\sigma(Q,e)>2$, then there is a cycle $C$ in $\mathring{M}$ such that
  \begin{itemize}
    \item $|C| \leq \frac{6|Q|}{\sigma(Q,e)-2} +3 $, and
    \item $\max_{y\in V^\delta[C]}X(v)-\min_{y\in V^\delta[C]}X(v) \geq \lfloor\sigma(Q,e)/3\rfloor -4$ for $\delta\in\{\Delta,\inn\}$.
  \end{itemize}
\end{proposition}
\begin{proof}
  Write $Q=(u=u_0,v=u_1,u_2,\ldots,v^\star).$ 
  
  Consider as above the subdivision of $Q$ into maximal subpaths $R_1,\ldots,R_t$. For $s\in\{1,\dots,t \}$ and $i\in\{1,\dots,8 \}$, let $n_i(s)$ denote the number of subpaths of type $i$ among $\{R_1,\ldots R_s\}$, and define $\sigma(s)= n_4(s)+n_7(s)$. In particular, $\sigma(0)=0$, $\sigma(t)=\sigma(Q,e)$, and for $s\in[t]$, 
  \begin{equation}
  \sigma(s)-\sigma(s-1) = 
  \begin{cases}
  1&\text{if }R_s\text{ is an excursion away from $P$ of type (4) or (7)},\\
  0&\text{otherwise. }
  \end{cases}
  \label{eq:sigma-diff}
  \end{equation}
  Let $m=\lfloor \sigma(Q,e)/3\rfloor$. By \eqref{eq:sigma-diff} there are at least $m$ values of $s$ such that $R_s$ has type (4) or (7) and 
  \begin{equation}
  m+1\leq \sigma(s)\leq 2m.
  \label{eq:ns}
  \end{equation}
  Among these values of $s$, let $s^\star\in \{1,\dots,t \}$ be such that $|R_{s^\star}|$ is minimal. Recall that the length of a path is defined as its number of edges. Then, first using that $|R_{s^\star}|$ is minimal and then using the definition of $m$, 
  \begin{equation}
  \wt d_M(u,v^\star)\geq m|R_{s^\star}|\geq \frac{\sigma(Q,e)-2}{3}|R_{s^\star}|.
  \label{eq:Rsstar}
  \end{equation}
  By Lemma \ref{lem:diffDistances}, the subpath of $P(e)$ between the endpoints of $R_{s^\star}$ has at most three more edges than $R_{s^\star}$, so the cycle $C$ obtained by concatenating this subpath and $R_{s^\star}$ has at most $2|R_{s^\star}|+3$ edges. Therefore, first using \eqref{eq:Rsstar} and then using that $Q$ is a shortest path, 
  \[
  |C|\leq \frac{6\wt d_M(u,v^\star)}{\sigma(Q,e)-2}+3= \frac{6|Q|}{\sigma(Q,e)-2}+3.
  \]

  For $w\in V(\mathring{M})$, and for any corner $\kappa$ incident to $w$, $\lambda(\kappa)-X(w)\in\{0,1,2\}$. Together with Corollary~\ref{cor:labelLMP}, this observation yields, $X(v_0)-X(v_{j_{s^\star}})\geq j_{s^\star}-5$ and $X(v_{i_{s^\star}})-X(v^\star)\geq \ell-i_{s^\star}-5$. Using again that $j_{s^\star} \geq m+2$ and $\ell-i_{s^\star} \geq m+1$, it completes the proof. 
  
\end{proof}

\subsection{Proof of Proposition~\ref{prop:lower}}\label{sub:lower}
Throughout this section, $F_p^*$ and $\Map^*_p=(\Map_p,f)=\chi(F^*_p)$ are \emph{random} forests and maps sampled as~in Proposition~\ref{prop:Bol}.

Before proving Proposition~\ref{prop:lower}, we state three lemmas. The first lemma rules out the possibility for $\Map^*_p$ to have a short cycle that separates it into two macroscopic parts. Its proof, postponed to the end of the section, relies mostly on enumerative arguments.
More precisely, fix $K\in \mathbb{N}$ and $\alpha>0$, let $\Gamma_K(\alpha)$ be the event that a triangulation of the $p$-gon with a marked triangular face admits a simple cycle $C$ such that
\begin{itemize}
  \item$|C|\leq K$,
  \item $C$ does not contain a vertex incident to the root face, and 
  \item $|V^\Delta[C]|\geq \alpha p^2$ and $|V^\inn[C]|\geq \alpha p^2$.
\end{itemize}

\begin{lemma}\label{lem:sepCircles}
  For any $K \in \mathbb{N}$ and $\alpha,\varepsilon>0$, for $p$ sufficiently large,
  \[
  \Bol_{\op{III,p}}^*\left(\Gamma_K(\alpha)\right)\leq \varepsilon.
  \]
\end{lemma}

To state the second lemma, we introduce some further notations. For $p\geq 3$ and $F\in \mathcal{F}^*_p$, recall the definition of the contour exploration $\beta:[0,2|V(F)|-p]_{\Z}\rightarrow V(F)$. For $i\in\{0,\dots, 2|V(F)|-p\}$ and $\Delta >0$, let: 
\begin{align*}
  g_F(i,\Delta)&=\sup\{j<i\,:\,|X(\beta(j))-X(\beta(i))|\geq \Delta\text{ or }j=0\}\\
  d_F(i,\Delta)&=\inf\{j>i\,:\,|X(\beta(j))-X(\beta(i))|\geq \Delta\text{ or }j=2|V(F)|-p\}
\end{align*}
Then let $N(i,\Delta)=\{v\in V(F)\,:\,\exists j\in \{g_F(i,\Delta),\ldots,d_F(i,\Delta)\} \text{ with }\beta(j)=v\}$ be the set of vertices visited by the contour exploration at least once between times $g_F(i,\Delta)$ and $d_F(i,\Delta)$. The following lemma, whose proof is omitted, is a straightforward generalization of~\cite[Lemma~8.2]{ABA-Simple} and relies on the convergence result established in Proposition~\ref{prop:contour-label}.\footnote{The convergence established in Proposition~\ref{prop:contour-label} is in fact for the height process rather than for the contour process but classical results (see e.g.\ \cite[Section~1.6]{LeGallSurvey}) ensure that both processes converge jointly to the same limit.}
\begin{lemma}\label{lem:regContour} For all $\varepsilon,\gamma>0$ there exists $\alpha$ such that for $p$ sufficiently large,
  \[
  \p{\inf\big\{|N(i,\gamma \sqrt{p})|\,:\,0\leq i \leq 2|V(F_{p}^*)|-p\big\}\geq \alpha p^2}\geq 1-\varepsilon.
  \]
\end{lemma}

Lastly, the third lemma is the following result.
\begin{lemma}\label{lem:nottouching} For all $\varepsilon>0$, let $u$ be sampled uniformly in $V(\Map_p^*)$. Then for $p$ sufficiently large,
  \[
  \p{\exists e\in E(\Map_p^*), \text{ with }u\in e\,:\,P(e)\text{ contains a vertex incident to the root face of }\Map_p^*}\leq \varepsilon.
  \]
\end{lemma}
We omit the proof of this result, but remark that \cite[Lemma~19]{BeMi15} gives the corresponding continuum property, and that, using Proposition~\ref{prop:contour-label} and Lemma~\ref{prop:successors}, we can transfer the result to the discrete setting, following the continuum proof.
\smallskip

The remainder of the proof of Proposition~\ref{prop:lower} is now an almost verbatim translation of the proof of~\cite[Theorem~8.1]{ABA-Simple} to our setting.
\begin{proof}[Proof of Proposition~\ref{prop:lower}]
  Denote by $\Delta X$ the maximal difference between the labels of two vertices of $F_p$, that is 
  \[
  \Delta X = \max\{X(v),\,v\in V(F_p)\}-\min\{X(v),\,v\in V(F_p)\}.
  \] 
  Thanks to Proposition~\ref{prop:contour-label}, $\Delta X\,p^{-1/2}$ converges in distribution to an almost surely finite random variable. It follows that there exists $y>0$ such that $\p{\Delta X\geq yp^{1/2}}<\varepsilon$. Fix such a $y$ for the rest of the proof. Now, denote by $\mathring{\Map}_p$ the subgraph of $\Map_p^*$ induced by its inner vertices. Let $B$ be the event that $\mathring{\Map}_p$ contains a cycle $C$ of length at most $y/\varepsilon$ and such that for $\delta\in\{\Delta,\inn\}$ we have: 
  \[
  \max_{v\in V^\delta[C]}X(v)-\min_{v\in V^\delta[C]}X(v)\geq 2\varepsilon p^{1/2}-4.
  \]
  Let now $u$ be a uniformly random vertex of $\Map_p^*$. Let $E$ be the event that there is no modified leftmost path starting at $u$ that contains a vertex incident to the root face of $\Map_p^*$.
  By Lemma~\ref{lem:nottouching}, $\lim_{p\rta\infty}\p{E}=1$, so it is enough to prove the result on $E$, and we assume that $E$ holds until the end of the proof.
  
  Suppose now  that $d_{\mathring{\Map}_p}(u,v^\star)<X(u)-15\cdot8\varepsilon p^{1/2}$. By Claim~\ref{claim:shortcuts} it implies that, for $p$ large enough, there exists $v\in \mathring{\Map}_p$ a neighbor of $u$ and $Q$ a geodesic path in $\mathring{\Map}_p$ from $u$ to $v^\star$ such that $\sigma(Q,uv)> 7\varepsilon p^{1/2}$. 
  
  Next, if $B$ does not occur, then by Proposition~\ref{prop:wrappings}, $\dfrac{6|Q|}{\sigma(Q,uv)-2}+3\geq \dfrac{y}{\varepsilon}$. 
  Since $E$ holds, we know that if $P$ is a modified leftmost path starting from $u$, then $P$ contains only vertices of $\mathring{\Map}_p$. Since $Q$ is a geodesic path in $\mathring{\Map}_p$, we get $|Q|\leq|P|$ for all modified leftmost paths $P$ starting from $u$. Using this and Proposition~\ref{prop:LMPreaches}, we get $|Q|\leq X(u)-X(v^\star)+3\leq \Delta X+3$. Combining this with $\dfrac{6|Q|}{\sigma(Q,uv)-2}+3\geq \dfrac{y}{\varepsilon}$, we get $\Delta X\geq yp^{1/2}$. Therefore, either $B$ occurs or $\Delta X\geq yp^{1/2}$. Since $\p{\Delta X\geq yp^{1/2}}<\varepsilon$, in order to conclude the proof it is sufficient to establish that: 
  \[
  \p{B} < 2\varepsilon.
  \]
  Suppose $B$ occurs, let $C$ be as in the definition of $B$ and let respectively $F_p^\Delta$ and $F_p^\inn$ be the subgraphs of $F_p^*$ induced by $V^\Delta[C]$ and $V^\inn[C]$. For $\delta\in\{\inn,\Delta\}$, each connected component of $F_p^\delta$ is a forest contained in $V^\delta[C]$. Furthermore all the edges of $\Map_p^*$ that do not belong to $F_p^\Delta\cup F_p^\inn$, belong to $E(\mathring{\Map}_p)$. Since Claim~\ref{claim:labelsNeighbours} gives that for any $\{u,v\}\in E(\mathring{\Map}_p)$ we have $|X(u)-X(v)|\leq 3$, it follows that, for $\delta\in \{\inn,\Delta\}$, if $V^\delta[C]$ contains $k$ components of $G$, then at least one such component $T_h$ must have: 
  \[
  \max_{u\in V(T_h)}X(u)-\min_{u\in V(T_h)}X(u)>\frac{2\varepsilon p^{1/2}-4}{k}-3.
  \]
  Since $|C|\leq y/\varepsilon$, $V^\delta[C]$ has at most $y/\varepsilon$ components of $G$, so:
  \[
  \max_{u\in V(T_h)}X(u)-\min_{u\in V(T_h)}X(u)>\frac{2\varepsilon^2 p^{1/2}-4\varepsilon}{y}-3.
  \]
  
  Using again that labels of adjacent vertices in $\mathring{\Map}_p$ differ by at most three and that $|C|\leq y/\varepsilon$, we get that for $\delta\in\{\inn,\Delta\}$, there is $v_\delta\in V^\delta[C]$ such that: 
  \[
  \min_{v\in V(C)}|X(v_\delta)-X(v)|\geq \frac{\varepsilon^2 p^{1/2}-2\varepsilon}{y}-\frac{3\varepsilon}{2}-\frac{3y}{2\varepsilon}.
  \]
  Now, let $j_\delta=\inf\{0\leq i \leq 2|F_p^*|-p\,:\,\beta(i)=v_\delta\}$. Fix $\gamma \in (0,\varepsilon^2/y)$ and let $\alpha$ be such that the conclusion of Lemma~\ref{lem:regContour} holds. For $p$ large enough such that $(\varepsilon^2 p^{1/2}-2\varepsilon)/y-3\varepsilon/2-3y/2\varepsilon>\gamma p^{1/2}$, for $\delta\in\{\inn,\Delta\}$, we also have $N(j_\delta,\gamma p^{1/2})\subset V^\delta[C]$, and it follows that for $p$ sufficiently large: 
  \begin{equation*}
    \p{B}< \varepsilon + \P\Big(\exists C\text{ cycle in }\mathring{\Map}_p\,:\,|C|\leq \frac{y}{\varepsilon},\min\big(V^\inn[C],V^\Delta[C]\big)\geq \alpha p^{2}\Big),
  \end{equation*}
  which concludes the proof in view of Lemma~\ref{lem:sepCircles}.
\end{proof}
\begin{proof}[Proof of Lemma~\ref{lem:sepCircles}]
  The number $t_{n,p}$ of simple triangulations of the $p$-gon with $n$ inner vertices has been computed in \cite{BrownTriangulation}, and has the asymptotic form $t_{n,p} \underset{n\rightarrow\infty}{\sim} A_p \rho^{-n} n^{-5/2}$, where $A_p$ is an explicit constant and $\rho=27/256$. Moreover, this asymptotic behavior can be refined when $p$ grows to infinity together with $n$ in the following way. There exists $c_1,c_2>0$ such that for all $n$ and $p$ with $n\geq \alpha p^2$, 
  \begin{equation}\label{eq:devtnp}
  c_1\gamma^p p^{1/2}\rho^{-n} n^{-5/2}\leq t_{n,p}\leq
  c_2\gamma^p p^{1/2}\rho^{-n} n^{-5/2}, \text{ with }\gamma=64/9.
  \end{equation}
  
  For $\ell_1,\ell_2 \geq 3$, a simple triangulation of a cylinder with perimeters $\ell_1$ and $\ell_2$ is a rooted simple planar map with a root face of degree $\ell_1$ and a marked face of degree $\ell_2$ both with simple boundaries, such that no vertices are incident to both the root and the marked faces and such that all the other faces are triangles. A vertex incident neither to the root nor the marked face is called an inner vertex. Let $t_{n,\ell_1,\ell_2}$ be the number of simple triangulations of a cylinder with perimeters $\ell_1$ and $\ell_2$, with $n$ inner vertices. By adding three vertices in the marked face of degree $\ell_2$ and triangulating it, we produce a simple triangulation of a $\ell_1$-gon with a marked triangular face, and with $n+\ell_2+3$ inner vertices. Reciprocally, given an element of $\mtrign[\ell_1]{n}$ (recall the definition of $\Delta_{p,n}^\Delta$ from the beginning of Section \ref{sub:orientations}), we can construct a triangulation of the cylinder with perimeters $\ell_1$ and $\ell_2$ and $n$ inner vertices by adding some triangles in the marked triangular face, such that the three vertices on the marked face are all inner vertices of the cylinder. Observing that both these constructions can be made injective, we get that
  \[
  t^\Delta_{n,\ell_1}\leq t_{n,\ell_1,\ell_2}\leq t^\Delta_{n+\ell_2+3,\ell_1},
  \]
  where $t_{n,p}^\Delta=|\Delta_{p,n}^\Delta|$.
  
  Since a triangulation of the $p$-gon with $n$ inner vertices has $2n+p-2$ triangular faces, the bounds given above combined with the one given in \eqref{eq:devtnp} translate immediately to the following. For any $K\in \mathbb{N}$, there exists $c_1,c_2>0$ such that for all $n$, $\ell_1$ and $\ell_2$ with $n\geq \alpha \ell_1^2$ and $\ell_2\leq K$, then
  \begin{equation*}
    c_1\gamma^{\ell_1} {\ell_1}^{1/2}\rho^{-n} n^{-3/2}\leq t^\Delta_{n,\ell_1}\leq
    c_2\gamma^{\ell_1} {\ell_1}^{1/2}\rho^{-n} n^{-3/2}
  \end{equation*}
  and
  \begin{equation}\label{eq:devtnl}
  c_1\gamma^{\ell_1} {\ell_1}^{1/2}\rho^{-n} n^{-3/2}\leq t_{n,\ell_1,\ell_2}\leq
  c_2\gamma^{\ell_1} {\ell_1}^{1/2}\rho^{-n} n^{-3/2}.
  \end{equation}
  
  To establish the bound on $\Bolm(p)(\Gamma_K(\alpha))$, let us first consider the case where $C$ separates the root face of $\Map_p^*$ from its marked triangular face. Then, the graph induced by $V^\Delta[C]$ is a triangulation of a $|C|$-gon with a marked triangular face and the graph induced by $V^\inn[C]$ is a triangulation of a cylinder with perimeters $p$ and $|C|$. Similarly, if $C$ does not separate the root face and the marked triangular face, then the graph induced by $V^\Delta[C]$ is a triangulation of a cylinder with perimeters $p$ and $|C|$ with an additional marked triangular face and the graph induced by $V^\inn[C]$ is a triangulation of a $|C|$-gon. 
  
  Consequently, there is an injection from the set of maps satisfying the event $\Gamma_K(\alpha)$ and pairs consisting of a cylinder and a triangulated polygon.  Recall that a triangulation of the $p$-gon with $n$ inner vertices has $2n+p-2$ triangular faces and that a triangulation of a $(\ell_1,\ell_2)$-cylinder with $n$ inner vertices has $2n+\ell_1+\ell_2$ triangular faces. Hence, for $n \geq \alpha p^2$ and sufficiently large $p$, 
  \begin{align*}
    \Bolm(p)&\left(\Gamma_K(\alpha)\,\big|\,|\Map_p^*|=n\right)\\
    \leq&\, \frac{1}{t^\Delta_{n,p}}\sum_{k=3}^{K}\sum_{m=\lfloor \alpha p^2\rfloor}^{n-\lfloor \alpha p^2\rfloor}\Big(t_{m,k}t_{n-m,p,k}\big((2m+k-2)+(2(n-m)+k+p)\big)\Big)\\
    \leq&\,2 \frac{p^2}{t^\Delta_{n,p}}\sum_{k=3}^{K}\int_{\alpha}^{\frac{n}{p^2}-\alpha}\Big(t_{\lfloor up^2\rfloor,k}t_{\lfloor n-up^2\rfloor,p,k}\big((2\lfloor up^2\rfloor+k-2)+(2\lfloor n-up^2\rfloor+k+p)\big)\Big)du\\
    \leq&\,p^{-3}n\sum_{k=3}^KA_k\int_{\alpha}^{\frac{n}{p^2}-\alpha}u^{-5/2}(1-up^2/n)^{-3/2}du\\
    =& \,p^{-3}n \sum_{k=3}^K A_k\Big(\frac{n}{p^2}\Big)^{-3/2}\int_{p^2\alpha/n}^{1-p^2\alpha/n} v^{-5/2}(1-v)^{-3/2}dv\\
    \leq&\,A_{K,\alpha}p^{-3}n,
  \end{align*}
  where $A_k$ is a constant depending only on $k$ and $A_{K,\alpha}$ is a constant depending only on $K$ and $\alpha$.
  
  Since $\Bolm(p)\left(\Gamma_K(\alpha)\,\big|\,|\Map_p^*|=n\right)=0$ for $n\leq \alpha p^2$, we get that
  \[
  \Bolm(p)\big(\Gamma_K(\alpha)\big)\leq A_{K,\alpha}p^{-3}\E{|\Map_p^*|}.
  \]
  By \cite[Proposition~6.4]{Angel-Growth}, the right-hand term converges to $0$ when $p$ tends to infinity, which yields the result.
\end{proof}

\section{Coupling between different types of triangulations}\label{sec:kernel}

In this section we present a folklore coupling of $\Bol_{i}(p)$ for $i\in \{\op{I},\op{II},\op{III}\}$. 
The idea is to use the so-called core construction due to Tutte \cite{Tutte}. Namely, 
if we start from a sample from $\Bol_{\op{I}}(p)$, 
then its so-called 2-connected core has the law of a sample from $\Bol_{\op{II}}(p)$ (see Proposition~\ref{prop:ratio}). 
Furthermore, its so-called maximal simple core has the law of a sample from $\Bol_{\op{III}} (p)$ (see Proposition~\ref{prop:ratio2}).
In this coupling, it is not hard to see that both the metric and the measure are unchanged in the proper scaling limit. We will explain this point in Section~\ref{subsec:compare}. 
We point out that this section is close in spirit to \cite{ABW-Core}, which studies the case of quadrangulations of a fixed size without boundary, and we take some technical input from that paper. However, since we focus on the Boltzmann case, we do not need the very fine enumerative asymptotic analysis used in \cite{ABW-Core}. 

\subsection{A coupling between type II and III triangulations}\label{subsec:coupling1}
For integers $n\ge p\ge 3$ and $i\in\{ \op{I},\op{II},\op{III} \}$, recall that $\Delta_i(p,n) $ is the set of type  $i$ triangulations of the $p$-gon of size $n$ and $\Delta_i(p)=\bigcup_{n\ge p}\Delta_i(p,n)$. For technical reasons we extend the definition of   $\Delta_{\op{II}}(p,n) $ and $\Delta_{\op{II}}(p) $ to the case $p=2$.
For $n>2$, let $\Delta_{\op{II}}(2,n) $ be the set of maps of size $n$ such that all its non-root faces have degree 3 and its root face has degree $2$, and, moreover, there are no self-loops.
We adopt the convention that $\Delta_{\op{II}}(2,2)$ is the set with a single element -- the map with two edges sharing the same endpoints. We still set $\Delta_{\op{II}}(2)=\bigcup_{n\ge 2}\Delta_{\op{II} }(2,n)$.
Recall the notation $\IF(M)$ in Section~\ref{subsec:notation} for the set of inner faces of $M$.  We make the convention that $\IF(M)=\emptyset$ if $M\in \Delta(2,2)$.

Following \cite[Section~2]{UIPT}, let $Z_p(t)\defeq\sum_n|\Delta_{\op{II}}(p,n)|t^n$ for each $p\ge 2$. 
Then  $\rho_{\op{II}}=2/27$
is the convergence radius of $Z_p(t)$. 
Let $\Bol_{\op{II}}(p)$ be the probability   on $\Delta_{\op{II}}(p)$ 
where each element in $\Delta_{\op{II}}(p,n)$ is assigned  probability $\rho_{\op{II}}^n/Z_p(\rho_{\op{II}})$. Then for $p\ge 3$ this measure coincides with  $\Bol_{\op{II}}(p)$ defined in Theorem~\ref{thm:rigor}.

Fix $p\ge 3$. 
We now describe the simple  core construction for maps in $\Delta_{\op{II}}(p)$ following the presentation of \cite[Section~1.1]{ABW-Core} for quadrangulations 
(see also the general framework presented in~\cite[Section~5.1]{Airy}). An example is given on Figure~\ref{fig:core}. 
Let $S\in \Delta_{\op{III}}(p)$ be a simple triangulation with root edge $uv$. 
List the vertices of $S$ in a breadth-first order
as $u_1,\dots, u_n$, i.e., the distance to $u$ is non-decreasing. Then list the edges of $S$ as $uv= e_1,\dots, e_k$, and orient the edges ``away from'' $uv$ so that the tail precedes
the head in the breadth-first order.\footnote{The breath-first ordering is introduced only for the purpose of  assigning orientations to the edges of $S$. 
  We assume the breath-first ordering is chosen in a way only depending on $S$.}
For  $1\le i\le k$, let $M_i\in \Delta_{\operatorname{II}}(2)$ be sampled from $\Bol_{\op{II}}(2)$, independently for each $i$.
For each  $1\le i\le k$, identify the edge $e_i$ with the root edge of $M_i$. This way we attach $M_i$ to $e_i$ so that it is on the left side of $e_i$.
When $M_i\in\Delta_{\op{II}}(2,2)$, we collapse its two edges into a single one.
This operation forms  a map $M\in \Delta_{\op{II}}(p)$ with
\begin{equation}\label{eq:decomph}
|\IF(M)|=|\IF(S)|+\sum_{i=1}^{k} |\IF(M_i)|.
\end{equation}
Moreover, this construction is a bijection in the sense that any $M\in \Delta_{\op{II} }(p)$ can be obtained this way, and given $M$, the maps $(S,M_1, \cdots, M_k)$ are uniquely determined.
In this bijection, 
we call $S$ the \notion{simple core} of $M$ and denote it by $\simple(M)$.  We call  $(S,M_1, \cdots, M_k)$ the \notion{simple core decomposition} of $M$.  
\begin{figure}
\centering
\includegraphics[width=0.9\linewidth,page=2]{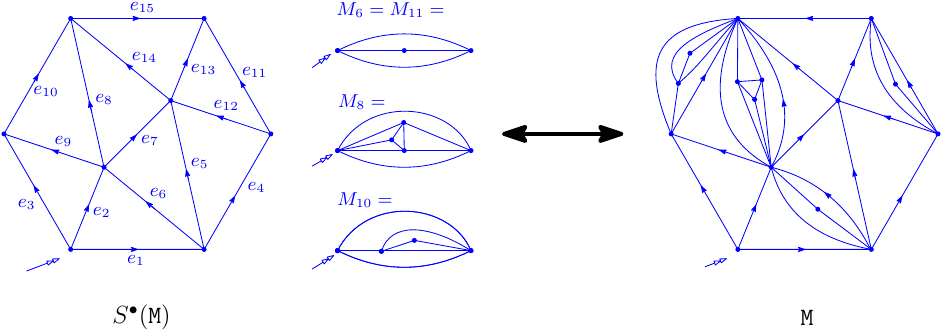}
\caption{\label{fig:core} An example of the simple core construction and decomposition. The maps $M_1$, $M_2$, $M_3$, $M_4$, $M_5$, $M_7$, $M_9$, $M_{12}$, $M_{13}$, $M_{14}$ and $M_{15}$ are all equal to the unique element of $\Delta_{\operatorname{II}}(2,2)$.\\ Note that the the edges are oriented according to a breadth-first search as described in the text. Their orientation has nothing to do with the minimal 3-orientation.}
\end{figure}

The main result of this subsection is the following.
\begin{proposition}\label{prop:ratio}
  For $p\ge 3$, let $\Map_{p}$ be a sample of $\Bol_{\op{II}}(p)$ and let $S_p=\simple(\Map_{p})$. Then 
  $\lim_{p\to \infty} |E(\Map_p)|/|E(S_p)|=2$ in probability. 
  Moreover, the law of $S_p$ is $\Bol_{\op{III}}(p)$.
\end{proposition}
The proof of Proposition~\ref{prop:ratio} relies on the following fact, which is immediate from the core construction above and the definition of $\Bol_{\op{II}}(p)$ for $p\ge 2$.
\begin{lemma}\label{lem:Bol}
  In the setting of Proposition~\ref{prop:ratio}, let $k_p=|E(S_p)|$ and let $(S_p,\Map^1,\cdots,\Map^{k_p})$ be the simple core decomposition of $\Map_p$.  Then, conditioning on $k_p$, $S_p$ is uniform among all type III triangulations of the $p$-gon with $k_p$ edges, and the maps $\{\Map^i\}_{1\le i\le k_p}$ are independent samples from $\Bol_{\op{II}}(2)$.
\end{lemma}
Another ingredient for Proposition~\ref{prop:ratio} is the following fact from \cite[Proposition~6.4]{Angel-Growth}.
\begin{lemma}\label{lem:gamma}
  $|V(\Map_p)|/p^2$  converge in law to the random variable with density proportional to $\1_{x>0}x^{-5/2}e^{-1/(3x)}$.
\end{lemma}
\begin{proof}[Proof of Proposition~\ref{prop:ratio}]
  Let $k_p=|E(S_p)|$. By  \eqref{eq:decomph}, we have
  \begin{equation}\label{eq:decomp}
  |\IF(\Map_p) |=|\IF(S_p)| + \sum_{i=1}^{k_p} |\IF(\Map^i)|.
  \end{equation}  
  
  Let $\Map_2$ be sampled from  $\Bol_{\op{II}}(2)$. From the definition of $Z_2(t)$,  we have
  \(\E{|V(\Map_2)|}=\rho_{\op {II}} Z'_2(\rho_{\op {II}})/Z_2(\rho_{\op {II}})\).   
  By \cite[Proposition~2.4]{UIPT} (see also \cite[Section 2.9.5]{Goulden-Jackson}), if $t=\theta(1-2\theta)^2$ with $\theta\in(0,1/2)$, then
  \begin{equation}\label{eq:partition}
  Z_p(t) =\frac{(2p-4)!((1-6\theta)p+6\theta)}{p!(p-2)!}\theta^{p}(1-2\theta)^{2}.
  \end{equation}
  Since $t=\rho_{\op {II}}=2/27$ when $\theta=1/6$, by the chain rule we have 
    \begin{equation}\label{eq:Vnum}
    \E{|V(\Map_2)|}=\rho_{\op {II}} Z'_2(\rho_{\op {II}})/Z_2(\rho_{\op {II}})=7/3.
  \end{equation} 
  By Euler's formula, $|\IF(\Map_2)|=2|V(\Map_2)|-4$. (Recall  that by our convention  $|\IF(\Map_2)|=0$ if $\Map_2\in\Delta(2,2)$.)
  Now $\BB E\big[|\IF(\Map_2)|\big]=2/3$.
  
  By Lemma~\ref{lem:gamma}, $\lim_{p\to\infty} |\IF(\Map_p)|/p=\infty$. Therefore $\lim_{p\to\infty} |\IF(\Map_p)|/|E(\Map_p)|=2/3$ in probability, since $3|\IF(\Map_p)|=2|E(\Map_p)|-p$ by Euler's formula. Similarly, by Lemma~\ref{lem:CT} and Euler's formula, $\lim_{p\to\infty}|\IF(S_p)|/k_p=2/3$ and $\lim_{p\to \infty}k_p=\infty$ in probability.
  By the law of large numbers and Lemma~\ref{lem:Bol}, 
  \[
  \lim_{p\to\infty} k_p^{-1}\sum_{i=1}^{k_p} |\IF(\Map^i)|=\BB E\big[\IF(\Map_2)\big]=2/3\quad \textrm{in probability}.
  \]

  Therefore \eqref{eq:decomp} yields that  \(\lim_{p\to \infty} |E(\Map_p)|/k_p=3(2/3+2/3)/2=2\) in probability.
  
  For the last assertion, let $\mathcal{Z}_p(\alpha)=\sum_{T\in \Delta_{\op{III}}(p)} \alpha^{|T|}$. Then $\rho_{\op{III}}$ is the convergence radius of $\mathcal{Z}_p$.
  It is clear that there exists 
  $0<\alpha\le \rho_{\op{III}}$ such that $\P[S=T]=\mathcal{Z}_p^{-1}(\alpha) \alpha^{|V(T)|}$ for each $T\in \Delta_{\op{III}}(p)$.
  Moreover,   $\lim_{p\to 0}|\IV(S)|/p^2=0$ in probability if $\alpha<\rho_{\op{III}}$.
  By Lemma~\ref{lem:gamma}, we must have  $\alpha=\rho_{\op{III}}$. 
\end{proof}
\begin{remark}\label{rmk:ratio}
  The fluctuation of $k_p$ around its mean is expected to be the Airy distribution, following the technique of \cite{Airy}, but here we only need the first order asymptotics.
\end{remark}

\subsection{Comparison between type II and III triangulations}\label{subsec:compare}
We retain the setting in Proposition~\ref{prop:ratio} and Lemma~\ref{lem:Bol}, where $S_p$ is the simple core of $\Map_p$.
Let $\cM_p\in \BMS^{\GHPU}$ be defined as in Theorem~\ref{thm:rigor} in the case $i=\op{II}$. 
Let $\cS_p$ be the element in $\BMS^{\GHPU}$ obtained by rescaling $S_p$ as in the $i=\op{III}$ case of Theorem~\ref{thm:rigor}. 
Since $S_p$ is sampled from $\Bol_{\op{III} }(p)$, by Propositions~\ref{prop:contour-label} and~\ref{prop:dist2points},  the routine argument (see e.g.\ \cite{legall-tightness,ABA-Simple}) shows the following. 
\begin{proposition}\label{prop:tight}
  $\{ \cS_p\}_{p\ge 3}$ is tight in the  GHP topology.
\end{proposition}
The main result in this subsection is the following.
\begin{proposition}\label{prop:ghpu-close}\
  The GHP distance between $\cS_p$ and $\cM_p$ tends to $0$ in probability.
\end{proposition}
Our strategy is similar to \cite{ABW-Core}, where the case of fixed size quadrangulations is treated. 
The idea is that as $p$ gets large, the contribution of the $\Map^i$ are all negligible hence the scaling limit of $\Map_p$ coincides with $S_p$. 
More precisely, viewing $\cS_p$ as a subset of $\cM_p$, the Hausdorff distance
between $\cS_p$ and $\cM_p$ is dominated by  $\sqrt{3/2}p^{-1/2}\max_{1\le i\le k_p}\diam(\Map^i)$. Here, given a planar map $M$, $\diam(M)$ denotes its diameter under the graph distance. 
The following lemma controls the diameter of a planar map sampled from $\Bol_{\op{II}}(2)$.
\begin{lemma}\label{lem:diameter}
  Let $\Map_2$ be sampled from $\Bol_{\op{II}}(2)$. Then \(\lim_{x\to\infty} x^4\P[\diam(\Map_2)>x]=0\).
\end{lemma}
We postpone the proof of Lemma~\ref{lem:diameter} and proceed to the proof of Proposition~\ref{prop:ghpu-close}. The following lemma is needed.
\begin{lemma}\label{lem:max}
  Let $\tau$ be a  random positive integer with finite mean. Let $\{X_i\}_{i\ge 1}$ be a sequence of  random variables independent of $\tau$.  Then 
  \[
  \P\Big[\max_{1\le i\le \tau}X_i>x \Big] \le \E{\tau} \max_{i\ge 1}\P[X_i>x] \qquad \textrm{for each }x\in \R.
  \]
\end{lemma} 
\begin{proof}
  By the independence of $\tau$ and $\{X_i\}_{i\ge 1}$, along with a union bound, we have 
  \[
  \P\Big[\max_{1\le i\le \tau}X_i>x \Big]
  =
  \sum_{n=1}^\infty \P[\tau=n]\P\Big[\max_{1\le i\le n}X_i>x \Big]
  \le 
  \sum_{n=1}^\infty \P[\tau=n]n\max_{i\ge 1}\P[X_i>x ].
  \]
  Since $\E{\tau}=\sum_{n=1}^\infty \P[\tau=n]n$, we conclude the proof.
\end{proof}
\begin{proof}[Proof of Proposition~\ref{prop:ghpu-close}]
  By Lemma~\ref{lem:RN} and Proposition~\ref{prop:ratio}, there exists $C>0$ independent of $p$ such that 
  \[
  \E{k_p} = \E{|E(S_p)|}\le Cp^2 \qquad \textrm{for all } p\ge 3.
  \]
  By  Lemmas~\ref{lem:Bol} and~\ref{lem:max}, for each $\eps>0$,  
  \[
  \P\Big[ \max_{1\le i\le k_p}\diam(\Map^i) >\eps p^{1/2}\Big] \le \E{k_p}  \P[\diam(\Map_2)>\eps p^{1/2}].
  \]
  By Lemma~\ref{lem:diameter},
  \[
  \lim_{p\to \infty}p^{-1/2}\max_{1\le i\le k_p}\diam(\Map^i)=0\qquad \textrm{in probability}.
  \]  
  
  To prove Proposition~\ref{prop:ghpu-close},  it remains to  handle the Prokhorov distance between the measures.
  Let $\mu_p$ and $\mu'_p$ be the area measures of $\cM_p$ and $\cS_p$, respectively. It suffices to show that
  the Prokhorov distance between the two measures tends to $0$ in probability as $p\to\infty$.  
  We omit the details of the proof of this assertion as the same problem in the setting of quadrangulations of a sphere has been treated in detail in \cite[Sections~5 and~6]{ABW-Core} via a robust argument based on the concentration of exchangeable random variables. After straightforward adaptation this proves our case. 
\end{proof}

A similar result to Lemma~\ref{lem:diameter} is established in \cite{ABW-Core}. However, one key ingredient used there is an exponential tail estimate (\cite[Proposition~1.11]{ABW-Core}) for the diameter of uniform large quadrangulations with self-loops and multiple edges. 
This relies on the simple form of the Schaeffer bijection in the quadrangulation case.
We cannot find such an estimate in the triangulation case in the literature. Therefore we conclude this subsection with  a detailed proof of Lemma~\ref{lem:diameter} only using weaker estimates. 
We start with  another variant of Lemma~\ref{lem:max}.
\begin{lemma}\label{lem:max1}
  Let $\tau$ be a  random positive integer with finite mean. Let $\{X_i\}_{i\ge 1}$ be a sequence of identically  distributed positive random variables. Then for each $p>0$  there exists a constant $C_p>0$ such that 
  \begin{equation}\label{eq:max}
  \P\Big[\max_{1\le i\le \tau}X_i>x\Big] \le  x^{-p} +C_p \E{\tau} \P[X_1>x]\log x \qquad \textrm{for each }x>0.
  \end{equation}
\end{lemma}
\begin{proof}
  Let  $\{\tau_j\}_{j \ge 1}$ be a sequence of independent copies of $\tau$, which is also independent of $\tau$ and $\{X_i\}_{i\ge 1}$.
  Let $Y_j=\max_{1\le i\le \tau_j}X_i$ and  $\xi=\inf\{j\ge 1: \tau\le \tau_j\}$. Then \(\max_{1\le i\le\tau} X_i \le Y_{\xi}. \)
  By Lemma~\ref{lem:max}, $\P[Y_j>x]\le \E{\tau}\P[X_1>x]$ for each $j\in\N$.  For each $C>0$,
  \begin{equation}
  \P\Big[\max_{1\le i\le \tau}X_i>x\Big] \le \P[\xi>C\log x]+\P\Big[\max_{1\le j\le C\log x}Y_j>x\Big]. 
  \label{eq:max2}
  \end{equation}
  Since $\xi$ is a geometric random variable, choosing $C$ sufficiently large (depending only on $p$),  the first term on the right side of \eqref{eq:max2} is smaller than $x^{-p}$. We bound the second term on the right side of \eqref{eq:max2} by a second application of Lemma \ref{lem:max}.
\end{proof}
Note that in Lemma~\ref{lem:max1} we do not require independence of $\tau$ and $\{X_i\}_{i\ge 1}$, in contrast to Lemma~\ref{lem:max}.

\begin{figure}[t]
\centering
\includegraphics[width=0.9\linewidth]{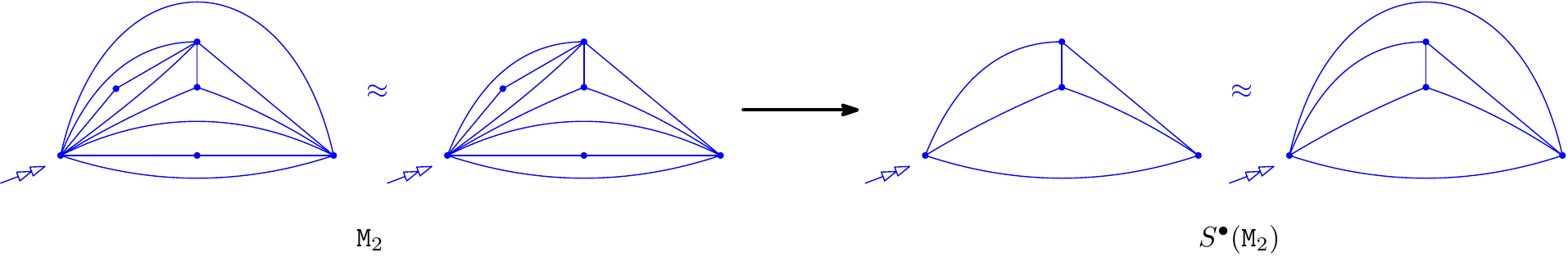}
\caption{\label{fig:core2}The construction of the simple core of an element of $\Delta_{\op{II}}(2)$. To obtain a non-trivial simple core, we first identify an element of $\Delta_{\op{II}}(2)$ with a triangulation of type II of the sphere.}
\end{figure}

By our definition of $\Delta_{\op{II}}(2)$, each $M\in \Delta_{\op{II}}(2)\setminus \Delta_{\op{II}}(2,2)$ can be identified with a type II triangulation (without boundary) by gluing the two edges on the root face of $M$, see Figure~\ref{fig:core2}. 
From now on we take this perspective and  view $\Delta_{\op{II}}(2)$ as the set of type II triangulations.  Under this identification, the measure $\Bol_{\op{II}}(2)$ coincides with the classical definition of the critical Boltzmann type II triangulation in the literature. We can further identify $M$ with an element of $\Delta_{\op{II}}(3)$ (i.e., view it as a triangulation with boundary of length 3 rather than a triangulation) and perform the simple core decomposition in the previous section to get $(S,M_1,\cdots, M_k)$. We still call $S$ the simple core of $M$ and denote it by $\op S^{\bullet}(M)$. 

Now let $\Map_2$ be sampled from $\Bol_{\op{II}}(2)$ viewed as a type II triangulation (without boundary). If $\Map_2\notin \Delta_{\op{II} }(2,2)$, let $S_2\defeq \simple(\Map_2)$ be its simple core. Let $k_2\defeq |E(S_2)|$. Let $(S_2,\Map^1,\cdots, \Map^{k_2})$  be the simple core decomposition of  $\Map_2$. For $1\le i\le k_2$, if  $\Map^i\notin \Delta_{\op{II}}(2,2)$ we perform the simple core decomposition for $\Map^i$ and iterate. This procedure terminates in finitely many steps and corresponds to a Galton-Watson tree, which we denote by $\cT$. The offspring distribution of $\cT$ is given by the law of $k_2\1_{\Map_2\notin \Delta_{\op{II}}(2,2)}$. 
For each vertex $v$ on the tree $\cT$, let $\Map_v$ be the corresponding triangulation in $\Delta_{\op{II} }(2)$. 
Let $D\defeq \max_{v\in V(\cT)} \diam(\simple(\Map_v))$. 
By the triangle inequality for the graph distance of $\Map_2$, we see that $\diam(\Map_2)\le HD$, where $H$ is the height of the Galton-Watson tree. 

We claim that  $\cT$ is subcritical  so that its height $H$ has an exponential tail. Namely, we claim that
\begin{equation}\label{eq:subcritical}
c\defeq\E{ k_2\1_{\Map_2\notin   \Delta_{\op{II}}(2,2)} }<1.
\end{equation}
To prove \eqref{eq:subcritical},
if $\Map_2\notin \Delta_{\op{II} } (2,2)$, since identifying $\Map_2$ as an element in $\Delta_{\op{II} }(3)$ will decrease the number of faces by 1, by~\eqref{eq:decomp} we have:
\[
|\IF(\Map_2)|-1=|\IF(S_2)|+\sum_{i=1}^{k_2} |\IF(\Map^i)|.
\]
We emphasize that here $|\IF(\Map_2)|$ (resp.\ $|\IF(\Map^i)|$) is the number of inner faces of $\Map_2$ (resp.\ $\Map^i$) viewed as an element of $\Delta_{\op{II} }(2)$ and $|\IF(S_2)|$ is the number of inner faces of $S_2$ viewed as an element of $\Delta_{\op{III} }(3)$.
Now by conditioning on $k_2$, we get 
\[
\BB E \big[(|\IF(\Map_2)| -1)\1_{\Map_2\notin   \Delta_{\op{II}}(2,2)}\big] =\BB E\big[|\IF(S_2)| \1_{\Map_2\notin   \Delta_{\op{II}}(2,2)}  \big]+ \BB E\big[ k_2\1_{\Map_2\notin   \Delta_{\op{II}}(2,2)} \big]   \BB E \big[|\IF(\Map_2)|\big].
\]
In particular, $c \BB E \big[|\IF(\Map_2)|\big]< \BB E \big[|\IF(\Map_2)|  \big] $ hence $c<1$. 

Since $\E{|V(\cT)|}<\infty$,   Lemma~\ref{lem:max1} implies that for $x$ large enough,
\begin{equation}\label{eq:tri}
\P[D>x]\lesssim x^{-100}+\P[\diam(\simple(\Map_2)) >x]\log x.
\end{equation} 
Since $\diam(\Map_2)\le HD$, $H$ has an exponential tail, and \eqref{eq:tri} holds,  Lemma~\ref{lem:diameter} follows from the lemma below and the fact that $14/3 > 4$.
\begin{lemma}\label{lem:diameter1}
  For $S_2=\simple(\Map_2)$, we have \(\lim_{n\to\infty} n^{14/3}\P[\diam(S_2)>n]=0\).
\end{lemma}
\begin{proof}
  Let $\{u_i\}$ be a sequence of random integers in $[|V(\Map_2)|]$ such that conditioning on $\Map_2$, they are independently and uniformly distributed. Fix an ordering $\{v_i\}_{1\le i\le |\IV(\Map_2)|}$  of $V(\Map_2)$ so that $v_{u_i}$ is a uniformly sampled vertex given $\Map_2$.
  Let $D_i$ be the distance between the root vertex and $v_{u_i}$.
  
  By the union bound, 
  \[
  \P[\diam(S_2)>n^{0.3}] \le \P[|V(S_2)|\geq n] +  \P[\diam(S_2)>n^{0.3}, \; |V(S_2)|<n]
  \]
  It is clear that  for each $n$ such that $\P[|V(S_2)=n|]>0$, the conditional law of $S_2$ given $|V(S_2)|=n$ is uniform among type III triangulations with $n$ vertices.
  By  \cite[Corollary~6.7 and~7.5]{ABA-Simple}, for each $\epsilon>0$, we have that  $\P[D_1>n^{1/4+\epsilon}\mid |V(S_2)|=n]$ decays faster than any power of $n$.
  Therefore  \(\P[\diam(S_2)>n^{0.3}, \; |V(S_2)|<n]\) also decays faster than any power of $n$.
  Since $ \P[|V(S_2)|>n]  \lesssim n^{-3/2}$, e,g., by \eqref{eq:devtnp},  we have \(\lim_{n\to\infty} n^{1.4}\P[\diam(S_2)>n^{0.3}]=0\).
\end{proof}

\subsection{A coupling between type I and II triangulations}\label{subsec:coupling2}
In this section we describe the coupling between $\Bol_{\op{I} }(p)$ and $\Bol_{\op{II}}(p)$ which allows to transfer the result obtained in Theorem~\ref{thm:rigor} for type $\op{II}$ to type I.
Suppose $M\in \Delta_{\op{I}}(p)$  with $p\ge3$. If $e\in E(M)$ is a self-loop, then $e$ must be an inner edge of $M$ since $M$ has simple boundary.
We may define a partial order $\prec$ on self-loops such that two self-loops $e,\wt e$ satisfy $\wt e\prec e$ if and only if $\wt e$ is separated from $\partial M$ by $e$.  
Suppose a self-loop $e$ is maximal in the partial order $\prec$.  Let $v_e$ be the vertex incident to $e$. 
The self-loop $e$ separates $M$ into two components.
The one that does not contain $\partial M$ together with $e$  forms a triangulation of a $1$-gon, which  we denote by $M_e$. There exists a unique $v^{\op o}_e\neq v_e$ in the other component such that $v^{\op o}_e$ and $e$ are incident to a common triangle.
Let $t^{\op o}_e$ be the triangle incident to both $e$ and $v^{\op o}_e$.  
There must be two edges $e'\neq e''$ linking $v_e$ and $v^{\op o}_e$ that are incident to $t^{\op o}_e$. 
We assume that $e',e,e''$ are counterclockwise arranged around $v_e$. 
Let $M^{\op{II}}$ be the map obtained from $M$ by removing the map of the 2-gon with boundary $\{e',e''\}$  (the one containing $M_e$) 
and then mering $e'$ and $e''$, where $e$ ranges over all maximal self-loops with respect to $\prec$. Then $M^{\op{II}}$ is a type II triangulation with a simple boundary, which is called the \emph{2-connected core} of $M$.
\begin{proposition}\label{prop:ratio2}
  For $p\ge 3$, let $\Map_{p}$ be a sample of $\Bol_{\op{I}}(p)$ and let $\Map^{\op{II}}_p$ be its 2-connected core. 
  Then  $\lim_{p\to \infty} |E(\Map_p)|/|E(\Map^{\op{II}}_p)|=2$ in probability. 
  Moreover, the law of $\Map^{\op{II}}_p$ is $\Bol_{\op{II}}(p)$.
\end{proposition}
We have the following analogue of Proposition~\ref{prop:ghpu-close}.
\begin{proposition}\label{prop:ghpu-close2}
  Let $\Map_{p}$  and $\Map^{\op{II}}_p$ be as Proposition~\ref{prop:ratio2}. 
  Let $\cM_{p}$  and $\cM^{\op{II}}_p$ be the rescaled spaces as defined in Theorem~\ref{thm:rigor}. Then the GHP distance between $\cS_p$ and $\cM_p$ tends to $0$ in probability.
\end{proposition}
We omit the proofs of Propositions~\ref{prop:ratio2} and~\ref{prop:ghpu-close2}, as they are very similar to the proofs of Propositions~\ref{prop:ratio} 
and~\ref{prop:ghpu-close}.

\section{Convergence of type III triangulations of polygons}\label{sec:BD}
Given a sample of  $\Map_p^*=(\Map_p,f^*)$   from $\Bolm(p)$ (see Proposition~\ref{prop:Bol}),
put mass $p^{-1}$ on each boundary vertex, mass $3p^{-2}$ on each inner vertex, and length $(\frac{3}{2p})^{1/2}$ to each edge.  As explained at the end of Section~\ref{subsec:GHPU}, $\Map_p$ induces a random variable $\cM_p\defeq\left( \Map_p , d_p , \mu_p , \bdy_p \right)$ belonging to the space  $\BMS^\GHPU$.
The main result of this section is the following.
\begin{theorem}\label{thm:Bol-conv-III}
  In the setting right above, $\cM_p$  converges in law  to   the pointed free Brownian disk $\BD^*_{1}$
  (see Section~\ref{subsec:BD}) in the GHPU topology. 
\end{theorem}
Given Theorem~\ref{thm:Bol-conv-III}, the original type III  case  of Theorem~\ref{thm:rigor} immediately follows by
reweighting the probability measure by the inverse of the number of inner faces of $\cM_p$. 
The Gromov-Hausdorff-Prokhorov (GHP)  convergence in the type I and II cases of Theorem~\ref{thm:rigor}  follows from Propositions~\ref{prop:ghpu-close} and~\ref{prop:ghpu-close2}.
As we will explain in Section~\ref{subsec:GHPU-conv}, in the GHPU convergence of Theorem~\ref{thm:Bol-conv-III}, the uniform part is rather straightforward once the convergence is established under the GHP metric.

Given Propositions~\ref{prop:contour-label}, \ref{prop:dist2points}, and~\ref{prop:lower}, 
for readers who are familiar with how a proof of GHP convergence to the Brownian disk/map goes (see e.g.\ \cite[Section 8]{BeMi15}), 
the GHP convergence part of Theorem~\ref{thm:Bol-conv-III} would be immediate if in Proposition~\ref{prop:lower} we had the actual graph distance instead of the one obtained from paths visiting only inner vertices. 
However,  it is a priori possible that adding the vertices incident to the root face dramatically shortens the distance between vertices. 
To resolve this issue,  we first use the results in the previous sections to prove the following.
\begin{proposition}\label{prop:random-bdy}
  There exists a sequence of random integers $\{\ell_p\}_{p\ge 3}$ in $[3,\infty)$ such that $\lim_{p\to\infty}\ell_p/p=1$ in probability, and, moreover, for $\{\Map_{p}^*\}_{p\ge 3 }$ samples from $\Bolm(p)$ independent of $\{\ell_p\}_{p\ge 3}$, $\{\cM_{\ell_p}\}_{p\geq 3}$ converges in law  to the pointed free Brownian disk $\BD^*_{1}$ in the GHP topology.
\end{proposition}
In Section~\ref{subsec:random-bdy}  we prove Proposition~\ref{prop:random-bdy}.
In  Section~\ref{subsub:couple} we will use Proposition~\ref{prop:Bol} to couple samples of 
$\Map_{\ell_p}$ in Proposition~\ref{prop:random-bdy} and $\Map_p$ such that they are close in the GHP distance.
This will give the GHP convergence of $\cM_p$.

\subsection{Convergence of type III triangulation with  random perimeters}\label{subsec:random-bdy}
To prove Proposition~\ref{prop:random-bdy}, we introduce the following setup.
Let $(\Omega,\P,\cF)$ be a probability space that contains the random variables
$\{(\Map_p,f^*)\}_{p\ge 3}$ in Theorem~\ref{thm:Bol-conv-III}.
Let $\Q_p$ be the measure obtained by reweighting $\P$ by the inverse of the number of inner faces of $\Map_p$, so that under $\Q_p$ 
the law of $\Map_p$ is $\Bol_{\op{III}}(p)$. In light of Proposition~\ref{prop:ratio},
we further require that the probability space $(\Omega,\P,\cF)$ contains random variables $\{\Map^{\op {II}}_p\}_{p\ge 3}$ such that  
the $\Q_p$-law of $\Map^{\op{II}}_p$ is $\Bol_{\op{II}}(p)$, and $\Map_p$ is the simple core of $\Map^{\op{II}}_p$. For 
$e\in E(\Map_p)$, let $M_e$ be the type II triangulation of a $2$-gon attached to $e$ in the simple core decomposition of $\Map^{\op{II}}_p$ (see Section~\ref{subsec:coupling1}).

Let $\mathring{\Map}_p$  be the maximal subgraph of $\Map_p$  whose vertex set consists of inner vertices of $\Map_p$.
Let $\mathring{\Map}^{\op{II}}_p$ be defined in the same way as $\mathring{\Map}_p$   with $\Map_p$ replaced by $\Map^{\op{II} }_p$. Then both  $\mathring \Map^{\op{II} }_p$ and $\mathring{\Map}_p$ could have several connected components,  each of which could have several 2-connected components.
If $M$ is a 2-connected component of $\mathring{\Map}_p$, then $M$ together with $\{M_e\}_{e\in E(M)}$ forms a 2-connected component of $\mathring \Map^{\op{II}}_p$.
A 2-connected component of $\mathring{\Map}^{\op{II} }_p$ can either be obtained in this way, or is a subgraph  of $M_e$ for some $e\in E(\Map_p)\setminus E(\mathring \Map_p)$. {}
We leave to the reader to check this fact, see also Figure~\ref{subfig:b1}.
\begin{figure}
\centering
\subfigure[An element  $\Map^{\op{II}}_9$ of $\Delta_{\op{II}}(9)$.]{\includegraphics[page=1,width=0.4\linewidth]{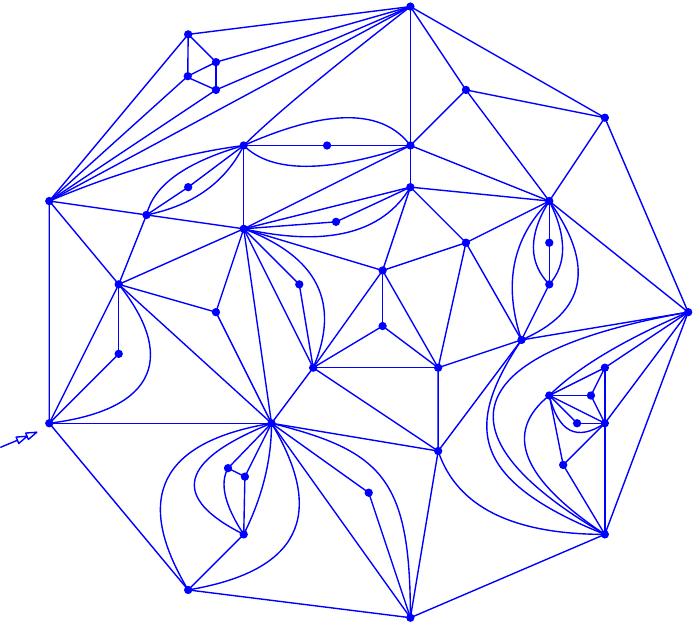}}
\qquad\qquad\qquad
\subfigure[\label{subfig:b1} The construction of $\cB_1$ and of $\ol\Map^{\op{II}}_9$]{\includegraphics[page=3,width=0.4\linewidth]{images/B1}}
\caption{\label{fig:constB1}An example of the construction of $\cB_1$ and of $\ol\Map^{\op{II}}$ on an element  of $\Delta_{\op{II}}(9)$. Edges of $\cB_1$ are dotted. Edges and vertices of $\ol\Map^{\op{II}}_9$ are fat and red.}
\end{figure}

The reason we consider $\mathring \Map^{\op{II}}_p$ is the following Markov property that is \emph{not} true for $\mathring \Map_p$, and which is exemplified on Figure~\ref{fig:constB1}.
\begin{lemma}\label{lem:Markov}
  Let $\partial \mathring\Map^{\op {II}}_p$ be the union of the boundaries of all the connected components of $\mathring\Map^{\op {II}}_p$.
  Let $\cB_1$ be the submap of $ \Map^{\op {II}}_p$ such that $V(\cB_1)=V(\partial \mathring\Map^{\op {II}}_p)\cup V(\partial \Map^{\op {II}}_p)$  and 
  $E(\cB_1)=E(\partial \mathring \Map^{\op{II}}_p)\cup (E(\Map^{\op{II} }_p)\setminus E(\mathring \Map^{\op{II}}_p))$.
  Conditioning on $\cB_1$,
  assign each $2$-connected component of $\mathring{\Map}^{\op{II} }_p$ a boundary root edge in a way only depending on $\cB_1$.
  Then the conditional $\Q_p$-law of these $2$-connected components 
  are independent Boltzmann type II triangulations with prescribed perimeters.
\end{lemma}

Geometrically, $\cB_1$ is the 1-neighborhood  of $\partial \Map^{\op{II}}_p$ under the graph distance, and $\mathring \Map^{\op{II}}_p$ can be viewed as  the complement of $\cB_1$.
As $r$ varies, the boundary lengths of the complement of the $r$-neighborhood  of $\partial \Map^{\op{II}}_p$ form a branching process.
The law of this process under $\Q_p$ is studied in great detail in \cite{Curien-branch} via the so-called peeling process, where fine scaling limit results are established.
Here we only draw the following simple conclusions.
\begin{lemma}\label{lem:byd-II}
  With $\Q_p$-probability $1-o_p(1)$, there is a unique $2$-connected component $\ol\Map^{\op{II}}_p$ of $\mathring\Map^{\op{II} }_p$ with  largest perimeter, and,
  moreover, $\ol\Map^{\op{II}}_p$  is not a subgraph of $M_e$ for any edge $e\in E(\Map_p)\setminus E(\mathring \Map_p)$.
  Let  $\ol\ell_p$ be the perimeter of $\ol \Map^{\op{II}}_p$. Then the $\Q_p$-law of both $|V(\ol \Map^{\op{II}}_p)|/|V(\Map^{\op{II}}_p)|$ and $\ol\ell_p/p$ converge to 1.
  The  same hold if $\Q_p$ is replaced by $\P$.
\end{lemma}
\begin{proof}
  The  statements  concerning $\Q_p$ are straightforward consequences of the peeling by layer process in \cite{Curien-branch} and we omit the details.
  By Lemma~\ref{lem:RN}, the same statements hold  for~$\P$.
\end{proof}
When  the  2-connected component with the largest perimeter is not unique, 
we extend the definition of $\ol\Map^{\op{II}}_p$ in Lemma~\ref{lem:byd-II}  to a
$2$-connected component $\ol\Map^{\op{II}}_p$ chosen in a way measurable with respect to the map $\cB_1$.
By Lemma~\ref{lem:Markov}, conditioning on $\ol\ell_p$, the conditional $\Q_p$-law of $\ol\Map^{\op{II}}_p$ is $\Bol_{\op{II}}(\ol\ell_p)$. 
Let $S_p$ be the simple core of $\ol\Map^{\op{II}}_p$. 
Then conditioning on $\ol\ell_p$, the $\Q_p$-law of $S_p$ is $\Bol_{\op{III}}(\ol\ell_p)$.
Proposition~\ref{prop:random-bdy} is an immediate consequence of the following.
\begin{lemma}\label{lem:random-bdy}
  Let $\cS_p$ be the element in $\BMS^{\GHPU}$ obtained by rescaling $S_p$ as in the $i=\op{III}$ case of Theorem~\ref{thm:rigor}. Then the $\P$-law of $\cS_p$ converges to the  pointed free Brownian disk $\BD^*_{1}$ in the GHP topology. 
\end{lemma}
\begin{proof}[Proof of Proposition~\ref{prop:random-bdy} given Lemma~\ref{lem:random-bdy}]
  By Lemma~\ref{lem:byd-II}, $|V(\ol \Map^{\op{II}}_p)|/|V(\Map^{\op{II}}_p)|$ converges to 1 in $\P$-probability.
  By Proposition~\ref{prop:ratio} and Lemma~\ref{lem:RN}, $|V(\Map_p)|/ |V(\Map^{\op{II}}_p)|$ converges to $1/2$ in $\P$-probability. 
  By the same  argument as in Proposition~\ref{prop:ratio} using the law of large numbers, 
  $|V(S_p)|/ |V(\ol\Map^{\op{II}}_p)|$ converges to $1/2$ in $\P$-probability.  Combining these three results and Euler's formula we  get
  \begin{equation}\label{eq:volume}
  |F(S_p)|/|F(\Map_p)|\to 1 \qquad \textrm{in $\P$-probability}.
  \end{equation}
  Lemma~\ref{lem:random-bdy} implies that the $\Q_p$-law of  $\cS_p$ converges to the free Brownian disk $\BD_{1}$ in the GHP topology.  Now we set $\{\ell_p\}_{p\ge 3}$ in Proposition~\ref{prop:random-bdy} to be $\{\ol\ell_p\}_{p\ge3}$. Then the law of $\Map_{\ell_p}$ in Proposition~\ref{prop:random-bdy} is the same as the $\Q_p$-law of $S_p$ weighted by its number of inner faces. Therefore, $\{\ell_p\}_{p\ge 3}$ satisfy the property required in Proposition~\ref{prop:random-bdy}.
\end{proof} 
It remains to prove Lemma~\ref{lem:random-bdy}. By Lemma~\ref{lem:byd-II}, with $\P$-probability $1-o_p(1)$,
$\ol\Map^{\op{II}}_p$ comes from a 2-connected component $\ol \Map_p$ of $\mathring{\Map}_p$, together with  $\{M_e\}_{e\in E(\ol \Map_p)}$. 
In particular, $\ol \Map_p$  and $\ol\Map^{\op{II}}_p$ have the same boundary length, which is $\ol\ell_p$.
On the low probability event that this is not the case, we set $\ol \Map_p=\emptyset$. 
When $\ol \Map_p\neq \emptyset$, since $S_p$ is the simple core of $\ol\Map^{\op{II}}_p$,
as planar maps  $\ol\Map_p$ and $S_p$ are isomorphic. However, as subgraphs of $\ol\Map^{\op{II}}_p$, they are not necessarily identical because for two edges with the same endpoints, it is possible  that one  is on $\ol \Map_p$ and the other is on $S_p$.

\begin{proof}[Proof of Lemma~\ref{lem:random-bdy}]
  Unless otherwise specified, we will work under the $\P$-probability.
  Let  $(H^{\op S}_{(p)},\lambda^{\op S}_{(p)})$ be defined as
  $(H_{(p)},\lambda_{(p)})$  in Proposition~\ref{prop:contour-label} with $S_p$ in place of $\Map_p$.
  Since $\ol\ell_p/p\to 1$ by Lemma~\ref{lem:byd-II}, and  $|F(S_p)|/|F(\Map_p)|\to 1$ in probability (see~\eqref{eq:volume}), 
  $(H^{\op S}_{(p)},\lambda^{\op S}_{(p)})$  converge in law to  $(\CT(t),  \LP(t)_{t\in [0,\cA]})$ in the uniform topology.
  Let $D^{\op S}_{p} (i,j )$ be the $S_p$-graph distance of the $i$-th and $j$-th vertex in the lexicographical order. Set 
  \[
  D^{\op S}_{(p)} (s,t)=\Big(\frac{3}{2p}\Big)^{1/2} D^{\op S}_p(p^2s/3, p^2t/3).
  \] Then by Proposition~\ref{prop:dist2points}, $\{D^{\op S}_{(p)}\}_{p\ge 3}$ is tight. 
  We choose a subsequence of $\{p\}_{p\ge 3}$ where  $(H^{\op S}_{(p)}$, $\lambda^{\op S}_{(p)})$ and $D^{\op S}_{(p)} (s)$  jointly converge in law, 
  and couple $\{S_p\}_{p\ge 3}$ by the Skorokhod embedding theorem such that  the convergence holds almost surely. Let $D'$ be the limit of $D^{\op S}_{(p)}$.
  Recall  the set $\cD$ of   pseudo-metrics  on $[0,\cA]$ in the definition of the Brownian disk in Section~\ref{subsec:BD} and the maximal element
  $D^*$ defined in terms of $(\CT(t),  \LP(t)_{t\in [0,\cA]})$. Again by Proposition~\ref{prop:dist2points}  we have that almost surely  
  $D'\in\cD$ so that
  \begin{equation}\label{eq:upper}
  D'(s,t)\le D^*(s,t)\qquad \textrm{for all }s,t\in [0,\cA].
  \end{equation}

  Let $u^{\op{II} }_p$ be a vertex uniformly sampled in $V(\Map^{\op{II}}_p)$.  
  Then with probability $1-o_p(1)$, $u^{\op{II}}_p\in V(\ol\Map^{\op{II}}_p)$, and there exists $\ol e_p\in E(\ol\Map_p)$ 
  such that $u^{\op{II}}_p\in V(M_{\ol e_p})$.
  Similarly, there exists $e^{\op S}_p\in E(S_p)$ such that $u^{\op{II} }_p$ is a vertex on the 2-gon attached to
  $e^{\op S}_p$ in the simple core decomposition. Now
  we  uniformly sample an inner face $\wh f$ of $S_p$  conditioning on everything else.
  Let $\wh v$ be a vertex   incident to $\wh f$ chosen arbitrarily.
  Then  $\wh v\in V(\ol\Map^{\op{II}}_p)$ with probability $1-o_p(1)$. 
  Let $\wh e^{\op S}_p\in  E(S_p)$ be defined as $e^{\op S}_p$ with $\wh v$ in place of $u^{\op{II} }_p$. 
  
  Given a map $M$, for $x\in E(M)\cup F(M)$, we let $V(x)$ be the set of vertices incident to $x$.  For $x,y\in E(M)\cup F(M)$, let 
  $d_M(x,y):=\min\{ d_G(v,u), v\in V(x), u\in V(y) \}$. By the definition of $e^{\op S}_{p}$ and $\wh e^{\op S}_{p}$, 
  along the chosen subsequence of   $\{p\ge 3\}$ we have 
  \begin{equation}\label{eq:Dbis}
  \lim_{p\to\infty} \Big(\frac{3}{2p}\Big)^{1/2}d_{S_p} (e^{\op S}_{p},\wh e^{\op S}_{p})=D'(U\cA,V\cA) \qquad\textrm{in law},
  \end{equation}
  where $U$ and $V$ are independent uniform random variables on $(0,1)$ independent
of everything else.
  Here we use the fact that the distinction between uniform sampling on edges, vertices, or faces on a Boltzmann triangulation is negligible in the scaling limit. 
  
  Let $e^{\op S}_*$ be defined as $\wh e^{\op S}_p$  with the root face $f_*$ in place of $\wh f$. 
  Since $(e^{\op S}_p,e^{\op S}_*)$ has the same distribution as $(e^{\op S}_{p},\wh e^{\op S}_{p})$, \eqref{eq:Dbis} remains true for $d_{S_p} (e^{\op S}_p,e^{\op S}_*)$
  instead of $d_{S_p} (e^{\op S}_{p},\wh e^{\op S}_{p})$. Using the Skorokhod embedding theorem again, we assume that along the chosen subsequence 
  \begin{equation}\label{eq:D'}
  \lim_{p\to\infty} \Big(\frac{3}{2p}\Big)^{1/2}d_{S_p} (e^{\op S}_p,e^{\op S}_*)=D'(U\cA,V\cA)\qquad \textrm{almost surely}.
  \end{equation}

  Recall $\wt d_{\Map_p}$
  in Proposition~\ref{prop:lower}.
  With probability $1-o_p(1)$, $f_*\in F(\ol \Map_p)$.  
  On this event, $d_{\ol\Map_p}(\ol e_p, f_*)\ge \wt d_{\Map_p}(\ol e_p, f_*)$.
  By Lemma~\ref{prop:ghpu-close}
  the pairwise differences between 
  \[
  \textrm{$p^{-1/2}d_{\ol\Map_p}(\ol e_p, f_*)$, 
    $p^{-1/2}d_{\ol\Map^{\op{II}}_p}(\ol e_p, f_*)$, and $p^{-1/2}d_{S_p} (e^{\op S}_p,e^{\op S}_*)$
  }
  \]
  all converge to 0 in probability.
  Therefore~\eqref{eq:D'} implies that 
  \[
  D'(U\cA,V\cA) \ge \liminf_{p\to\infty}\Big(\frac{3}{2p}\Big)^{1/2}\wt d_{\Map_p}(\ol e_p, f_*) \qquad\textrm{almost surely}.
  \]
  Thanks to Proposition~\ref{prop:lower} and Corollary~\ref{cor:distvstar}, we know that
  \[
  \lim_{p\to\infty}\Big(\frac{3}{2p}\Big)^{1/2}\wt d_{\Map_p}(\ol e_p, f_*)=\LP(U\cA)-\min\{\LP(s),s\in[0,\mathcal{A}]\} \text{ in distribution.}
  \]
  Therefore, $D'(U\cA,V\cA)$ stochastically dominates $\LP(U\cA)-\min\{\LP(s),s\in[0,\mathcal{A}]\}$. 
  By the re-rooting invariance of the Brownian disk established in \cite[Corollary 21]{BeMi15},
  \[
  \LP(U\cA)-\min\{\LP(s),s\in[0,\mathcal{A}]\} \eqdist D^*(U\cA,V\cA).
  \]
  On the other hand, by~\eqref{eq:upper}, $D^*(U\cA,V\cA)\ge D'(U\cA,V\cA)$ almost surely.
  Therefore  $D^*(U\cA,V\cA)=D'(U\cA,V\cA)$ almost surely.
  By the continuity of $D^*$ and $D'$, almost surely $D^*(s,t)=D'(s,t)$ for all $s,t\in[0,\cA]$.
  Therefore $H^{\op S}_{(p)}$, $\lambda^{\op S}_{(p)}$, and $D^{\op S}_{(p)} (s)$  jointly converge to $\CT$,  $\LP$, and $D^*$.
  Now the GHP convergence of $\cS_{(p)}$ follows from the routine argument (see e.g.\  \cite[Section~5.2]{BeMi15}). 
\end{proof}

\subsection{Comparison between type III triangulations with close perimeters}\label{subsub:couple}
In this subsection we use $(\ell_p)_{p\ge 3 }$ to denote a general  random sequence such that $\ell_p/p$ tends to 1. For each $p$, let $M'_p$ be coupled with $\ell_p$ such that, conditioned on $\ell_p$, the law of $M'_p$  is $\Bolm(\ell_p)$. Let $M_p$ be sampled from $\Bolm(p)$ as in Theorem~\ref{thm:Bol-conv-III}. 
The goal of this subsection is to prove the following.
\begin{proposition}\label{prop:compare}
  For each $\delta\in (0,1)$,
  there exists a coupling of $(\cM_p, \cM'_p)_{p\ge 3}$ satisfying the following.
  For large enough $p$, the probability of the event that $\cM_p$ and $\cM'_p$ 
  have GHP-distance at least $\delta$  is smaller than $\delta$.
\end{proposition} 
\begin{proof}
  Recall the sampling method of $M_p$ given in Proposition~\ref{prop:Bol}. The input is a random  $p$-bridge $(b_i)_{0\le i\le 2p-3}$ and the i.i.d.\ sequence of blossoming trees $\{T_i\}_{i\ge 1}$. 
  Here by a random $p$-bridge we mean a simple random walk conditioned on visiting $-3$ at time $2p-3$.
  We first use $(b_i)_{1\le i\le 2p-3}$  to determine the stems attached to boundary vertices to create $2p-3$ corners. Then we attach $(T_i)_{i\ge 1}$ to these corners and perform the closure operation $\chi$.
  
  Given a sample of  $\ell_p$, let $ (b'_i)_{1\le i\le 2\ell_p-3}$ be such that  $(b')_{1\le i\le 2\ell_p-3}$ is a random $\ell_p$-bridge.   We may use $(b'_i)_{1\le i\le 2\ell_p-3}$ and the same sequence $(T_i)_{i\ge 1}$ above to sample $M'_p$.
  
  Using a local limit theorem (see  e.g.\ \cite[Theorem~2.3.11]{Lawler-modern}), for a small but fixed $\eps\in(0,1)$, 
  we may further couple $(b_i)_{1\le i\le 2p-3}$ and $(b'_i)_{1\le i\le 2\ell_p-3}$ such that with probability $1-o_p(1)$ they agree on $[0, 2(1-\eps)p]\cap \Z$. We omit this simple random walk exercise. (See e.g.\  \cite[Proposition 5.19]{LSW-Schnyder}  for a similar but more complicated statement of this type in the setting of random walk in cones.)
  Using Propositions~\ref{prop:contour-label} and~\ref{prop:dist2points},
  by making $\eps$ as small as we want depending on $\delta$, we conclude the proof. 
\end{proof}
\subsection{GHPU convergence}\label{subsec:GHPU-conv}
It is clear from Propositions~\ref{prop:random-bdy} and~\ref{prop:compare}  that $\cM_p$ converges in the  GHP topology. 
In fact, we see that $(H_{(p)},\lambda_{(p)},\cM_p)$ jointly converge to their continuum counterparts.
Recall  that  $\cM_p$ can be written as $\left( \Map_p , d_p , \mu_p , \bdy_p \right)$ as in Section~\ref{subsec:GHPU}.
Using Proposition~\ref{prop:dist2points} and the exact same argument as in the proof of \cite[Theorem~4.1]{gwynne-miller-uihpq} (based on Lemma~2.14 here), 
we can upgrade the GHP convergence of $\cM_p$ to the GHPU convergence. 
This concludes the type III case of Theorem~\ref{thm:rigor}. The GHP convergence in the type I and II cases follows from Propositions~\ref{prop:ghpu-close} and~\ref{prop:ghpu-close2}. Since a type I triangulation of a polygon, its 2-connected core, and its simple core share the same set of boundary vertices, the uniform convergence of the boundary curve in the type I and II cases of Theorem~\ref{thm:rigor} follows from the couplings in Proposition~\ref{prop:ratio} and~\ref{prop:ratio2}.



\providecommand{\bysame}{\leavevmode\hbox to3em{\hrulefill}\thinspace}
\providecommand{\MR}{\relax\ifhmode\unskip\space\fi MR }
\providecommand{\MRhref}[2]{%
  \href{http://www.ams.org/mathscinet-getitem?mr=#1}{#2}
}
\providecommand{\href}[2]{#2}

\end{document}